\documentclass[11pt]{article}
\usepackage{fullpage, amssymb}
\usepackage{caption}
\usepackage{subcaption}
\usepackage{graphicx}
\usepackage{amsmath}
\usepackage{float}
\usepackage{listings}
\usepackage{xcolor}
\usepackage{amsthm}
\usepackage{algorithmic}
\usepackage{algorithm}
\usepackage[colorlinks,
linkcolor=blue,
anchorcolor=blue,
citecolor=blue]{hyperref}
\usepackage{natbib}
\usepackage{comment}
\usepackage{bbm}
\usepackage{multirow}
\usepackage{array}
\usepackage[shortlabels]{enumitem}
\usepackage{romannum}
\AtBeginDocument{\pagenumbering{arabic}}
\usepackage{multirow}
\usepackage{stmaryrd}

\setcitestyle{authoryear,round}

\newtheorem{theorem}{Theorem}
\newtheorem{lemma}{Lemma}

\newtheorem{proposition}{Proposition}
\newtheorem{corollary}{Corollary}
\newtheorem{assumption}{Assumption}
\newtheorem{remark}{Remark}
\newtheorem{claim}{Claim}
\newtheorem{condition}{Condition}

\newtheorem{innercustomthm}{Theorem}
\newenvironment{customthm}[1]
{\renewcommand\theinnercustomthm{#1}\innercustomthm}
{\endinnercustomthm}

\def\rank{\text{rank}}
\def\wt{\widetilde}

\newcommand\fro[1]{\| #1 \|_{\rm{F}}}
\newcommand\fror[1]{\| #1 \|_{\rm{F},r}}

\newcommand\op[1]{\| #1 \|}

\newcommand\lone[1]{\| #1 \|_{1}}
\newcommand\ltwo[1]{\| #1 \|_{2}}

\newcommand{\inp}[2]{\langle #1,#2\rangle}

\def\calH{{\mathcal H}}
\def\calI{{\mathcal I}}

\def\calM{{\mathcal M}}
\def\calN{{\mathcal N}}
\def\calO{{\mathcal O}}
\def\calP{{\mathcal P}}

\def\calS{{\mathcal S}}

\def\bSigma{{\boldsymbol{\Sigma}}}

\def\bcalE{{\boldsymbol{\mathcal E}}}

\def\BB{{\mathbb B}}

\def\EE{{\mathbb E}}

\def\MM{{\mathbb M}}

\def\PP{{\mathbb P}}

\def\RR{{\mathbb R}}
\def\SS{{\mathbb S}}
\def\TT{{\mathbb T}}

\def\VV{{\mathbb V}}

\def\b{{\mathbf b}}

\def\s{{\mathbf s}}

\def\u{{\mathbf u}}
\def\v{{\mathbf v}}

\def\x{{\mathbf x}}

\def\A{{\mathbf A}}
\def\B{{\mathbf B}}
\def\C{{\mathbf C}}
\def\D{{\mathbf D}}

\def\G{{\mathbf G}}

\def\I{{\mathbf I}}

\def\L{{\mathbf L}}
\def\M{{\mathbf M}}

\def\Q{{\mathbf Q}}
\def\R{{\mathbf R}}

\def\T{{\mathbf T}}
\def\U{{\mathbf U}}
\def\V{{\mathbf V}}

\def\X{{\mathbf X}}
\def\Y{{\mathbf Y}}
\def\Z{{\mathbf Z}}
\def\frakP{{\mathfrak P}}

\def\lmax{{l_{\textsf{max}}}}
\def\tauc{\tau_{\textsf{\tiny comp}}}
\def\taus{\tau_{\textsf{\tiny stat}}}
\def\muc{\mu_{\textsf{\tiny comp}}}
\def\mus{\mu_{\textsf{\tiny stat}}}
\def\Lc{L_{\textsf{\tiny comp}}}
\def\Ls{L_{\textsf{\tiny stat}}}
\def\eps{\varepsilon}

\def\bSigma{{\boldsymbol{\Sigma}}}
\def\bLambda{{\boldsymbol{\Lambda}}}
\def\Bbeta{{\boldsymbol{\beta}}}
\def\bxi{{\boldsymbol{\xi}}}
\def\ku{C_{{\tiny u}}}
\def\kl{C_{{\tiny l}}}

\begin{document}
	\pagenumbering{arabic}
	
	\title{Computationally Efficient and Statistically Optimal Robust High-Dimensional Linear Regression}
	
	\author{Yinan Shen, Jingyang Li, Jian-Feng Cai\footnote{Jian-Feng Cai’s research was partially supported by Hong Kong RGC Grant GRF 16310620 and GRF 16309219.}  and Dong Xia\footnote{Dong Xia's research was partially supported by Hong Kong RGC Grant ECS 26302019, GRF 16303320 and GRF 16300121.}\\
		{\small Hong Kong University of Science and Technology}}

	\date{(\today)}
	
	\maketitle
	\begin{abstract}
		High-dimensional linear regression under heavy-tailed noise or outlier corruption is challenging,  both computationally and statistically.  Convex approaches have been proven statistically optimal but suffer from high computational costs,  especially since the robust loss functions are usually non-smooth.   More recently,  computationally fast non-convex approaches via sub-gradient descent are proposed,  which,  unfortunately,  fail to deliver a statistically consistent estimator even under sub-Gaussian noise.  In this paper,  we introduce a projected sub-gradient descent algorithm for both the sparse linear regression and low-rank linear regression problems. The algorithm is not only computationally efficient with linear convergence but also statistically optimal, be the noise Gaussian or heavy-tailed with a finite $1+\eps$ moment.  The convergence theory is established for a general framework and its specific applications to absolute loss,  Huber loss and quantile loss are investigated.  Compared with existing non-convex methods,  ours reveals a surprising phenomenon of {\it two-phase convergence}.  In phase one,  the algorithm behaves  as in typical non-smooth optimization that requires gradually decaying stepsizes.  However,  phase one only delivers a statistically sub-optimal estimator, which is already observed in the existing literature.  Interestingly,   during phase two,  the algorithm converges linearly {\it as if} minimizing a smooth and strongly convex objective function, and thus a constant stepsize suffices.  Underlying the phase-two convergence is the {\it smoothing effect} of random noise to the non-smooth robust losses in an area {\it close but not too close} to the truth.  Numerical simulations confirm our theoretical discovery and showcase the superiority of our algorithm over prior methods. 
	\end{abstract}

\section{Introduction}\label{sec:intro}
Let $\{(\X_i, Y_i)\}_{i=1}^n$ be a collection of independent observations satisfying 
\begin{equation}\label{eq:tr_model}
	Y_i=\langle \X_i, \T^{\ast}\rangle+\xi_i
\end{equation}
with $\X_i$ and $Y_i$ being the {\it covariate} and {\it response}, respectively. Here, $\langle \cdot, \cdot\rangle$ denotes the inner product in Euclidean space. The latent noise $\xi_i$ is usually assumed {\it random} for ease of analysis. We assume that $\T^{\ast}$ resides in a high-dimensional ambient space but has a low-dimensional structure. We focus on two high-dimensional linear regression problems: the {\it sparse linear regression} where $\T^{\ast}\in \RR^{d}$ is sparse and {\it low-rank linear regression} where $\T^{\ast}\in \RR^{d\times d}$ is low-rank. For both problems, it is of great interest to effectively,  both computationally and statistically, recover $\T^{\ast}$ using as few observations as possible. There is a vast literature discussing the motivations and applications, e.g., \cite{tibshirani1996regression,needell2009cosamp,blumensath2009iterative,blumensath2010normalized,loh2011high,
	shen2017tight,zhu2020polynomial,zhao2022high, gross2010quantum,xia2016estimation, chen2011integrating, negahban2011estimation, chiu2021low, siddiqi2010reduced}.


The methods of high-dimensional linear regression can be broadly classified into two categories: the {\it convex} and {\it non-convex} approaches. 
The typical convex approaches introduce an additional penalization such as the matrix nuclear norm to promote low-rank solutions and the $\ell_1$-norm to promote sparsity. 
The convexity renders well-understood algorithms for convex programming immediately applicable. Therefore statisticians  can just focus on studying the statistical performances without paying too much attention to the computational implementations.  
For instance, by assuming the i.i.d.  (sub-){\it Gaussian} noise with a variance $\sigma^2$ and the so-called {\it restricted isometry property} (RIP) or {\it restricted strong convexity} (RSC)  in low-rank linear regression for estimating a  $d\times d$ matrix,  \cite{candes2011tight, rohde2011estimation,negahban2011estimation,cai2015rop,chen2015fast,davenport2016overview,chandrasekaran2012convex} proved that the nuclear-norm penalized least square estimator attains the error rate $\tilde{O}_p(\sigma^2rdn^{-1})$\footnote{Here $\tilde{O}_p(\cdot)$ stands for the typical big-O notation up to logarithmic factors and holds with high probability.} in the squared Frobenius norm where $r$ is the rank.  This rate is shown to be optimal in the minimax sense. See,  e.g.,  \cite{xia2014optimal,ma2015volume} and references therein.  
Despite the appealing theoretical performances,  convex approaches have two major drawbacks.  First,  convex methods operate directly in the high-dimensional ambient space making them run slowly and are unscalable to ultra-high dimensional problems. Secondly,  while the established error rate is theoretically minimax optimal,  the implicit constant factor seems large and the resultant error rate is,  {\it in practice},  often inferior to that by non-convex approaches. These issues can be nicely addressed by non-convex approaches. In low-rank linear regression, the non-convex methods could operate on the basis of matrix factorization and resort to the gradient descent type algorithms \citep{chen2015fast, burer2003nonlinear,zheng2016convergent, zhao2015nonconvex,xia2021statistical, wei2016guarantees}; in sparse linear regression, the typical non-convex approaches include the least angel regression \citep{efron2004least}, iterative hard thresholding \citep{blumensath2009iterative}, and best subset selection \citep{zhu2020polynomial}, etc. Under similar conditions,  it has been demonstrated that these non-convex algorithms converge fast (more exactly,  linearly) and also deliver minimax optimal estimators under (sub-)Gaussian noise. We remark that non-convex methods often require more tuning parameters. For instance, the matrix factorization algorithms often require knowing the true rank \citep{chandrasekaran2012convex,chen2015fast} and the iterative hard thresholding algorithms need the information of the underlying sparsity \citep{blumensath2009iterative,efron2004least}. 


The recent boom of data technology poses new challenges to high-dimensional linear regression,  among which the heavy-tailed noise and outliers appear routinely in numerous applications such as diffusion-weighted imaging \citep{chang2005restore},  on-line advertising \citep{sun2017provable},  and gene-expression data analysis \citep{sun2020adaptive}.  The aforementioned approaches minimizing a square loss become vulnerable or even completely useless when the noise has a heavy tail or when data are partially corrupted.  To mitigate the influence of heavy-tailed noise and corruptions,   a natural solution is to replace the square loss by more robust ones.  Notable examples include the absolute loss \citep{candes2011robust, cambier2016robust},  the renowned Huber loss \citep{huber1965robust}, and quantile loss \citep{koenker2001quantile},  all of which are convex but non-smooth. \cite{elsener2018robust} proposed a convex approach based on nuclear-norm penalized absolute or Huber loss for low-rank linear regression,  and proved that the estimator attains the error rate $\tilde{O}_p(rd_1n^{-1})$ in squared Frobenius norm as long as the noise has a non-zero density in a neighbourhood of origin.  Interestingly, \cite{elsener2018robust} doesn't need to estimate the noise level for choosing the tuning parameters.  A more general framework requiring only Lipschitz and convex loss but imposing a so-called {\it Bernstein condition} on the noise is investigated by \cite{alquier2019estimation}.  See also \cite{klopp2017robust}.  The foregoing approaches still rely on convex programming  and suffer from the issue of computational inefficiency explained above.  In fact, the computational issue is severer here with non-smooth objective functions \citep{boyd2004convex}.


The sub-gradient descent algorithms based on  matrix factorization were studied recently by  \cite{charisopoulos2021low,li2020nonconvex,tong2021low} in low-rank linear regression. These algorithms converge linearly if equipped with a good initialization and with properly chosen stepsizes,  but the statistical performances of these non-convex approaches under heavy-tailed noise are either largely missing or, generally,  sub-optimal.  To be more specific,    \cite{li2020nonconvex} only proved the exact recovery under sparse outliers {\it without noise}; the error rates in squared Frobenius norm established in \cite{charisopoulos2021low} and \cite{tong2021low} turn into $\tilde{O}_p(\sigma^2)$ even when the noise is Gaussian with variance $\sigma^2$, which is not statistically consistent. 
There is also a vast literature of sparse linear regression under heavy-tailed noise or partial corruption, e.g., \cite{sun2020adaptive,zheng2015globally,belloni2011square,alquier2019estimation,lecue2017sparse,
	lugosi2019regularization,pan2021iteratively,wang2021new,wang2020tuning,thompson2020outlier}. The adpative Huber loss was investigated by \cite{sun2020adaptive} under a finite $1+\eps$ moment condition of noise. \cite{depersin2020robust} studied the median-of-means approach and derived sharp error rates by VC-dimension, achieving the first sub-Gaussian sparse estimators. \cite{belloni2011square} and \cite{wang2020tuning} proposed the square-root LASSO and the absolute loss, respectively. While being statistically robust to heavy tails, these prior works didn't provide provably fast converging algorithms. In \cite{foucart2017iht}, an iterative hard thresholding algorithm was proposed for sparse linear regression by minimizing the absolute loss, which converges fast but a sub-optimal error rate was derived. 

In this paper,  we propose a computationally efficient non-convex algorithm based on the projected sub-gradient descent for high-dimensional linear regression using robust loss functions. It is essentially the absolute loss hard thresholding algorithm ({\it IHT}-$\ell_1$,\cite{foucart2017iht}) for sparse linear regression and the Riemannian sub-gradient descent algorithm ({\it RsGrad}, \cite{vandereycken2013low,cambier2016robust,wei2016guarantees,cai2021generalized}) for low-rank linear regression. In both problems, the sub-gradient descent algorithm admits fast computation and linear convergence. Unlike \cite{foucart2017iht} and \cite{tong2021low}, we discover that the sub-gradient descent algorithm exhibits an intriguing phenomenon referred to as the {\it two-phase convergence}. In phase one,  the algorithm behaves like typical {\it non-smooth} optimization requiring gradually decaying stepsizes through iterations.  For example, in low-rank linear regression, if the noise is Gaussian with a variance $\sigma^2$ and the absolute loss is equipped,  the phase-one iterations converge linearly and reach an estimator whose squared Frobenius norm error is $O_p(\sigma^2)$.  This rate has been achieved by \cite{tong2021low,charisopoulos2021low} whereas theirs is based on a scaled sub-gradient (ScaledSM) descent algorithm.  In phase two,   the algorithm (RsGrad) behaves like smooth optimization in that a constant stepsize guarantees a linear convergence.   As a result,  after phase-two iterations,  RsGrad can output an estimator which is statistically optimal under mild conditions,  e.g.  it achieves the rate $O_p(\sigma^2rdn^{-1})$ for low-rank linear regression.  The two-phase convergence phenomenon also occurs for sparse linear regression. For both problems, minimax optimal rates can be achieved under heavy-tailed noise.

Our contributions are multi-fold.   First,  we propose a projected sub-gradient descent algorithm for robust high-dimensional linear regression that is applicable to a wide class of non-smooth loss functions for both the sparse linear regression and low-rank linear regression problems.  While the Riemannian sub-gradient descent (RsGrad) has been introduced by \cite{cambier2016robust} for minimizing the absolute loss,  to our best knowledge,  there exists no theoretical guarantees for its convergence.  Under mild conditions,  we prove that the algorithm converges linearly.  Secondly,  we demonstrate the statistical optimality of the final estimator delivered by our algorithm for both problems,  be the noise Gaussian or heavy-tailed.   For instance, the specific applications of RsGrad to low-rank linear regression with the absolute loss,  Huber loss and quantile loss confirm that a rate $O_p(rdn^{-1})$ is attainable as long as  the noise has a fine $1+\eps$ moment and its density satisfies mild regularity conditions.  Similar conditions have appeared in \cite{elsener2018robust}.  Unlike \cite{fan2021shrinkage} and \cite{minsker2018sub},  our rate is proportional to the noise size. Thirdly,  our analysis reveals a new phenomenon of {\it two-phase convergence} that enables us to achieve the statistically optimal error rates,  which is a significant improvement over existing works  \citep{foucart2017iht,  li2020nonconvex, tong2021low,charisopoulos2021low}.  While the phase-one convergence is typical for non-smooth optimization,  the phase-two convergence,  interestingly,  behaves like smooth optimization.  It seems that the random noise has an effect of smoothing in the phase-two convergence. 

The rest of the paper is organized as follows.  In Section~\ref{sec:loss},  we review three typical examples of non-smooth robust loss functions.  We provide a warm-up example of robust sparse linear regression in Section~\ref{sec:vector} and investigate the two-phase convergence of the iterative hard thresholding algorithm. A general framework of studying the two-phase convergence performance of the Riemannian sub-gradient descent algorithm for robust low-rank linear regression is presented in Section~\ref{sec:matrix}. Non-convex algorithms inevitably involve important tuning parameters, which will be discussed in Section~\ref{sec:discussion}. In Section~\ref{sec:huber}, we further investigate our methods under the Huber's contamination model. We showcase the results of numerical experiments and comparison with prior methods in Section~\ref{sec:simulation}. All the proofs are relegated to the Appendix.


\section{Robust Loss Functions}
\label{sec:loss}
Denote $\MM$ the feasible set of the optimization program. It can be the set of matrices with rank bounded by $r$ in low-rank regression (Section~\ref{sec:matrix}) and the set of vectors with support size bounded by $\tilde{s}$ in sparse regression (Section~\ref{sec:vector}).  The goal of linear regression is to solve 
\begin{align}\label{eq:loss}
	\hat\T:=\underset{\T\in\MM}{\arg\min}\ f(\T)  \  \textrm{ where } f(\T) := \sum_{i=1}^n\rho\big(\langle \T,  \X_i\rangle-Y_i\big).
\end{align}
Here $\rho(\cdot):\RR\mapsto\RR_+$ is a properly chosen loss function. 
The statistical property of $\hat\T$ crucially relies on the loss function.  For instance,  $\hat\T$ attained by using the $\ell_2$-loss ({\it square loss}),  i.e.,  $\rho(x):=x^2$,  has been proven effective \citep{chen2015fast,wei2016guarantees,xia2021statistical,jain2014iterative,agarwal2010fast} in dealing with sub-Gaussian noise but is well-recognized sensitive to outliers and heavy-tailed noise \citep{huber1965robust}.  There are the so-called {\it robust} loss functions which are relatively immune to outliers and heavy-tailed noise and capable to deliver a more reliable estimate $\hat\T$. The examples of robust loss functions include:
{\it
	\begin{enumerate}[1.]
		\item {\it Absolute loss ($\ell_1$-loss):}  $\rho(x):=|x|$ for any $x\in\RR$;
		\item {\it Huber loss:} $\rho_{H,\delta}(x):=x^2\mathbbm{1}(|x|\leq \delta)+(2\delta|x|-\delta^2)\mathbbm{1}(|x|>\delta)$ for any $x\in\RR$ where $\delta>0$ is a tuning parameter;
		\item {\it Quantile loss:} $\rho_{Q,\delta}(x):=\delta x\mathbbm{1}(x\geq 0)+(\delta-1)x\mathbbm{1}(x<0)$ for any $x\in \RR$ with $\delta:=\PP(\xi\leq 0)$.
	\end{enumerate}
}

The absolute loss,  Huber loss and quantile loss are all convex functions.  They have appeared in the literature of sparse regression including the LASSO-type convex approaches \citep{sun2020adaptive,fan2021shrinkage} and the non-convex approaches by iterative algorithms \citep{foucart2017iht}. These robust losses have also been investigated for low-rank regression, e.g.,  the convex approach based on nuclear norm penalization \citep{elsener2018robust,klopp2017robust,candes2011robust} and the non-convex approach based on gradient-descent-type algorithms \citep{li2020nonconvex,tong2021accelerating}.  Compared with the square loss, the aforementioned robust loss functions are non-smooth,  i.e.,  their derivatives are dis-continuous,  which brings new challenge to its computation. 
Indeed,  the optimizing program \eqref{eq:loss} usually exploits the sub-gradient of robust loss functions, which is written as $\partial f(\T)$.  See,  for instance,  \cite{charisopoulos2021low} and references therein. Recently, \cite{thompson2020outlier} showed that regularization using the sorted Huber loss could significantly improve the constant in the error rate than that using the standard Huber loss. However, since the computation of sorted Huber loss is challenging in our framework, we lave it for future work.

The existing statistics literature have proved that $\hat\T$ attained by using robust loss functions enjoys the statistical robustness to heavy-tailed noise. However, the optimal solution $\hat\T$ is often solved by heuristic algorithms without convergence guarantees \citep{sun2020adaptive} or by non-smooth optimization algorithms \citep{elsener2018robust} which may converge slowly in general and the computational cost can be considerably high, rendering the applicability of aforementioned robust loss functions questionable for real-world high-dimensional problems. More recently, \cite{charisopoulos2021low,tong2021low,foucart2017iht} investigated sub-gradient descent algorithms for minimizing the robust loss functions showing that these algorithms can converge fast if the stepsizes are properly scheduled along the iterations. {Unfortunately, their analysis frameworks all failed to demonstrate that the statistical benefits of robust loss functions claimed by prior statistics literature can be achieved by the sub-gradient descent algorithms.  There exists a gap between the statistical and computational efficiencies. 
	
	In the subsequent Sections~\ref{sec:vector} and \ref{sec:matrix}, we will show that a revised sub-gradient descent algorithm equipped with a novel stepsize schedule converges fast and achieves statistically optimal rates in various settings. Besides the algorithmic design, a novel analysis framework also plays a critical role in developing our  statistical theories. 
	
	\section{Warm-up: Sparse Linear Regression}
	\label{sec:vector}
	We begin with a warm-up example of high-dimensional sparse linear regression using the absolute loss. Let $\Bbeta^*\in\RR^{d}$ be sparse with a support size $|\text{supp}(\Bbeta^*)|= s\ll d$. We observe a collection of independent random pairs $\{(\X_i,Y_i)\}_{i=1}^n$ satisfying $Y_i=\langle \X_i, \Bbeta^*\rangle+\xi_{i}$  where the predictor vector $\X_i\in\RR^d$ is assumed randomly sampled. Our goal is to estimate $\Bbeta^{\ast}$ when the noise $\xi_i$ may have heavy tails. For ease of exposition, we only study the absolute loss \citep{foucart2017iht} but parallel results can be derived for other robust loss functions (see Section~\ref{sec:matrix}). Denote
	\begin{align}
		\hat{\Bbeta}:=\underset{|\text{supp}(\Bbeta)|\leq\tilde{s}}{\arg\min}f(\Bbeta),\quad \textrm{where } f(\Bbeta):=\lone{\X\Bbeta-\Y}:=\sum_{i=1}^n |Y_i-\langle \Bbeta, \X_i\rangle|.
		\label{eq:vec-absloss:sparse}
	\end{align}
	Here the chosen sparsity level $\tilde{s}\in[s, d]$, $\X=(\X_1,\dots,\X_n)^{\top}$ and $\Y=(Y_1,\dots,Y_n)^{\top}$. Let $\|\cdot\|_{p}$ denote the $\ell_p$-norm of vectors for $1\leq p\leq 2$. 
	
	\subsection{Algorithm: Iterative Hard Thresholding} Our method for optimizing the program (\ref{eq:vec-absloss:sparse}) is essentially the {\it projected sub-gradient descent} algorithm. The projected gradient descent algorithm has been explored in minimizing the square loss \citep{agarwal2010fast,chen2015fast}, which is easy to implement and fast computable. We resort to the sub-gradient because the absolute loss is non-smooth.  
	
	Since projecting a vector onto the set of sparse vectors can be realized by hard thresholding, our algorithm is often referred to as the {\it Iterative Hard Thresholding} \citep{blumensath2009iterative,blumensath2010normalized}. The pseudocodes are presented in Algorithm~\ref{alg:IHTl1}. The operator $\calH_{\tilde{s}}(\V)$ keeps $\V$'s $\tilde{s}$ entries with the largest absolute values and zeros out the others. In practice, we choose the sub-gradient $\G_l=\sum_{i=1}^n sign(\langle \Bbeta_l, \X_i\rangle-Y_i)\X_i$ and set $sign(0)=0$. Note that Algorithm~\ref{alg:IHTl1} allows arbitrary initializations.

	\begin{algorithm}
		\caption{Iterative Hard Thresholding ($\ell_1$-loss)}\label{alg:IHTl1}
		\begin{algorithmic}
			\STATE{\textbf{Input}: observations $\{(\X_i, Y_i)\}_{i=1}^n$,  max iterations $\lmax$,  step sizes $\{\eta_l\}_{l=0}^{\lmax}$ and sparsity parameter $\tilde{s}$.}
			\STATE{Initialization: {\it arbitrary} $\Bbeta_0$ with a support size $|\text{supp}(\Bbeta_0)|\leq\tilde{s}$}
			\FOR{$l = 0,\ldots,\lmax$}
			\STATE{Choose a sub-gradient:  $\G_l\in\partial f(\Bbeta_{l})$}
			\STATE{Truncation: $\Bbeta_{l+1} = \calH_{\tilde{s}}(\Bbeta_l - \eta_{l}\G_l)$}
			\ENDFOR
			\STATE{\textbf{Output}: $\hat\Bbeta=\Bbeta_{\lmax}$}
		\end{algorithmic}
	\end{algorithm}
	
	
	\subsection{Theory: two-phase convergence}
	The convergence performance of Algorithm~\ref{alg:IHTl1} and statistical property of $\hat\Bbeta$ rely on the following two assumptions. 
	
	\begin{assumption}	\label{assump:sensing operators:vec}
		The covariate vectors $\X_1,\dots,\X_n$ are independent Gaussian $\X_i\sim N(\boldsymbol{0},\bSigma_i)$ where $\bSigma_i$'s are symmetric and positive definite. There exist absolute constants $\kl,\ku>0$ such that 
		\begin{align*}
			\kl\leq\lambda_{\min}(\bSigma_i)\leq\lambda_{\max}(\bSigma_i)\leq\ku.
		\end{align*}
	\end{assumption}
	
	\begin{assumption}(Heavy-tailed noise \Romannum{1}) \label{assump:heavy-tailed}
		The noise $\xi_1,\cdots,\xi_n$ are i.i.d. with the density $h_{\xi}(\cdot)$ and distribution function $H_{\xi}(\cdot)$,  respectively. 
		There exists an $\varepsilon>0$ such that $\EE |\xi|^{1+\varepsilon}<+\infty$.  The noise has median zero,  i.e.,  $H_{\xi}(0)=1/2$.  Denote $\gamma=\EE|\xi|$.  There exist constants $b_0, b_1>0$ (may be dependent on $\gamma$) such that \footnote{The lower bound condition can be slightly relaxed to $|H_{\xi}(x)-H_{\xi}(0)|\geq |x|/b_0$ when $|x|\leq 8(\ku/\kl)^{1/2}\gamma$.}
		\begin{align*}
			h_{\xi}(x)\geq b_0^{-1}, &\ \ \  \textrm{ for all } |x|\leq 8(\ku/\kl)^{1/2}\gamma;\\
			h_{\xi}(x)\leq b_1^{-1}, &\ \ \ \forall x\in \RR.
		\end{align*}
	\end{assumption}
	
	By Assumption~\ref{assump:heavy-tailed},  a simple fact is $b_0\geq 8(\ku/\kl)^{1/2}\gamma$ and $b_0\geq b_1$.  We require the noise has a finite $1+\eps$ moment,  which is essential to guarantee $\lim_{x\to+\infty}x(1-H_{\xi}(x))=0$.  This moment condition is fairly weak compared with the existing literature.  For instance,  \cite{minsker2018sub,depersin2020robust} require a finite second-order moment condition;  \cite{fan2021shrinkage} imposes a $2+\eps$ moment condition.  The lower bound on the density function also appeared in \cite{elsener2018robust} and is a special case of the (local) Bernstein condition \citep{alquier2019estimation,chinot2020robust}.  The upper bound condition on the density function is mild, e.g.,  a Lipschitz distribution function ensures such a uniform upper bound.  
	
	\begin{theorem}
		Suppose Assumptions~\ref{assump:sensing operators:vec} and \ref{assump:heavy-tailed} hold. There exist absolute constants $C_1,\dots,C_6>0$, $c_2^*\in(0,1)$ such that if $n\geq C_1\ku\kl^{-1}\tilde{s}\log(2d/\tilde{s})$, $\tilde{s}\geq C_2 (\ku/\kl)^8(b_0/b_1)^8\cdot s$ and the initial stepsize $\eta_{0}\in(n\ku)^{-1}\kl^{1/2}\ltwo{\Bbeta_0-\Bbeta^*}\cdot\left[1/8,\ 3/8\right]$, then with probability over $1-\exp(-C_3\tilde{s}\log (2d/\tilde{s}))-3\exp(-(n\log(2d/\tilde{s}))^{1/2}\log^{-1} n)$, the sequence $\{\Bbeta_l\}_{l\leq \lmax}$ of Algorithm~\ref{alg:IHTl1} has the following dynamics: 
		\begin{enumerate}[(1)]
			\item in phase one, i.e., when $\ltwo{\Bbeta_l-\Bbeta^*}\geq 8\kl^{-1/2}\gamma$, by taking the stepsize $\eta_{l}=(1-c_1)^{l}\eta_0$ with any fixed constant $c_1\leq\kl\ku^{-1}/64$,  then $\ltwo{\Bbeta_{l+1}-\Bbeta^*}\leq(1-c_1)^{l+1}\ltwo{\Bbeta_0-\Bbeta^*}$;
			\item in phase two, i.e., when $ C_4b_0\kl^{-1}n^{-1/2}\big(\ku\tilde{s}\log(2d/\tilde{s})\big)^{1/2}\leq\ltwo{\Bbeta_l-\Bbeta^*}\leq8\kl^{-1/2}\gamma$, by taking a constant stepsize $\eta_{l}=\eta\in \kl^{1/2}b_1^2(nb_0\ku)^{-1}\cdot\left[C_5, C_6\right]$,  then $\ltwo{\Bbeta_{l+1}-\Bbeta^*}\leq(1-c_2^*)\ltwo{\Bbeta_l-\Bbeta^*}$
		\end{enumerate}
		Therefore, after at most $O\big(\log(\ltwo{\Bbeta_0-\Bbeta^*}/\gamma)+\log(n\gamma b_0^{-1}\log(2d/\tilde{s}))\big)$ iterations, Algorithm~\ref{alg:IHTl1} outputs an estimator with the error rate 
		$$
		\|\hat\Bbeta-\Bbeta^*\|_2^2\leq \frac{C_4\ku}{\kl^2}  \cdot \frac{b_0^2\tilde{s}\log(2d/\tilde{s})}{n}.
		$$ 
		\label{thm:vec:sparse}
	\end{theorem}
	
	Note that in Theorem~\ref{thm:vec:sparse} if a constant depends on $b_0$ or/and $b_1$, a star sign is placed on its top right corner. 
	The phase one convergence is typical in non-smooth optimization which requires decaying stepsizes; see also \cite{foucart2017iht}.  Surprisingly, the phase two convergence requires a constant stepsize as if minimizing a smooth objective function \citep{blumensath2009iterative,chen2015fast}. 

	\paragraph*{Gaussian noise} If the noise is Gaussian $N(0,\sigma^2)$, we have $b_0\asymp b_1\asymp \sigma$. If further $\ku\asymp\kl\asymp 1$, we can choose a sparsity level $\tilde{s}\gtrsim s$ so that the output of Algorithm~\ref{alg:IHTl1} attains the rate  $\|\hat{\Bbeta}-\Bbeta^*\|_2^2=O_p\big(\sigma^2\cdot sn^{-1}\log(d/s)\big)$, which is minimax optimal \citep{rigollet2011exponential,ye2010rate,raskutti2011minimax,lounici2011oracle}.

	\paragraph*{Heavy-tailed noise} Under a finite second moment condition, \cite{sun2020adaptive} proved that a Huber-loss-based estimator can achieve the error rate $O\big(s/n\cdot \log d\big)$ with probability exceeding $1-d^{-\Omega(1)}$. See also related works in \cite{pan2021iteratively,wang2021new,wang2020tuning}. A median-of-means estimator was studied in \cite{depersin2020robust} showing that, under a finite second moment condition, the rate $O\big(s/n\cdot \log (d/s)\big)$ can be attained with probability exceeding $1-\Omega(1)\cdot\exp(-s\log(ed/s))$. Our Theorem~\ref{thm:vec:sparse} only requires a finite $1+\eps$ moment condition, but the attained error rate is comparable to \cite{depersin2020robust}. Though our estimator is not sub-Gaussian, most existing sub-Gaussian estimators require a finite second moment condition \citep{catoni2012challenging,lugosi2019risk,lugosi2019mean,hopkins2020mean,depersin2022robust}. \cite{chinot2020robust} relaxed the noise moment condition but it is unclear how their estimator can be efficiently computed with theoretical guarantees. More importantly, unlike the aforementioned works, our estimator is attained by a {\it provably} fast converging algorithm.
	
	\begin{proof}[Proof sketch of Theorem~\ref{thm:vec:sparse}]
		Note that not all the entries of the sub-gradient $\G_l$ affect $\Bbeta_{l+1}$. Denote 
		$$
		\Omega_{l+1}:=\text{supp}(\Bbeta_{l+1}),\quad\Omega_{l}:=\text{supp}(\Bbeta_{l}),\quad \Pi_l:=\text{supp}\big(\calH_{\tilde{s}}(\calP_{\Omega_l^{\rm c}}(\G_l))\big), \quad \Omega^*:=\text{supp}(\Bbeta^*),
		$$
		where the operator $\calP_{\Omega}(\G)$ keeps the entries $\G$ in $\Omega$ unchanged and zeros out the others and $\Omega^{c}$ is complement of $\Omega$.  It is clear that $\Omega_{l+1}\subseteq\Omega_{l}\cup\Pi_l\subseteq\Omega_{l}\cup\Pi_l\cup\Omega^*$. We can re-write the update rule as
		\begin{align*}
			\Bbeta_{l+1}=\calH_{\tilde{s}}\left(\Bbeta_l-\eta_{l}\G_l\right)=\calH_{\tilde{s}}\left(\Bbeta_l-\eta_{l}\calP_{\Omega_{l}\cup\Pi_l\cup\Omega^*}(\G_l)\right).
		\end{align*}
		It thus suffices to focus on the sparse vector $\calP_{\Omega_{l}\cup\Pi_l\cup\Omega^*}(\G_l)$. The following lemma depicts the two-phase regularity properties of the loss $f(\Bbeta)$  around the oracle $\Bbeta^*$.
		
		\begin{lemma}
			\label{lem:vec:sparse}
			Suppose Assumptions~\ref{assump:sensing operators:vec} and \ref{assump:heavy-tailed} hold. There exist absolute constants $C_1,C_2,C_3,c_0>0$ such that if $n\geq C_1\ku\kl^{-1}\tilde{s}\log(2d/\tilde{s})$, then with probability over $1-\exp\big(-c_0\tilde{s}\log \big(2d/\tilde{s})\big)-3\exp\big(-(n\log(2d/\tilde{s}))^{1/2}\log^{-1} n\big)$, we have
			\begin{enumerate}[(1)]
				\item for all $\Bbeta\in\left\{\Bbeta\in\RR^{d}:\; \ltwo{\Bbeta-\Bbeta^*}\geq 8\kl^{-1/2}\gamma,\ |\text{supp}(\Bbeta)|\leq\tilde{s}\right\}$ and for all sub-gradient $\G\in\partial f(\Bbeta)$, 
				\begin{align*}
					f(\Bbeta)-f(\Bbeta^*)\geq \frac{n}{4}\kl^{1/2}\ltwo{\Bbeta-\Bbeta^*}\quad {\rm and}\quad \ltwo{\calP_{\Omega\cup\Pi\cup\Omega^*}(\G)}\leq n \ku^{1/2},
				\end{align*}
				where $\Omega:=\text{supp}(\Bbeta)$ and $\Pi:=\text{supp}\big(\calH_{\tilde{s}}(\calP_{\Omega^{\rm c}}(\G))\big)$;
				
				\item for all $\Bbeta\in\left\{\Bbeta\in\RR^{d}:\ C_2\ku^{1/2}\kl^{-1}\cdot b_0\big(\tilde{s}/n\cdot
				\log(2d/\tilde{s})\big)^{1/2}\leq\ltwo{\Bbeta-\Bbeta^*}\leq8\kl^{-1/2}\gamma,\ |\text{supp}(\Bbeta)|\leq\tilde{s}\right\}$ and for all sub-gradient $\G\in\partial f(\Bbeta)$, 
				\begin{align*}
					f(\Bbeta)-f(\Bbeta^*)\geq \frac{n\kl}{12b_0}\ltwo{\Bbeta-\Bbeta^*}^2\quad {\rm and}\quad \ltwo{\calP_{\Omega\cup\Pi\cup\Omega^*}(\G)}\leq C_3\frac{n\ku}{b_1}\ltwo{\Bbeta-\Bbeta^*},
				\end{align*}
				where $\Omega:=\text{supp}(\Bbeta)$ and $\Pi:=\text{supp}\big(\calH_{\tilde{s}}(\calP_{\Omega^{\rm c}}(\G))\big)$.
			\end{enumerate}
		\end{lemma}
		
		Based on Lemma~\ref{lem:vec:sparse}, we show that during phase one
		$$
		\|\Bbeta_l-\eta_l\calP_{\Omega_l\cup \Pi_l\cup \Omega^{\ast}}(\G_l)-\Bbeta^{\ast}\|_2^2\leq \|\Bbeta_l-\Bbeta^{\ast}\|_2^2-\frac{n\eta_l\kl^{1/2}}{2}\|\Bbeta_l-\Bbeta^{\ast}\|_2+n^2\eta_l^2 \ku.
		$$
		The selected stepsize guarantees $\eta_l\asymp \kl^{1/2}\ku^{-1}\cdot n^{-1}(1-c_1)^l\|\Bbeta_0-\Bbeta^{\ast}\|_2$, by which the contraction of $\|\Bbeta_{l+1}-\Bbeta^{\ast}\|_2$ is derived using the fact $\Bbeta_{l+1}=\calH_{\tilde{s}}\left(\Bbeta_l-\eta_{l}\calP_{\Omega_{l}\cup\Pi_l\cup\Omega^*}(\G_l)\right)$. During phase two, by Lemma~\ref{lem:vec:sparse}, we get
		$$
		\|\Bbeta_l-\eta_l\calP_{\Omega_l\cup \Pi_l\cup \Omega^{\ast}}(\G_l)-\Bbeta^{\ast}\|_2^2\leq \bigg(1-\frac{\eta_l n \kl}{6b_0}+\frac{\eta_l^2C_3^2n^2\ku^2}{b_1^2}\bigg)\|\Bbeta_l-\Bbeta^{\ast}\|_2^2.
		$$
		Similarly, the selected constant stepsize guarantees the contraction of $\|\Bbeta_{l+1}-\Bbeta^{\ast}\|_2$. 
	\end{proof}
	
	\begin{remark}[{Smoothing effect of random noise}]\label{rmk:smootheffect}
		By Theorem~\ref{thm:vec:sparse}, the phase two convergence of Algorithm~\ref{alg:IHTl1} requires a constant stepsize as if minimizing a smooth function. Lemma~\ref{lem:vec:sparse} shows that, in the small region around $\Bbeta^{\ast}$ where the phase two convergence occurs, the loss function $f(\Bbeta)$ indeed possesses the regularity conditions which usually appear in smooth optimization,  e.g.,  the square loss \citep{cai2021generalized,zhao2015nonconvex} and logistic loss \citep{lyu2021latent}.  The second phase occurs when the individual random noise starts to (stochastically) dominate $\|\Bbeta_l-\Bbeta^{\ast}\|_2$.  This suggests that, when $\Bbeta^{\ast}$ is close enough to $\Bbeta^{\ast}$,  the random noise has the effect of smoothing the objective function. In Figure~\ref{fig:smooth}, we present the shape of the stochastic function $g(t):=n^{-1}\sum_{i=1}^n |\xi_i-t|$ and also its sub-differential where $\xi_i$'s are i.i.d. $N(0, \tau^2)$. 
		The function is clearly non-smooth when $\tau=0$ and becomes smoother as $\tau$ increases.  Randomized smoothing has been observed in optimization and statistics literature.   For instance, \cite{zhang2020edgeworth} proved that the random edge-wise observational error can naturally smooth the Edgeworth expansion of the network moment statistics, whose empirical c.d.f. can be a step function without the smoothing effect. In order to optimize a non-smooth convex function $h(\Bbeta)$, e.g. $\|\Bbeta\|_1$,  with optimal convergence rates, \cite{duchi2012randomized} proposed to optimize the randomly-smoothed version $\EE h(\Bbeta+Z)$ where $Z$ is some random variable with a known distribution. 
		
		\begin{figure}[t]
			\centering
			\begin{subfigure}[b]{0.45\textwidth}
				\centering
				\includegraphics[width=\textwidth]{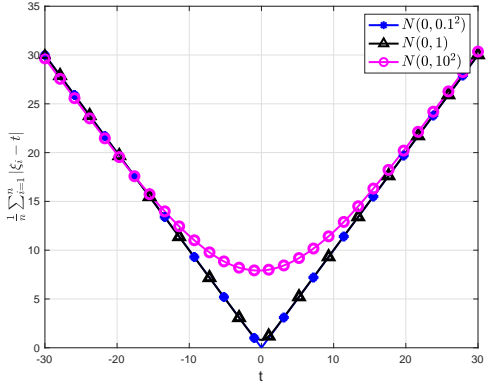}
				\caption{Value of $\frac{1}{n}\sum_{i=1}^{n}|\xi_i-t| $ against $t$ }
				\label{fig:f_value_smooth}
			\end{subfigure}
			\hfill
			\begin{subfigure}[b]{0.45\textwidth}
				\centering
				\includegraphics[width=\textwidth]{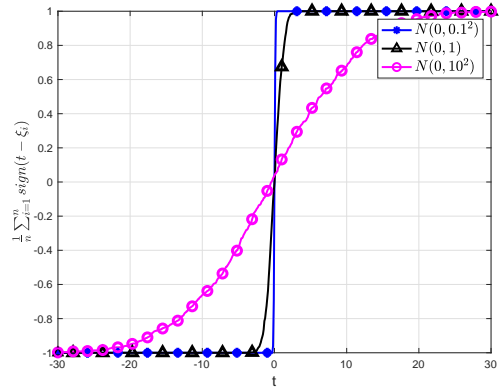}
				\caption{Value of $\frac{1}{n}\sum_{i=1}^{n}\text{sign}(t-\xi_i) $ against $t$}
				\label{fig:df_value_smooth}
			\end{subfigure}
			\caption{Sample size $n=1000$ and $\{\xi_i\}_{i=1}^n$ follow $N(0,0.1^2)$, $N(0,1)$, $N(0,10^2)$ respectively.}
			\label{fig:smooth}
		\end{figure}   
	\end{remark}
	
\section{Low-rank Linear Regression}\label{sec:matrix}

Let $\M^{\ast}$ be a $d_1\times d_2$ matrix with rank $r\ll d_2$. Without loss of generality, we assume $d_1\geq d_2$. Denote $\{(\X_i, Y_i)\}_{i=1}^n$ the collection of i.i.d.  observations satisfying $Y_i=\langle \X_i, \M^{\ast}\rangle+\xi_i:={\rm tr}(\X_i^{\top}\M^{\ast})+\xi_i$ 
with $\X_i\in\RR^{d_1\times d_2}$ known as the {\it measurement matrix}. Our goal is to estimate $\M^{\ast}$ by solving the optimization program
\begin{align}\label{eq:loss:matrix}
	\hat\M:=\underset{\M\in\MM_r}{\arg\min}\ f(\M)  \quad \textrm{ where } f(\M) := \sum_{i=1}^n\rho\big(\langle \M,  \X_i\rangle-Y_i\big),
\end{align}
where $\MM_r:=\{\M\in\RR^{d_1\times d_2}:\; \text{rank}(\M)\leq r\}$. Here $\rho(\cdot)$ is a general robust loss function including the absolute loss \citep{tong2021low, charisopoulos2021low}, Huber loss \citep{elsener2018robust,sun2020adaptive} and quantile loss \citep{alquier2019estimation,chen2022robust}.

\subsection{Algorithm: Riemannian optimization} Our algorithm is the projected sub-gradient descent on the Riemannian manifold (RsGrad). At the $l$-th iteration with a current low-rank estimate $\M_l$,  the algorithm consists of two major steps.  It begins with computing the Riemannian sub-gradient,  which is the projection of a {\it vanilla} sub-gradient $\G_l\in\partial f(\M_l)\subset \RR^{d_1\times d_2}$ onto the tangent space of $\MM_r$ at the point $\M_l$,  denoted by $\TT_l$. The second step is to  update the low-rank estimate along the direction of negative Riemannian sub-gradient and then retract it back to the manifold $\MM_r$,  for which,  it suffices to take the SVD.  Here ${\rm SVD}_r(\cdot)$ returns the best rank-$r$ approximation by SVD. The details can be found in Algorithm~\ref{alg:RsGrad}. 

We note that the Riemannian sub-gradient is used only for the benefit of computational efficiency. The theoretical results all hold if taking the vanilla sub-gradient. The detailed computation of Riemannian sub-gradient can be found in the supplementary file. RsGrad has its origin in \cite{cambier2016robust} for minimizing the absolute loss without convergence and statistical analysis.  Our framework covers general robust loss functions,  and  we prove its computational efficiency and statistical optimality in several important applications.

\paragraph*{Comparison with existing algorithms} Sub-gradient descent algorithms have been investigated by \cite{tong2021low,charisopoulos2021low} for low-rank linear regression with the absolute loss. Our algorithm is different in two aspects: we take Riemannian sub-gradient and have two phases of stepsize schedule.  Factorization-based optimizations \citep{zheng2016convergent,zhao2015nonconvex} reparametrize the program \eqref{eq:loss:matrix} by $\M=\U\V^{\top}$.  Then it suffices to update $(\U,  \V)$ sequentially to minimize (\ref{eq:loss}) by sub-gradient descent algorithms. These algorithms can suffer a great loss of computational efficiency if $\M^{\ast}$ is {\it ill-conditioned}. Though the issue can be theoretically remedied by a proper inverse scaling \citep{tong2021accelerating},  it can cause a potential computational instability especially when $\M^{\ast}$ has small non-zero singular values. Recently,  it is discovered that optimizing (\ref{eq:loss}) directly on the low-rank manifold $\MM_r$,   called {\it Riemannian},  enjoys the fast computational speed of factorization-based approaches and,  meanwhile,  converges linearly regardless of the condition number of $\M^{\ast}$. See,  e.g.,  \cite{cai2021generalized} and \cite{cai2021provable},  for the convergence of Riemannian gradient descent (RGrad) algorithms in minimizing {\it strongly convex and smooth} functions with tensor-related applications. We note that RGrad is similar to the projected gradient descent (PGD, \cite{chen2015fast}) except that RGrad utilizes the Riemannian gradient while PGD takes the vanilla one. 

\begin{algorithm}
	\caption{Riemannian Sub-gradient Descent (RsGrad)}\label{alg:RsGrad}
	\begin{algorithmic}
		\STATE{\textbf{Input}: observations $\{(\X_i, Y_i)\}_{i=1}^n$,  max iterations $\lmax$,  step sizes $\{\eta_l\}_{l=0}^{\lmax}$.}
		\STATE{Initialization: $\M_0\in\MM_r$}
		\FOR{$l = 0,\ldots,\lmax$}
		\STATE{Choose a vanilla subgradient:  $\G_l\in\partial f(\M_l)$}
		\STATE{Compute the Riemannian sub-gradient: $\wt\G_l = \calP_{\TT_l}(\G_l)$}
		\STATE{Retraction to $\MM_r$: $\M_{l+1} = \text{SVD}_r(\M_l - \eta_{l}\wt\G_l)$}
		\ENDFOR
		\STATE{\textbf{Output}: $\hat\M=\M_{\lmax}$}
	\end{algorithmic}
\end{algorithm}

\subsection{General Convergence Performance}
We now present the general convergence performance of RsGrad Algorithm~\ref{alg:RsGrad},  which essentially relies on the regularity conditions of the objective function.   These conditions,  in spirit,  largely inherit those from existing literature \citep{charisopoulos2021low,elsener2018robust,alquier2019estimation,tong2021accelerating}.  However,  as explained in Section~\ref{sec:intro},  these prior works only delivered statistically sub-optimal estimates.  It turns out that more delicate characterizations of these conditions are necessary for our purpose.  More exactly,  we discover that the aforementioned robust functions exhibit strikingly different regularity conditions near and   far away from $\M^{\ast}$,  referred to as the {\it two-phase} regularity conditions. Here $\|\cdot\|_{\rm F}$ represents the Frobenius norm of a matrix. 

\begin{condition}\label{assump:two-phase}(Two-phase regularity conditions)
	Define the two regions around $\M^{\ast}$ with $\tauc>\taus>0$ by
	$$
	\BB_1:=\left\{\M\in\MM_r: \|\M-\M^{\ast}\|_{\rm F}\geq \tauc\right\}\quad {\rm and}\quad \BB_2:=\left\{\M\in\MM_r: \tauc>\|\M-\M^{\ast}\|_{\rm F}\geq \taus\right\}.
	$$
	The function $f(\cdot): \RR^{d_1\times d_2}\mapsto \RR_+$ satisfies
	\begin{enumerate}[(1)]
		\item the rank-$r$ restricted $(\tauc,\taus,\muc,\mus)$ {\bf two-phase sharpness} with respect to $\M^{\ast}$,
		$$
		f(\M)-f(\M^{\ast})\geq 
		\begin{cases}
			\muc \|\M-\M^{\ast}\|_{\rm F}, & \textrm{ for }\ \  \M\in\BB_1; \\
			\mus \|\M-\M^{\ast}\|_{\rm F}^2, & \textrm{ for }\ \  \M\in\BB_2;
		\end{cases}
		$$
		\item the rank-r restricted $(\tauc,\taus,\Lc,\Ls)$ {\bf two-phase sub-gradient bound} with respect to $\M^{\ast}$,
		$$
		\|\G\|_{\rm F,  r}\leq 
		\begin{cases}
			\Lc, & \textrm{ for }\ \ \M\in\BB_1;\\
			\Ls\|\M-\M^{\ast}\|_{\rm F},& \textrm{ for }\ \ \M\in\BB_2,
		\end{cases}
		$$
		where $\G\in\partial f(\M)$ and the truncated Frobenius norm $\|\G\|_{\rm F, r}:=\|{\rm SVD}_r(\G)\|_{\rm F}$.
	\end{enumerate}
\end{condition}
Basically,  Condition~\ref{assump:two-phase} dictates the distinct behaviors of  $f(\cdot)$ in two neighbourhoods of $\M^{\ast}$.  The quantity $\taus$ reflects the statistical limit while $\tauc$ is usually the rate achieved by the computational analysis in existing literature \citep{tong2021accelerating,charisopoulos2021low} without assuming noise distributions.  These prior works only reveal the first phase regularity conditions,  i.e.,  on $\BB_1$.  Oftentimes, we have $\tauc\gg \taus$,  in which case  these prior works only deliver statistically sub-optimal estimates.  
Deriving the second phase regularity condition is challenging where more precise calculations are necessary.  We remark that,  in statistics literature,  the sub-gradient bound is related to the Lipschitz continuity \citep{alquier2019estimation} and the  sharpness condition is called the {\it one-point-margin condition} \citep{elsener2018robust}.  However,  their estimators are built upon convex program so that the analysis is only made in a small neighbour of $\M^{\ast}$ where the two-phase regularity conditions are not pivotal.  
We summarize the two-phase regularity conditions for various applications in Table~\ref{table:two-phase}. Proposition~\ref{prop:main} characterizes the general dynamic of Algorithm~\ref{alg:RsGrad}. 

\begin{table}
	\begin{center}
		\begin{tabular}{||l|l| c c c c c c ||}
			\hline
			Loss function&Noise&$\tauc$&$\taus$&$\muc$&$\Lc$&$\mus$&$\Ls$\\
			\hline
			$\ell_1$-Loss&$N(0,\sigma^2)$&$\sigma$&$O\left(\sigma\left(\frac{rd_1}{n}\right)^{1/2}\right)$&$\frac{n}{12}$&$2n$&$\frac{n}{12\sigma}$&$O\left(\frac{n}{\sigma}\right)$\\
			$\ell_1$-Loss&Heavy-tailed&$8\gamma$&$O\left(b_0\left(\frac{rd_1}{n}\right)^{1/2}\right)$&$\frac{n}{4}$&$2n$&$\frac{n}{12b_0}$&$O\left(\frac{n}{b_1}\right)$\\
			Huber $\rho_{H,\delta}$&Heavy-tailed&$8\gamma+2\delta$&$O\left(b_0\left(\frac{rd_1}{n}\right)^{1/2}\right)$&$\frac{\delta n}{2}$&$4\delta n$&$\frac{\delta n}{3b_0}$&$O\left(\frac{\delta n}{b_1}\right)$\\
			Quantile $\rho_{Q,\delta}$ &Heavy-tailed&$8\gamma$&$O\left(b_0\left(\frac{rd_1}{n}\right)^{1/2}\right)$&$\frac{n}{8}$&$ n$&$\frac{ n}{24b_0}$&$O\left(\frac{n}{b_1}\right)$\\
			\hline
		\end{tabular}
	\end{center}
	\caption{Summary of the two phase regularity conditions. In fact, $\tauc$ is the error rate achieved after the phase one iterations and $\taus$ is the error rate achieved after phase two iterations, i.e., the ultimate statistical accuracy. Here $\gamma, b_0$ and $b_1$ are as described in Assumption~\ref{assump:heavy-tailed}. Note that $b_0\asymp \gamma$ in many cases such as the Gaussian noise, zero symmetric Pareto noise, and Student's t noise, etc. 
		By viewing $b_0/\gamma$ as a constant,  we have $\taus\ll \tauc$ if the sample size is large.}
	\label{table:two-phase}
\end{table}


\begin{proposition}\label{prop:main}
	Suppose Condition~\ref{assump:two-phase} holds,  the initialization $\M_0\in\MM_r$ satisfies $\fro{\M_0-\M^*}\leq c_0\sigma_r(\M^{\ast})\cdot\min\{\mus^2\Ls^{-2},  \muc^{2}\Lc^{-2}\}$ for a small but absolute constant $c_0>0$,  and the initial stepsize $\eta_0\in\left[0.2\fro{\M_0-\M^*}\muc\Lc^{-2}, \ 0.3\fro{\M_0-\M^*}\muc\Lc^{-2}\right] $.   At the $l$-th iteration of Algorithm~\ref{alg:RsGrad},  
	\begin{enumerate}[(1)]
		\item when $\fro{\M_{l}-\M^*}\geq\tauc $,  namely in phase one,  take the stepsize $\eta_{l}=\big(1-c\muc^2\Lc^{-2} \big)^{l}\eta_{0}$ with any fixed constant $c\leq0.04$,  then we have
		\begin{align*}
			\Vert \M_{l+1}-\M^{*}\Vert_{\mathrm{F}}\leq \left(1-c\frac{\muc^2}{\Lc^2}\right)^{l+1}\cdot \| \M_{0}-\M^{*}\|_{\rm F};
		\end{align*}
		\item when $\tauc>\fro{\M_l-\M^*}\geq\taus$,  namely in phase two,  take the stepsize $\eta_l=\eta\in \mus\Ls^{-2}\cdot\big[0.125,\ 0.75\big]$, then we have 
		\begin{align*}
			\fro{\M_{l+1} -\M^{*}} \leq \left(1-\frac{\mus^2}{32\Ls^2}\right)\cdot \fro{\M_{l}-\M^{*}}.
		\end{align*}
	\end{enumerate}
\end{proposition}

By Proposition~\ref{prop:main},  the RsGrad Algorithm~\ref{alg:RsGrad} converges linearly using the proposed stepsizes and outputs an estimate within an $O(\taus)$ distance in Frobenius norm from the oracle.  As discussed above,  the phase one convergence is typical in non-smooth optimization,  for which the decaying stepsizes are needed.  For simplicity,  we apply the geometrically decaying stepsizes \citep{tong2021accelerating,  goffin1977convergence}.  The phase two convergence behaves like smooth optimization so that a fixed stepsize suffices to yield the linear convergence.  Note that the convergence dynamic is free of the matrix condition number, which is a benefit of Riemannian-type algorithms.   

\paragraph*{Comparison with prior works} The algorithmic dynamic of RsGrad in phase one is similar to that in \cite{charisopoulos2021low} and \cite{tong2021low},  where the geometrically decaying stepsizes in phase one will eventually bring the stepsize into the level $\Omega(\mus\Ls^{-2})$ which is desired by the phase-two convergence.   However,  \cite{charisopoulos2021low} and \cite{tong2021low} will continue to shrink the stepsize geometrically and fail to deliver a statistically optimal estimator.   In contrast, we set a constant stepsize in phase two and the ultimate estimator is statistically optimal.  The difference is observed in numerical experiments in Section~\ref{sec:simulation}.  

\subsection{Applications}
\label{sec:app}

This section presents the statistical performances of RsGrad in some concrete examples. Due to page limit, some examples are deferred to the supplementary file.  We note that,  compared with the square loss \citep{xia2021statistically,zhao2015nonconvex,chen2015fast},  low-rank regression using the robust loss functions requires no signal-to-noise ratio (SNR) conditions (except for initialization). 

\subsubsection{Absolute loss with Gaussian noise}
Though motivated by heavy-tailed noise, to start with, we demonstrate that the absolute loss can yield a statistically optimal estimator under Gaussian noise. The objective function is $f(\M)=\sum_{i=1}^n |Y_i-\langle \M,  \X_i\rangle |$.

\begin{theorem}\label{thm:Gaussian-l1}
	Assume $\xi_1,\dots, \xi_n \stackrel{i.i.d.}{\sim} N(0,\sigma^2)$ and $\{{\rm vec}(\X_i)\}_{i=1}^n$ satisfy Assumption~\ref{assump:sensing operators:vec}.  There exist absolute constants $c_0,c_1,  c_2\in(0,1), C_1, C_2>0$ such that if the sample size $n\geq C_1\ku\kl^{-1}d_1r$,  the initialization $\fro{\M_0-\M^*}\leq c_0\sigma_r$ for $\M_0\in\MM_r$, the initial stepsize $\eta_{0}\in \kl^{1/2}(48n\ku)^{-1}\fro{\M_0-\M^*}\cdot\left[0.2,0.3\right]$, the phase one stepsize $\eta_l=(1-c_3)^{l}\eta_0$ with any fixed constant $c_3\leq c_1\kl\ku^{-1}$, and the phase two stepsize $\eta_{l}=\eta\asymp \kl\ku^{-2}\sigma/n$, then with probability at least  $1-\exp(-c_2rd_1)-3\exp(-\sqrt{n}/\log n)$, after at most $O\big(\log(\sigma_{r}/\sigma)+\log(n/d_1)\big)$ iterations, the Algorithm~\ref{alg:RsGrad} outputs an estimator with the error rate
	$$
	\|\hat \M -\M^{\ast}\|_{\rm F}^2\leq C_2\frac{\ku}{\kl^2}\cdot\frac{\sigma^2d_1r}{n}.
	$$
\end{theorem}
Note that the initial stepsize $\eta_0$ can be as large as $O(\sigma_rn^{-1})$ depending on the initialization.   The attained error rate is minimax optimal \citep{ma2015volume,xia2014optimal}.  In comparison,  the algorithms proposed in \cite{tong2021low,charisopoulos2021low} and their analysis only achieved the sub-optimal rate $O_p(\sigma^2)$.  The proof of Theorem~\ref{thm:Gaussian-l1} is a combination of the Proposition~\ref{prop:main} and the two-phase regularity conditions of $f(\M)$. The two phase convergence dynamics are observed in numerical experiments in Section~\ref{sec:simulation}.

\subsubsection{Absolute loss with heavy-tailed noise}
We now demonstrate the effectiveness of the absolute loss in handling heavy-tailed noise. Note that in Theorem~\ref{thm:heavytail-l1} if a constant factor depends on $b_0$ or/and $b_1$,    a star sign is placed on top right of it.

\begin{theorem}\label{thm:heavytail-l1}
	Assume $\{{\rm vec}(\X_i)\}_{i=1}^n$ and  $\{\xi_i\}_{i=1}^n$ satisfy Assumptions~\ref{assump:sensing operators:vec} and \ref{assump:heavy-tailed}, respectively. There exist constants $c_0^{\ast},  c_1,c_2\in(0,1 ), C_1,  C_2>0$ such that if the sample size $n\geq C_1\ku\kl^{-1}d_1r$,  the initialization $\fro{\M_0-\M^*}\leq c_0^{\ast}\sigma_{r}$ for $\M_0\in\MM_r$, the initial stepsize $\eta_{0}\in \kl^{1/2}(16n\ku)^{-1}\fro{\M_0-\M^*}\cdot\left[0.2, 0.3\right]$, the phase one stepsize $\eta_{l}=(1-c_3)^{l}\eta_{0}$ with any fixed constant $c_3\leq c_1$, and the phase two stepsize $\eta_{l}=\eta\asymp \kl \ku^{-2}b_1^2(nb_0)^{-1}$, then with probability at least $1-\exp(-c_2rd_1)-3\exp(-\sqrt{n}/\log n)$, after at most $O\big(\log(\sigma_r/\gamma)+\log\big(\gamma n/(d_1rb_0)\big)\big)$ iterations, the Algorithm~\ref{alg:RsGrad} outputs an estimator with the error rate 
	$$
	\|\hat \M-\M^{\ast}\|_{\rm F}^2\leq \frac{C_2\ku}{\kl^2}\cdot \frac{b_0^2 d_1 r}{n}.
	$$
\end{theorem}

The proof of Theorem~\ref{thm:heavytail-l1} is an immediate consequence of Proposition~\ref{prop:main} if we have the two-phase regularity conditions on $f(\M)$ under heavy-tailed noise, for which we borrowed a powerful tool of bounding an empirical process \citep{adamczak2008tail,van2000empirical}. This allows us to derive a stronger tail probability bound by avoiding bounding the term $\sum_{i=1}^n |\xi_i|$. Indeed, since $\xi$ has a finite $1+\eps$ moment, bounding the sum of $|\xi_i|$ results into a tail probability $1-O(n^{-\eps})$. This is not a desirable tail probability for a small $\eps$. 
Note that the Gaussian noise also satisfies Assumption~\ref{assump:heavy-tailed},  in which case the attained rate in Theorem~\ref{thm:heavytail-l1} has a matching minimax lower bound. The rate is minimax optimal w.r.t. the dimension and sample size.  Note that $b_0\asymp \gamma=\EE |\xi|$ in many cases such as the Gaussian, zero symmetric Pareto, and Student's t noise. 

The density requirement in Assumption~\ref{assump:heavy-tailed} is similar to the one in \cite{elsener2018robust} whose estimator is based on the nuclear-norm penalized convex program without considering the computational efficiency. The same error rate $O(b_0^2d_1r/n)$ was also derived there. But our approach computes fast and requires only a logarithmic factor of iterations. The analysis frameworks in \cite{tong2021low,charisopoulos2021low} would produce the sub-optimal rate $O(\gamma^2)$ under the heavy-tailed noise, which won't decrease as the sample size $n$ gets larger. The error rates derived by \cite{fan2021shrinkage,minsker2015geometric} and \cite{minsker2018sub} were not proportional to the noise level and were not zero even if the observations are noiseless. In contrast, our algorithm can exactly recover $\M^{\ast}$ when there is no noise. 

\subsubsection{Huber loss with heavy-tailed noise}
Huber loss is prevalent in robust statistics \citep{huber1965robust,sun2020adaptive,elsener2018robust} and defined by $\rho_{H,\delta}(x):=x^2\mathbbm{1}(|x|\leq \delta)+(2\delta|x|-\delta^2)\mathbbm{1}(|x|>\delta)$ where $\delta>0$ is referred to as the robustification parameter.   The function $\rho_{H,\delta}(\cdot)$ is Lipschitz with a constant $2\delta$.  The loss function is given by $	f(\M):=\sum_{i=1}^n \rho_{H,\delta}(Y_i-\langle \M,  \X_i\rangle)$. Due to technical issues, we need a slightly different assumption of the heavy-tailed noise. 

\begin{assumption}\label{assump:heavytail-huber}(Heavy-tailed noise \Romannum{2}) 
	There exists an $\eps>0$ such that $\EE |\xi|^{1+\eps}<+\infty$.  The noise distribution function $H_{\xi}(\cdot)$ is symmetric\footnote{The symmetric condition can also be replaced by $\int_{-\delta}^{\delta} H_\xi(x)\, dx=\delta$.} in that  $H_{\xi}(x)=1-H_{\xi}(-x)$. Denote $\gamma=\EE|\xi|$. There exist constants $b_0, b_1$ (may be dependent on $\gamma$ and $\delta$) such that 
	\begin{align*}
		H_{\xi}(x+\delta)-H_{\xi}(x-\delta)\geq 2\delta b_0^{-1}, &\ \ \  \textrm{ for all } |x|\leq 8(\ku/\kl)^{1/2}\gamma+2(\ku/\kl)^{1/2}\delta;\\
		H_{\xi}(x+\delta)-H_{\xi}(x-\delta)\leq 2\delta b_1^{-1}, &\ \ \ \forall x\in \RR,
	\end{align*}
	where $\delta$ is the Huber loss parameter.   
\end{assumption}

Theorem~\ref{thm:heavytail-huber} describes the convergence and statistical performance of RsGrad Algorithm~\ref{alg:RsGrad} for Huber loss.  Similarly,  the constants dependent on $b_0$ or/and $b_1$ are marked with a star.  

\begin{theorem}\label{thm:heavytail-huber}
	Assume $\{{\rm vec}(\X_i)\}_{i=1}^n$ and $\{\xi_i\}_{i=1}^n$ satisfy Assumptions~ \ref{assump:sensing operators:vec} and \ref{assump:heavytail-huber}, respectively. There exist constants $c_0^{\ast}, c_1,c_2\in(0,1 ), C_{1},  C_{2}>0$ such that if the sample size $n\geq C_1\ku\kl^{-1}d_1r$,  the initialization $\fro{\M_0-\M^*}\leq c_0^{\ast}\sigma_r$ for $\M_0\in\MM_r$, the initial stepsize $\eta_{0}\in\kl^{1/2}(32n\delta\ku)^{-1}\fro{\M_0-\M^*}\cdot\left[0.2,0.3\right]$,  the phase one stepsize $\eta_{l}=(1-c_3)^{l}\eta_{0}$ with any fixed constant $c_3\leq c_1$, and the phase two stepsize $\eta_{l}=\eta\asymp \kl \ku^{-2}b_1^2(n\delta b_0)^{-1}$, then with probability at least $1-\exp(-c_2rd_1)-3\exp(-\sqrt{n}/\log n)$, after at most $O\big(\log(\sigma_{r}/\gamma)+\log\big(\gamma n/(b_0d_1r)\big)\big)$ iterations, the Algorithm~\ref{alg:RsGrad} outputs an estimator with the error rate 
	$$
	\|\hat \M-\M^{\ast}\|_{\rm F}^2\leq \frac{C_2\ku}{\kl^2}\cdot \frac{b_0^2 d_1 r}{n}.
	$$
\end{theorem}

Note that Assumption~\ref{assump:heavytail-huber} implies that $b_0$ is lower  bounded by $\Omega(\gamma+\delta)$ and thus the error rate is lower bounded by $\Omega\big((\delta^2+\gamma^2) rd_1/n\big)$. This suggests an interesting transition with respect to the robustification parameter: when $\delta\lesssim \gamma$,  increasing $\delta$ does not affect the error rate; when $\delta\gtrsim \gamma$,  the error rate appreciates if $\delta$ becomes larger.  
\cite{sun2020adaptive} derived the statistical performance of Huber loss for linear regression without imposing regularity conditions on the noise density resulting into 
an error rate that is slower than $O(n^{-1})$.

\subsubsection{Quantile loss with heavy-tailed noise}
Quantile loss was initially proposed by \cite{koenker1978regression} and has been a popular loss function for robust statistics  \citep{welsh1989m,koenker2001quantile,wang2012quantile,tan2022high}.  Denote the quantile loss by $\rho_{Q,\delta}:=\delta x\mathbbm{1}(x\geq 0)+(\delta-1)x\mathbbm{1}(x<0)$ for any $x\in \RR$ with $\delta:=\PP(\xi\leq 0)$. Notice that the function $\rho_{Q,\delta}(\cdot)$ is Lipschitz with a constant $\max\{\delta,1-\delta\}$. The loss function is given by $f(\M):=\sum_{i=1}^n \rho_{Q,\delta}(Y_i-\langle \M,  \X_i\rangle)$. Similar to the case of Huber loss, due to technical reasons, our analysis framework needs a slightly different assumption on the heavy-tailed noise.

\begin{assumption}(Heavy-tailed noise \Romannum{3}) \label{assump:heavy-tailed quantile}
	There exists an $\varepsilon>0$ such that $\EE |\xi|^{1+\varepsilon}<+\infty$.  Suppose $\delta:=H_{\xi}(0)$ lies in $(0,1)$. Denote $\gamma=\EE|\xi|$.  There exist constants $b_0, b_1>0$ (may be dependent on $\gamma$) such that
	\begin{align*}
		h_{\xi}(x)\geq b_0^{-1}, &\ \ \  \textrm{ for all } |x|\leq 8(\ku/\kl)^{1/2}\gamma;\\
		h_{\xi}(x)\leq b_1^{-1}, &\ \ \ \forall x\in \RR.
	\end{align*}
\end{assumption}

When the noise has a zero median,  namely $\delta=1/2$, Assumption~\ref{assump:heavy-tailed quantile} becomes identical to the Assumption~\ref{assump:heavy-tailed}. 

\begin{theorem}\label{thm:heavytail-quantile}
	Assume $\{\text{vec}(\X_i)\}_{i=1}^n$ and  $\{\xi_i\}_{i=1}^n$ satisfy Assumptions~\ref{assump:sensing operators:vec} and \ref{assump:heavy-tailed quantile}, respectively. There exist constants $c_0^{\ast},c_1, c_2\in(0,1 ), C_1,  C_2>0$ such that if sample size $n\geq C_1\ku\kl^{-1}d_1r$, the initialization $\fro{\M_0-\M^*}\leq c_0^{\ast}\sigma_r$, the initial stepsize $\eta_{0}\in \kl^{1/2}(8n\ku)^{-1} \fro{\M_0-\M^*}\cdot[0.2,0.3]$, the phase one stepsize $\eta_l=(1-c_3)^l\eta_0$ with any fixed constant $c_3\leq c_1$, and the phase two stepsize $\eta_l=\eta\asymp \kl\ku^{-2}b_1^2(nb_0)^{-1}$, then with probability over $1-\exp(-c_2rd_1)-3\exp(-\sqrt{n}/\log n)$, after at most $O\big(\log(\sigma_{r}/\gamma)+\log\big(\max\{\delta^{-1}, (1-\delta)^{-1}\}\cdot\gamma n/(d_1rb_0)\big)\big)$ iterations, the Algorithm~\ref{alg:RsGrad} outputs an estimator with the error rate
	$$
	\|\hat \M - \M^{\ast}\|_{\rm F}^2\leq \frac{C_2\ku\cdot \max\{\delta^2,\ (1-\delta)^2\}}{\kl^2}\cdot \frac{b_0^2 rd_1}{n}. 
	$$
\end{theorem}

\section{Initialization and Algorithmic Parameters Selection}\label{sec:discussion}
\paragraph*{Initialization} 
The sparse linear regression Algorithm~\ref{alg:IHTl1} allows arbitrary initialization. We only discuss the initialization for low-rank linear regression. 
The convergence of Algorithm~\ref{alg:RsGrad} crucially relies on the warm initialization $\M_0$.  One can simply apply the shrinkage low-rank approximation \citep{fan2021shrinkage} for initialization but a finite $2+\eps$ moment condition on noise is required.  For simplicity,  we investigate the performance of the  vanilla low-rank approximation. Here $vec(\cdot)$ and $mat(\cdot)$ denote the vectorization and matricization operators, respectively. 

\begin{theorem}(Initialization Guarantees)\label{thm:init-general-known} Suppose  $\gamma_1:=\EE|\xi|^{1+\varepsilon}<+\infty$ for some $\varepsilon\in (0,1]$ and $\{vec(\X_i)\}_{i=1}^n$ satisfy Assumption~\ref{assump:sensing operators:vec} with known covariances $\{\bSigma_i\}_{i=1}^n$. 
	Initialize by $$\M_0:=\operatorname{SVD}_r(n^{-1}\sum_{i=1}^{n}\text{mat}(\bSigma_i^{-1}\text{vec}(\X_{i}))Y_i).$$ For any small $c_0>0$,  there exist constants $C, c_1,c_2,c_3>0$ depending only on $c_0$ such that if $n>C\max\big\{\ku\kl^{-1}\kappa^2d_1r^2\log d_1, \allowbreak (\kl^{-1}d_1r)^{(1+\eps)/(2\eps)}\sigma_r^{-(1+\eps)/\eps}(\gamma_1\log d_1)^{1/\eps}\big\}$, then with probability over $1-c_1d_1^{-1}-c_2\log^{-1} d_1\allowbreak-c_3e^{-d_1}$, the initialization satisfies $\fro{\M_0-\M^*}\leq c_0\sigma_{r}$. Here $\kappa:=\sigma_1(\M^{\ast})\sigma_r^{-1}(\M^{\ast})$ is the condition number. 
\end{theorem}

Compared with \cite{fan2021shrinkage},  Theorem~\ref{thm:init-general-known} only requires a finite $1+\eps$ moment condition. 
We note that if the noise term $\xi_i$ has a finite $2+\eps$ moment, the probability term $O(\log^{-1} d_1)$ can be replaced by $O(n^{-\min\{1,\eps/2\}})$. Moreover, with a finite $2+\eps$ moment, one can even obtain a sub-Gaussian initial estimator by the median-of-means \citep{minsker2015geometric,lugosi2019risk,lugosi2019regularization} or the Catoni's estimator \citep{catoni2012challenging,minsker2018sub}. Since the shrinkage-based approaches \citep{fan2021shrinkage,minsker2018sub} are already effective to provide warm initializations,  we spare no additional efforts to improve Theorem~\ref{thm:init-general-known}. If the covariances $\bSigma_i$'s are unknown, we suggest to apply \cite{fan2021shrinkage} or \cite{minsker2018sub} for obtaining the initialization.

\paragraph*{Selection of initial stepsize}	Though the main theorems in Sections~\ref{sec:vector} and \ref{sec:matrix}  require the initial stepsize $\eta_0\asymp n^{-1}\|\Bbeta_0-\Bbeta^{\ast}\|_2$ and $\eta_0\asymp n^{-1}\|\M_0-\M^{\ast}\|_{\rm F}$,  respectively,  the following theorems show that similar results can be attained even if we only have an upper bound of $\|\Bbeta_0-\Bbeta^{\ast}\|_2$ and $\|\M_0-\M^{\ast}\|_{\rm F}$. 
\begin{customthm}{\ref{thm:vec:sparse}*}[alternative of Theorem~\ref{thm:vec:sparse}]
	
	Suppose Assumptions~\ref{assump:sensing operators:vec} and \ref{assump:heavy-tailed} hold. There exist absolute constants $C_1,\dots,C_6>0$, $c_2^*\in(0,1)$ such that if $n\geq C_1\ku\kl^{-1}\tilde{s}\log(2d/\tilde{s})$, $\tilde{s}\geq C_2 (\ku/\kl)^8(b_0/b_1)^8\cdot s$,  the initial stepsize $\eta_{0}\in(n\ku)^{-1}\kl^{1/2}D_0\cdot\left[1/8,\ 3/8\right]$ for any $D_0\geq \|\Bbeta_0-\Bbeta^{\ast}\|_2$,  the phase-one stepsize $\eta_l=(1-c_1)^l\eta_0$ for any $c_1\leq \kl\ku^{-1}/64$,  and the phase-two stepsize $\eta_l=\eta\in \kl^{1/2}b_1^2(nb_0\ku)^{-1}\cdot [C_3,\  C_4]$,  then with probability  exceeding $1-\exp\big(-C_5\tilde{s}\log(2d/\tilde{s})\big)-3\big(-(n\log(2d/\tilde{s})^{1/2}\log^{-1} n\big)$,  after at most $O\big(\log(D_0/\gamma)+\log(n\gamma b_0^{-1}\log(2d/\tilde{s}))\big)$ iterations,  the Algorithm~\ref{alg:IHTl1} outputs an estimator with the error rate
	$$
	\|\hat\Bbeta-\Bbeta^{\ast}\|_2^2\leq \frac{C_6\ku}{\kl^2}\cdot \frac{b_0^2\tilde{s}\log(2d/\tilde{s})}{n}.
	$$
\end{customthm}


\begin{customthm}{\ref{thm:heavytail-l1}*}[alternative of Theorem~\ref{thm:heavytail-l1}]\label{thm:heavytail-l1-alt}
	Assume $\{\xi_i\}_{i=1}^n$ and $\{{\rm vec}(\X_i)\}_{i=1}^n$ satisfy Assumptions~\ref{assump:sensing operators:vec} and \ref{assump:heavy-tailed}, respectively. There exist constants $c_0^{\ast},  c_1, c_2\in(0,1 ), C_1,  C_2>0$ such that if $n\geq C_1\ku\kl^{-1}d_1r$, 
	the initialization $\fro{\M_0-\M^*}\leq D_0$ for any $D_0\leq c_0^{\ast}\sigma_{r}$, the initial stepsize $\eta_{0}\in \kl^{1/2}(16n\ku)^{-1}D_0\cdot\left[0.2, 0.3\right]$, the phase one stepsize $\eta_{l}=(1-c_3)^{l}\eta_{0}$ with any fixed constant $c_3\leq c_1$, and the phase two stepsize $\eta_{l}=\eta\asymp \kl \ku^{-2}b_1^2(nb_0)^{-1}$, then with probability at least $1-\exp(-c_2rd_1)-3\exp(-n^{1/2}/\log n)$, after at most $O\big(\log(D_0/\gamma)+\log\big(\gamma n/(d_1rb_0)\big)\big)$ iterations, the Algorithm~\ref{alg:RsGrad} outputs an estimator with the error rate 
	$$
	\|\hat \M-\M^{\ast}\|_{\rm F}^2\leq \frac{C_2\ku}{\kl^2}\cdot \frac{b_0^2 d_1 r}{n}.
	$$
\end{customthm}

By Theorem~\ref{thm:heavytail-l1-alt},  it thus suffices to obtain an estimate of $\|\M_0-\M^{\ast}\|_{\rm F}$.  Towards that end,  one appropriate method is to take the average $n^{-1}\sum_{i=1}^n |Y_i-\langle \M_0, \X_i\rangle|$,  which,  under mild conditions,  is of the same order of $\|\M_0-\M^{\ast}\|_{\rm F}$ with high probability.  The same approach is also applicable in sparse vector regression.  The general quantities $\muc$ and $\Lc$ involved in Algorithm~\ref{alg:RsGrad} are often $\asymp n$ up to a factor of loss-related parameters,  which can be easily tuned in practice.   

\paragraph*{Tuning Parameters in RsGrad}
Our algorithm assumes knowing the rank of $\M^{\ast}$,  which is typical in non-convex methods\citep{tong2021accelerating,chen2015fast,cai2021generalized}.  One may apply the BIC or AIC for the selection of the rank \citep{cai2021generalized}.  While we believe that our algorithm still works for the over-parameterized case  \citep{zhuo2021computational,ma2022global,luo2022tensor} where the algorithmic rank $r$ is larger than the true rank of $\M^{\ast}$,  a thorough theoretical investigation is beyond the scope of this paper.  Convex methods such as \cite{elsener2018robust} and \cite{belloni2011square} need no initialization or the pre-requisite of rank.  But our primary interest is to study both the computational and statistical issues in robust estimation.  Our algorithms (i.e., the phase two stepsize) also need the noise level $\gamma$, which is not required by the convex approaches based on the square-root LASSO or the nuclear-norm penalized absolute loss \citep{elsener2018robust, belloni2011square,minsker2022robust}. Fortunately, $\gamma$ can be reliably estimated by our methods. The phase one iterations terminate until $\sum_{i=1}^{n}|Y_i-\inp{\M_{l_{1}}}{\X_i}|\asymp\sum_{i=1}^{n}|\xi_i|\asymp n\gamma$ where $l_1$ is the number of iterations in phase one.  Then the phase two stepsize can be taken as  $\eta=c\cdot f(\M_{l_{1}})/n^2$ for some tuning parameter $c>0$.  In simulation experiments,  we observe that the algorithm is very tolerant with the selection of $c$.  


\paragraph*{Determine the phase} Due to the geometrically decay of stepsizes during phase one,  after some iterations,  the stepsize will enter the level $O(\mus\Ls^{-2})$ desired by the phase two convergence.  On the other hand,  if the phase one iterations continue,  the stepsize will diminish fast and the value of objective function becomes stable.  This is indeed observed in numerical experiments.  See Figure~\ref{fig:conv-Gaussian} and Figure~\ref{fig:conv-t2} in Section~\ref{sec:simulation}.  Therefore,  once the stepsize falls below a pre-chosen small threshold,  the phase two iterations can be initiated and the stepsizes are fixed afterwards until convergence.

\section{Huber's Contamination Model}\label{sec:huber}
The robustness against outliers is often studied for the Huber's $\eps$-contamination model \citep{huber2009robust,chen2016general,chen2018robust, diakonikolas2019recent},  which assumes the data are  i.i.d. with a common distribution $d\,P=(1-\epsilon)\, d\,P_i+\epsilon\, d\,P_o$ where $P_i$ denotes the distribution of inliers,  $P_o$ is an arbitrary distribution of outliers,  and $\epsilon\in[0,1)$.  We now demonstrate the robustness of our algorithms against outliers.  Without loss of generality,  throughout this section,  we assume $\epsilon$-fraction of the responses $\{Y_i\}_{i=1}^n$ are corrupted with outliers.  Denote the indices of outliers as $\calO$ and inliers as $\calI$ where $\calO\cup\calI=\{1,\dots,n\}$ and $|\calO|=\lceil \epsilon n\rceil$.  Basically,  $(\X_i, Y_i)$ satisfies the linear model (\ref{eq:tr_model}) for $i\in\calI$.  For outliers $i\in\calO$,  the response $Y_i$ is arbitrarily corrupted.  Note that we still assume all the covariates $\X_i$'s or $vec(\X_i)$'s satisfy the Assumption~\ref{assump:sensing operators:vec}.  Interested readers are suggested to refer to \cite{depersin2020robust,lecue2019learning,chinot2020robust,thompson2020outlier,diakonikolas2019robust} for more general settings of Huber's $\epsilon$-contamination model.  Due to page limit,  we only consider the absolute loss. 

\subsection{Sparse linear regression}
Using the same notations from Section~\ref{sec:vector}, the following theorem describes the statistical guarantees of the IHT-$\ell_1$ algorithm under Huber's $\epsilon$-contamination model. 

\begin{theorem}
	Suppose Assumption~\ref{assump:sensing operators:vec} holds and $\{\xi_i\}_{i\in\calI}$ satisfy Assumption~\ref{assump:heavy-tailed}. There exist absolute constants $C_1,\dots,C_6>0$, $c_2^*\in(0,1)$ such that if $n\geq C_1\ku\kl^{-1}\tilde{s}\log(2d/\tilde{s})$, $\tilde{s}\geq C_2 (\ku/\kl)^8(b_0/b_1)^8\cdot s$, the corruption rate $\epsilon\leq \kl^{1/2}(4\kl^{1/2}+\ku^{1/2})^{-1}$, the initial stepsize $\eta_{0}\in(n\ku)^{-1}\kl^{1/2}\cdot\allowbreak\ltwo{\Bbeta_0-\Bbeta^*}\cdot\left[1/8,\ 3/8\right]$, the phase one stepsize $\eta_l=(1-c_1)^l\eta_0$ with any fixed constant $c_1\leq \kl\ku^{-1}/64$, and the phase two stepsize $\eta_{l}=\eta\in \kl^{1/2}b_1^2(nb_0\ku)^{-1}\cdot\left[C_5, C_6\right]$, then with probability over $1-2\exp(-C_3\tilde{s}\log (2d/\tilde{s}))-3\exp(-(n\log(2d/\tilde{s}))^{1/2}\log^{-1} n)$, after at most  $O\big(\log(\ltwo{\Bbeta_0-\Bbeta^*}/\gamma)+\log(n\gamma b_0^{-1}\log(2d/\tilde{s}))\wedge\log(\gamma/b_0\epsilon)\big)$ iterations, Algorithm~\ref{alg:IHTl1} outputs an estimator with the error rate 
	$$
	\|\hat\Bbeta-\Bbeta^*\|_2\leq \frac{C_4\ku^{1/2}}{\kl}  \cdot b_0\cdot\max\left\{\Big(\frac{\tilde{s}\log(2d/\tilde{s})}{n}\Big)^{1/2},\epsilon\right\}.
	$$ 
	%
	\label{thm:vec:sparse:outlier}
\end{theorem}

Our theorem allows $\epsilon$ to be any value smaller than $\kl^{1/2}(4\kl^{1/2}+\ku^{1/2})^{-1}$.
When the noise of inliers is Gaussian, the error rate matches with the minimax lower bound \citep{chen2016general,chen2018robust}. The rates derived by  \cite{chinot2020robust,depersin2020robust} have an $O(\epsilon^{1/2})$ term and the rate by \cite{thompson2020outlier} has an $O(\epsilon\log(1/\epsilon))$ term. 

\subsection{Low-rank linear regression} Using the same notations from Section~\ref{sec:matrix}, the following theorem describes the statistical guarantees of the RsGrad algorithm under Huber's $\epsilon$-contamination model.

\begin{theorem}\label{thm:heavytail-l1:outlier}
	Suppose that $\{{\rm vec}(\X_i)\}_{i=1}^n$ and  $\{\xi_i\}_{i\in\calI}$ satisfy Assumptions~\ref{assump:sensing operators:vec} and \ref{assump:heavy-tailed}, respectively. There exist constants $c_0^{\ast},  c_1, c_2^{\ast}\in(0,1 ), C_1^{\ast},  C_2, C_3>0$ such that if the sample size $n\geq C_3\ku\kl^{-1}d_1r$,  the corruption rate $\epsilon\leq \kl^{1/2}(4\kl^{1/2}+\ku^{1/2})^{-1}$, the initialization $\fro{\M_0-\M^*}\leq c_0^{\ast}\sigma_{r}$, the initial stepsize $\eta_{0}\in \kl^{1/2}(16n\ku)^{-1}\fro{\M_0-\M^*}\cdot\left[0.2, 0.3\right]$, the phase one stepsize $\eta_{l}=(1-c_3)^{l}\eta_{0}$ with any fixed constant $c_3\leq c_1$, and the phase two stepsize $\eta_{l}=\eta\asymp \kl \ku^{-2}b_1^2(nb_0)^{-1}$, then with probability at least $1-\exp(-c_1rd_1)-3\exp(-\sqrt{n}/\log n)$, after at most $O\big(\log(\sigma_r/\gamma)+C_1^{\ast}\log\big(\gamma n/(d_1rb_0)\big)\wedge\log(\gamma/b_0\epsilon)\big)$ iterations, the Algorithm~\ref{alg:RsGrad} outputs an estimator with the error rate 
	$$
	\|\hat \M-\M^{\ast}\|_{\rm F}^2\leq \frac{C_2\ku}{\kl^2}\cdot b_0^2\cdot\max\left\{\frac{ d_1 r}{n},\epsilon^2\right\}.
	$$
\end{theorem}

Exact recovery under sparse corruptions has been studied by \cite{li2020nonconvex}. \cite{tong2021low} derived sub-optimal rates under sub-Gaussian noise and sparse corruptions. Our method can deal with heavy-tailed noise and is outlier-robust and minimax optimal. 


\section{Numerical Experiments}\label{sec:simulation}
This section presents the numerical simulation results to showcase the two-phase convergence phenomena and compare the attained accuracy with the existing algorithms. Due to page limit, some results are relegated to the supplement.

\subsection{Sparse linear regression}

Our algorithm is referred to as the  IHT-$\ell_1$. We compare with  with the {\it adaptive Huber} method from \cite{sun2020adaptive} and the  {\it truncation} method from \cite{fan2021shrinkage}.  The algorithmic parameters of adaptive Huber and truncation methods are fine-tuned for their best performances. Note that our theoretical analysis of IHT-$\ell_1$ can only treat the case $\tilde s>s$ because of technical difficulties. The simulation results show that ITH-$\ell_1$ still performs well when setting $\tilde s =s$. The two-phase convergence dynamic of ITH-$\ell_1$ is similar to that of the low-rank linear regression and is  thus skipped. 


\paragraph*{Accuracy} We set $\Bbeta^*=(16,4,1,0,\dots,0)^{\top}$ where $d=50$ and $s=3$.  The covariate vectors $\{\X_i\}_{i=1}^n$ have i.i.d. $N(0,1)$ entries. 
The sample sizes are chosen from $n\in\{50, 300\}$. The noise can be Gaussian or have a Student's t distribution with d.f. $\nu=2$.  Figure~\ref{fig:vec:accu-t2} presents the box-plot of $\|\hat\Bbeta-\Bbeta^{\ast}\|_2$ based on 50 simulations. 
\begin{figure}[t]
	\centering
	\begin{subfigure}[b]{0.45\textwidth}
		\centering
		\includegraphics[width=\textwidth]{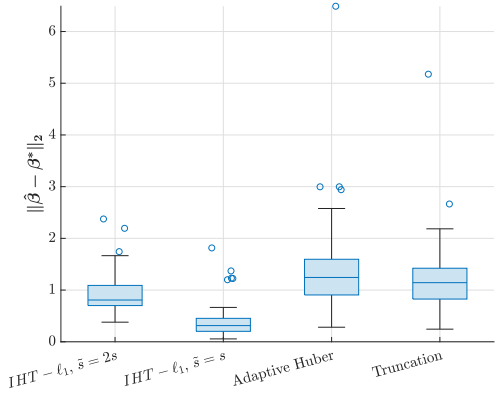}
		\caption{Student's t noise with d.f. $\nu=2$, $n=50$}
		\label{fig:111}
	\end{subfigure}
	\hfill
	\begin{subfigure}[b]{0.45\textwidth}
		\centering
		\includegraphics[width=\textwidth]{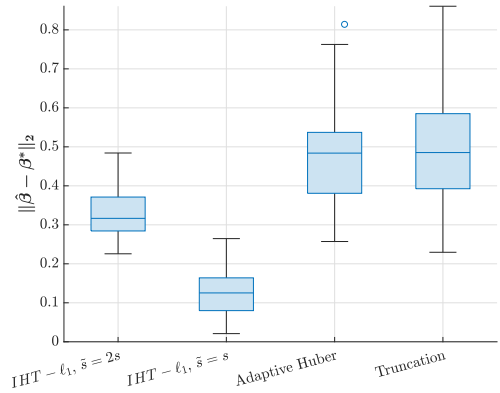}
		\caption{Student's t noise with d.f. $\nu=2$, $n=300$}
		\label{fig:112}
	\end{subfigure}
	
	\begin{subfigure}[b]{0.45\textwidth}
		\centering
		\includegraphics[width=\textwidth]{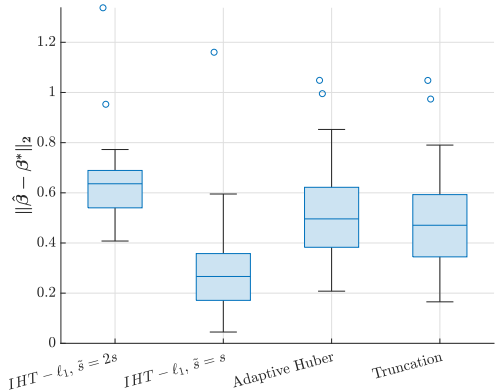}
		\caption{Gaussian noise, $n=50$}
		\label{fig:121}
	\end{subfigure}
	\hfill
	\begin{subfigure}[b]{0.45\textwidth}
		\centering
		\includegraphics[width=\textwidth]{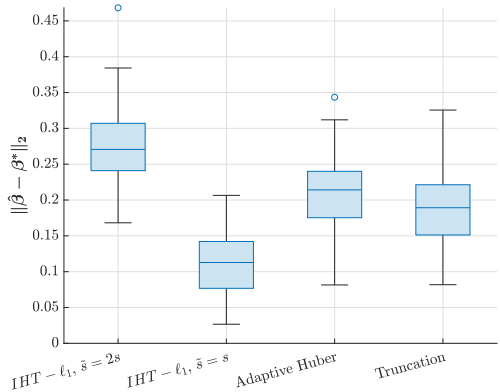}
		\caption{Gaussian noise, $n=300$}
		\label{fig:122}
	\end{subfigure}
	
	\caption{Accuracy comparisons of IHT-$\ell_1$, the adaptive Huber \citep{sun2020adaptive} and truncation method \citep{fan2021shrinkage} for sparse linear regression.} 
	\label{fig:vec:accu-t2}
\end{figure}

\paragraph*{Real data}
We apply IHT-$\ell_1$ to the NCI-60 cancer cell line dataset. Note that since the dataset has been updated recently,  our experiment results may slightly differ from \cite{sun2020adaptive}. The dataset has two parts:  \textit{Protein: Antibody Array DBT} and \textit{RNA: Affy HG-U133(A-B) - GCRMA}, both of which may be found in website \url{https://discover.nci.nih.gov/cellminer/loadDownload.do}. The first part \textit{Protein: Antibody Array DBT} contains the levels of $162$ antibodies in $60$ cell lines \citep{nishizuka2003proteomic}. We focus on the antibody ``KRT19", which is the most used marker for the tumor cells disseminated in lymph nodes, peripheral blood, and bone marrow of breast cancer
patients \citep{nakata2004serum}. The second part \textit{RNA: Affy HG-U133(A-B) - GCRMA} reflects the gene transcripts in $60$ cell lines \citep{shankavaram2007transcript}. Our goal is to select the significant genes in generating the protein ``KRT19". By combining overlapped genes and multiple tested genes and removing one missing data cell line, we obtain a design matrix $\X$ with $n=59$ and $d=14,342$. See \cite{sun2020adaptive,shankavaram2007transcript} for more details about the data set.

\begin{table}
	\begin{center}
		\begin{tabular}{|m{7em}|m{3em}|m{3em}|m{20em}|}
			\hline
			Method&MAE&Size&Selected Genes\\
			\hline
			Adaptive Huber&2.13&11&MALL, TM4SF4, ANXA3, ADRB2, NRN1, AUTS2, CA2, BAMBI, EPS8L2, CEMIP, KRT19\\
			\hline
			Truncation&2.01&10&MALL, ANXA3, ADRB2, NRN1, CA2, BAMBI, COCH, CEMIP, KRT19, EMP3\\
			\hline
			IHT-$\ell_1$ ($\tilde s = 7$)&1.91&7& MALL, TGFBI, NRN1, DSP, EPS8L2, KRT19, LPXN\\
			\hline
			IHT-$\ell_1$ ($\tilde s = 12$)&2.12&12&S100A6, EPCAM, MALL, ZIC1, TGFBI, NRN1, DSP, DKK1, EPS8L2, MYO5C, KRT19, LPXN\\
			\hline
		\end{tabular}
	\end{center}
	\caption{Prediction performance on NCI-60 data. The mean absolute error of test data, final model size and selected genes are reported for three methods: IHT-$\ell_1$, the adaptive Huber \citep{sun2020adaptive} and truncation method \citep{fan2021shrinkage}.}
	\label{table:realdata}
\end{table}

We test the prediction performance and report the selected genes.We split the data into a training set and a test set.  The design matrix is $\log_2$ transformed and later centered.  The algorithmic parameters of adaptive Huber, truncation method, and IHT-$\ell_1$ are fine-tuned via 10-fold cross-validations. The mean absolute error on test data, the final model size, and the selected genes are reported in Table~\ref{table:realdata}. Note that the selected genes are obtained by applying the algorithms to the whole dataset. Here we report the performance of IHT-$\ell_1$ only for the two cases $\tilde{s}=7$ and $\tilde{s}=12$. Table~\ref{table:realdata} shows that the gene ``KRT19'' is significant, which has been selected by all the algorithms.  The genes ``NRN1'' and ``EPS8L2'' are selected by all three algorithms, whose effect to breast cancer has been verified by \cite{kim2011differential,colas2011molecular}. The
gene ``MALL'' is also selected by all algorithms, which is known to be related with some other cancers so far. The genes ``TGFBI'', ``DSP'', ``DKK1'', ``S100A6'', ``EPCAM'', ``ZIC1'', ``MYO5C'', ``LPXN'' selected by IHT-$\ell_1$ are also important to breast cancer \citep{li2012role,yang2012desmoplakin,niu2019dkk1,zhang2017distinct,osta2004epcam,han2018zic1,ye2022e2f1,abe2016etv6}. Besides, we note that the genes ``ANXA3'', ``ADRB2'', ``CA2'', ``MAMBI'', ``CEMIP'' selected by the adaptive Huber and truncation methods are significant by existing literature  \citep{zhou2017silencing,xie2019beta,annan2019carbonic,yu2017inhibition,xue2022cemip}.


\subsection{Low-rank linear regression} We now showcase the two-phase convergence of RsGrad and compare the achieved accuracy with the {\it convex-huber}, {\it convex-$\ell_1$} \citep{elsener2018robust}, {\it convex-$\ell_2$} \citep{candes2011tight}, and the (non-convex) ScaledSM methods \citep{tong2021low}. The accuracy comparison is deferred to the supplement. For ScaledSM and the phase one iterations of RsGrad, the stepsizes are set as $\eta_0 = c_1 \|\M_0\|n^{-1}$ and $\eta_l=q\cdot \eta_{l-1}$ where $c_1>0$, $0<q<1$ are parameters to be tuned and $\|\cdot\|$ denotes the matrix operator norm. As \cite{tong2021low} suggested,  we choose $q=0.91$.  For the phase two iterations of RsGrad,  we set the constant stepsize at $\eta_l=c_2\EE|\xi| n^{-1}$ where $c_2>0$ is a tuning parameter. Here we assume $\EE|\xi|$ is known for simplicity.  In practice,  $\EE|\xi|$  can be estimated by $n^{-1}\sum_{i=1}^n|Y_i-\langle \M_{l_1},\X_i\rangle|$ where $\M_{l_1}$ is the output after phase one iterations.  Following \cite{tong2021low}, denote the signal-to-noise ratio $\text{SNR}:=20\log_{10}\big(\fro{\M^*}/ \EE\vert\xi\vert\big)$. The relative error is calculated by $\|\M_l-\M^{\ast}\|_{\rm F}\|\M^{\ast}\|_{\rm F}^{-1}$. 

\paragraph*{Two-phase convergence}
We compare the convergence dynamics between RsGrad and ScaledSM.  We experiment RsGrad with the absolute loss ({\it RsGrad-$\ell_1$}) and Huber loss ({\it RsGrad-Huber}), but ScaledSM was proposed only for the absolute loss.  The dimensions and ranks are fixed at $d_1=d_2=80$ and rank $r=5$, respectively. The phase one iterations are terminated when the stepsize becomes rather small, e.g., smaller than $10^{-10}$, and then the phase two iterations are initiated. 

The first set of simulations is to showcase the convergence performance under Gaussian noise.  The SNR is set at $\{40,  80\}$ and the sample sizes are varied among $\{2.5,  5\}\times rd_1$.  The convergence dynamics are displayed in Figure~\ref{fig:conv-Gaussian},  which clearly show a two-phase convergence of RsGrad.  We note that ScaledSM performs poorly when sample size is small ($n=2.5rd_1$).  It is possibly due to the instability of inverse scaling under a small sample size.  The bottom two plots in Figure~\ref{fig:conv-Gaussian}  show that,  after the phase one iterations,  RsGrad achieves a similar perfomrance as ScaledSM.  However,  with a second phase convergence,  RsGrad (using either Huber or absolute loss) eventually delivers a more accurate estimate than ScaledSM.

\begin{figure}[t]
	\centering
	\begin{subfigure}[b]{0.45\textwidth}
		\centering
		\includegraphics[width=\textwidth]{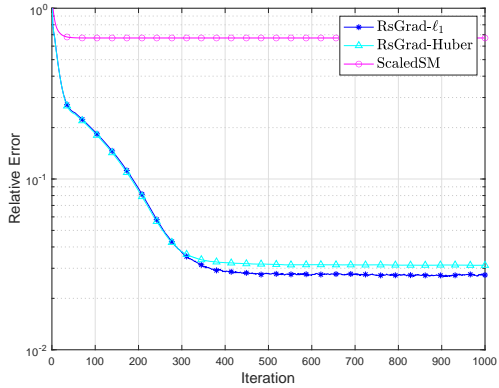}
		\caption{SNR=40}
		\label{fig:11}
	\end{subfigure}
	\hfill
	\begin{subfigure}[b]{0.45\textwidth}
		\centering
		\includegraphics[width=\textwidth]{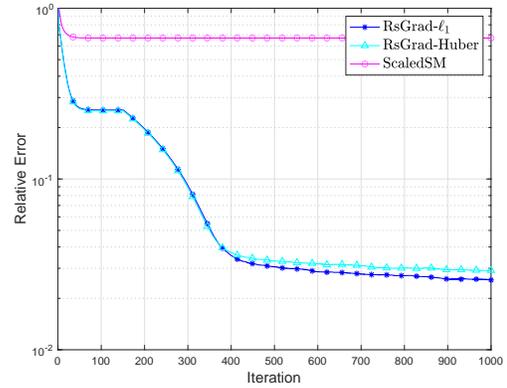}
		\caption{SNR=80}
		\label{fig:12}
	\end{subfigure}

	\begin{subfigure}[b]{0.45\textwidth}
		\centering
		\includegraphics[width=\textwidth]{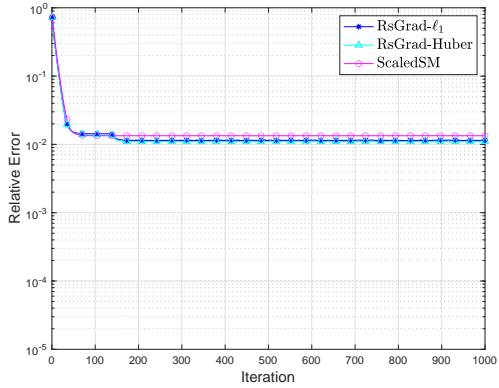}
		\caption{SNR=40}
		\label{fig:31}
	\end{subfigure}
	\hfill
	\begin{subfigure}[b]{0.45\textwidth}
		\centering
		\includegraphics[width=\textwidth]{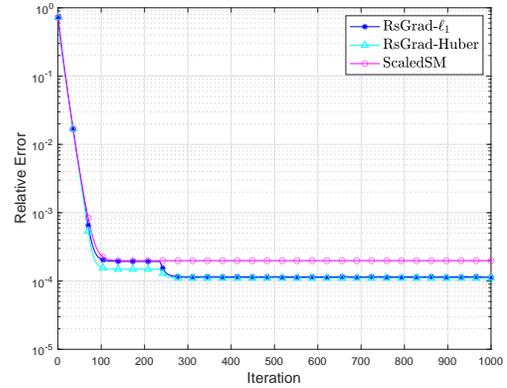}
		\caption{SNR=80}
		\label{fig:32}
	\end{subfigure}
	
	\caption{Convergence dynamics of RsGrad and ScaledSM \citep{tong2021low} under Gaussian noise.  Top two figures:  $n=2.5rd_1$;  bottom two figures: $n=5rd_1$.}
	\label{fig:conv-Gaussian}
\end{figure}

The second set of simulation is to test the convergence performance under heavy-tailed noise.  For simplicity,  the noise is sampled independently from a Student's $t$-distribution with ${\rm d.f.}$ $\nu=2$,  which has a finite first (absolute) moment and an infinite second moment.  Other parameters are selected the same as in the Gaussian case.   Figure~\ref{fig:conv-t2} presents the convergence performance showing that both RsGrad (either absolute loss or Huber loss)  and ScaledSM are robust to heavy-tailed noise.  
Similarly,  we can observe the two-phase convergence of RsGrad whose second phase iterations deliver a more accurate estimate than ScaledSM,  especially when SNR is large.  

\begin{figure}[t]
	\centering
	\begin{subfigure}[b]{0.45\textwidth}
		\centering
		\includegraphics[width=\textwidth]{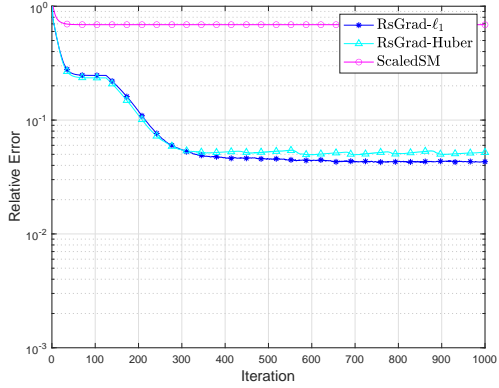}
		\caption{SNR=40}
		\label{fig:21}
	\end{subfigure}
	\hfill
	\begin{subfigure}[b]{0.45\textwidth}
		\centering
		\includegraphics[width=\textwidth]{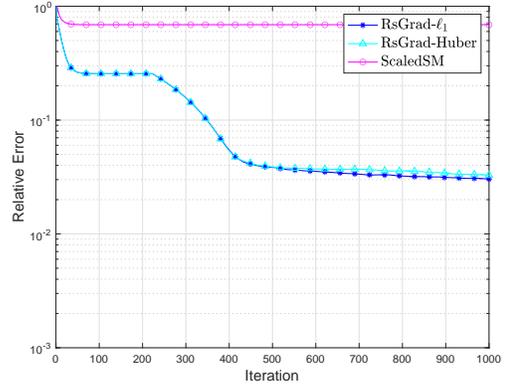}
		\caption{SNR=80}
		\label{fig:22}
	\end{subfigure}
	
	\begin{subfigure}[b]{0.45\textwidth}
		\centering
		\includegraphics[width=\textwidth]{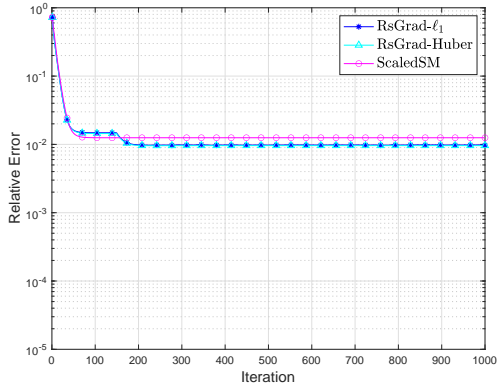}
		\caption{SNR=40}
		\label{fig:41}
	\end{subfigure}
	\hfill
	\begin{subfigure}[b]{0.45\textwidth}
		\centering
		\includegraphics[width=\textwidth]{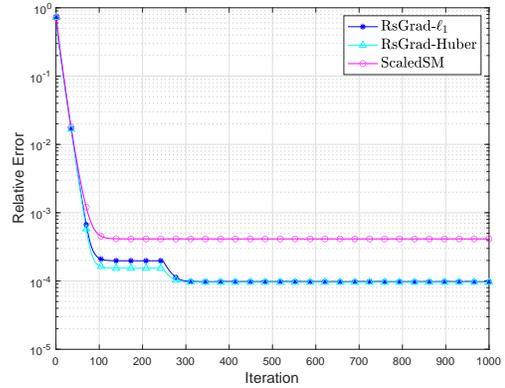}
		\caption{SNR=80}
		\label{fig:42}
	\end{subfigure}

	\caption{Convergence dynamics of RsGrad and ScaledSM \citep{tong2021low} under Student's $t$-distribution with d.f.  $\nu=2$.  Top two figures:  $n=2.5rd_1$;  bottom two figures: $n=5rd_1$.}
	\label{fig:conv-t2}
\end{figure}


\newpage

\bibliographystyle{plainnat}
\bibliography{reference}       

\appendix
\newpage

\begin{center}
	{\LARGE Appendix of ``Computationally Efficient and Statistically Optimal Robust High-Dimensional Linear Regression''}
	
	\bigskip
	
	{Yinan Shen, Jingyang Li, Jian-Feng Cai and Dong Xia\\
		{\small Hong Kong University of Science and Technology}}
	
\end{center}

	Section~\ref{sec:add-result} presents materials that are not shown in the main text due to limited length, including Riemannian optimization details and simulation results. In Section \ref{proof:tec:vec}, we present proofs of sparse vector recovery. Section~\ref{proof:prop:main} proves Proposition~\ref{prop:main}, convergence dynamics of low-rank regression under two-phase regularity conditions. Section~\ref{sec:proof-app} contains four sub-sections, \ref{proof:lem:Gaussian-l1}-\ref{proof:lem:heavytail-quantile}, which verify the regularity conditions in certain different cases, absolute loss, Huber loss and quantile loss. The initialization analysis is provided in Section~\ref{proof:initialization}. Then Section~\ref{proof:teclem} exhibits the technical lemmas. Finally, lengthy proofs of two significant lemmas \ref{teclem:perturbation} matrix perturbation theories and Proposition \ref{thm:empirical process} empirical processes for low-rank regression are put in Section~\ref{proof:teclem:perturbation} and Section~\ref{proof:thm:empirical} respectively. In the last section~\ref{proof:outlier}, we analyze outlier robustness. 

The following key ingredients act as a guidance for the proof of low-rank linear regression.
\begin{enumerate}[1.]
	\item \textit{Tangent Space over $\MM_r$.}  In the algorithm, it updates in the way of $\T_{l+1}=\text{SVD}_r(\T_l-\eta_{l}\calP_{\TT_l}(\G_l))$. Then in the convergence analyses, it involves the upper bound of $\fro{\calP_{\TT_l}(\G_l)}^2$ and $\left|\inp{\calP_{\TT_l}^{\perp}\G_l}{\T_l-\T^*}\right|$. As for the first term, note that \begin{align*}
		\calP_{\TT_l}(\G_l)=\U_l\U_l^{\top}\G_l+\G_l\V_l\V_l^{\top}-\U_l\U_l^{\top}\G_l\V_l\V_l^{\top}=\U_l\U_l^{\top}\G_l+(\I-\U_l\U_l^{\top})\G_l\V_l\V_l^{\top}
	\end{align*}
	could be written into sum of two rank at most $r$ matrices. Then we use `Partial Frobenius Norm' (see \cite{tong2021accelerating} or Lemma~\ref{teclem:partial F norm}) to bound $\fro{\calP_{\TT_l}\G_l}$; see Lemma~\ref{teclem:partial F norm} and Lemma~\ref{teclem:projected subgradient norm} for more details. As for the second term, we use
	\begin{align*}
		\left|\inp{\calP_{\TT_l}^{\perp}\G_l}{\T_l-\T^*}\right|=\left|\inp{\G_l}{\calP_{\TT_l}^{\perp}(\T_l-\T^*)}\right|,
	\end{align*}
	and use $\calP_{\TT_l}^{\perp}(\T_l-\T^*)=-(\I-\U_l\U_l^{\top})\T^*(\V\V^{\top}-\V_l\V_l^{\top})=(\I-\U_l\U_l^{\top})(\T_l-\T^*)(\V\V^{\top}-\V_l\V_l^{\top})$. Hence, with definition of partial Frobenius norm Lemma~\ref{teclem:partial F norm}, one has
	\begin{align*}
		\left|\inp{\calP_{\TT_l}^{\perp}\G_l}{\T_l-\T^*}\right|\leq\fro{(\I-\U_l\U_l^{\top})(\T_l-\T^*)(\V\V^{\top}-\V_l\V_l^{\top}) }\fror{\G_l}.
	\end{align*}
	\item \textit{Perturbation lemma.} Based on Theorem 1 in work \cite{xia2021normal}, we develop powerful analysis scheme for matrix perturbations, where we have $\fro{\T_{l+1}-\T^*}\leq \fro{\T_l-\T^*-\eta_{l}\calP_{\TT_l}(\G_l)}+O(\fro{\T_l-\T^*-\eta_{l}\calP_{\TT_l}(\G_l)}^2)$. For clarity, we present Theorem 1 in work \cite{xia2021normal} as Theorem~\ref{tecthm:pertbation-expansion}.  Lemma~\ref{teclem:symmetric perturbation} presents symmetric matrix perturbations, based on which we have general matrix perturbations Lemma~\ref{teclem:perturbation}.
	\item \textit{Empirical Processes.} We use emprical process theories \citep{ludoux1991probability,van1996weak,adamczak2008tail} to prove the regularity properties. One of the merits is that it enables strong exponential probability.
	\item \textit{Initialization Guarantees.} Note that the perturbation bound and convergence dynamics Equation~\eqref{eq:prop1-eq1} require $\fro{\M_0-\M^*}\lesssim \sigma_r$. Thus we start with a warm initialization. A clean and simple spectral initialization scheme is provided. Section~\ref{proof:initialization} shows the proofs.
	\item \textit{Proposition 1.} Proposition~\ref{prop:main} analyzes the general convergence under two-phase regularity property assumption and depicts the stepsize selections. Lemma~\ref{lem:Gaussian-l1}, \ref{lem:heavytail-l1}, \ref{lem:heavytail-huber} and \ref{lem:heavytail-quantile} prove two-phase regularity properties under several settings. Namely, the loss function behaves in the  typical non-smooth way when it is far from the oracle matrix $\M^*$, while when it gets sufficiently close to $\M^*$, the loss function performs alike smooth $\ell_2$ loss. 
\end{enumerate}

\section{Additional Results}\label{sec:add-result}
This section collects some results which are not presented in the main context due to the page limit. 

\subsection{Computational benefit of Riemannian sub-gradient algorithm} The Riemannian sub-gradient is fast computable.  Let $\M_l=\U_l\bSigma_l\V_l^{\top}$ be the thin SVD of $\M_l$.  It is well-known \citep{absil2009optimization,  vandereycken2013low} that the tangent space $\TT_l$ can be characterized by $\TT_l:=\{\Z\in\RR^{d_1\times d_2}: \Z=\U_l\R^{\top}+\L\V_l^{\top},  \R\in\RR^{d_2\times r},  \L\in\RR^{d_1\times r}\}$.  Then projection of $\G_l$ onto $\TT_l$ is
$$
\calP_{\TT_l}(\G_l)=\U_l\U_l^{\top}\G_l+\G_l\V_l\V_l^{\top}-\U_l\U_l^{\top}\G_l\V_l\V_l^{\top},
$$
which is of rank at most $2r$. The first step in Algorithm~\ref{alg:RsGrad} could be written as
\begin{align*}
	&{~~~~}\M_l-\eta_{l}\calP_{\TT_l}(\G_l)\\
	&=\U_l\left(\bSigma_l-\eta_{l}\U_l^{\top}\G_l\V_l\right)\V_l^{\top}-\eta_l\U_l\U_l^{\top}\G_l\left(\I_{d_2}-\V_l\V_l^{\top}\right)-\eta_{l}\left(\I_{d_1}-\U_l\U_l^{\top}\right)\G_l\V_l\V_l^{\top}.
\end{align*} Let $\left(\I_{d_2}-\V_l\V_l^{\top}\right)\G_l^{\top}\U_l=\Q_1\R_1$ and $\left(\I_{d_1}-\U_l\U_l^{\top}\right)\G_l\V_l=\Q_2\R_2 $ be the QR decomposition. Then the update is
\begin{align*}
	\M_l-\eta_{l}\calP_{\TT_l}(\G_l)=\left[\U_l\; \Q_2\right]\left[\begin{matrix}
		\bSigma_l-\eta_{l}\U_l^{\top}\G_l\V_l& -\eta_{l}\R_1^{\top}\\
		-\eta_{l}\R_2&\boldsymbol{0}
	\end{matrix}\right] \left[\V_l\; \Q_1\right]^{\top},
\end{align*} with $ \left[\U_l\; \Q_2\right]$, $\left[\V_l\; \Q_1\right]$ being orthogonal matrices. Consequently,  the final step of retraction only requires the SVD of the medium term, a $2r\times 2r$ matrix, which can be computed using $O(r^3)$ floating point operations.  See,  e.g.,  \cite{vandereycken2013low} and \cite{mishra2014fixed} for more details.

\subsection{More simulation results of low-rank linear regression}

\paragraph*{Accuracy} We now compare the statistical accuracy of the final estimator output by RsGrad,  ScaledSM and those by convex approaches including the  nuclear norm penalized \citep{elsener2018robust} absolute loss (convex $\ell_1$),  Huber loss (convex Huber) and the square loss ($\ell_2$-loss) by seminal work \cite{candes2011tight}.   Functions from Matlab library cvx \citep{grant2014cvx} are borrowed to implement the convex methods for convex $\ell_1$,  convex Huber and $\ell_2$-loss.  For each setting,  every method is repeated for $10$ times and the box-plots of error rates are presented.

The parameter settings are similar as above where both the Gaussian noise and Student's $t$-distribution are experimented.  The comparison of statistical accuracy under Gaussian noise is displayed in Figure~\ref{fig:accu-Gaussian}. The top two plots suggest that all the convex methods and ScaledSM performs poorly when sample size is small ($n=2.5rd_1$).  On the other hand,  ScaledSM becomes much better when sample size is large but still under-performs RsGrad,  as predicted by our theory.   Similar results are also observed when noise has a Students' $t$-distribution with d.f.  $\nu=2$.  See Figure~\ref{fig:accu-t2}.  

\begin{figure}[t]
	\centering
	\begin{subfigure}[b]{0.45\textwidth}
		\centering
		\includegraphics[width=\textwidth]{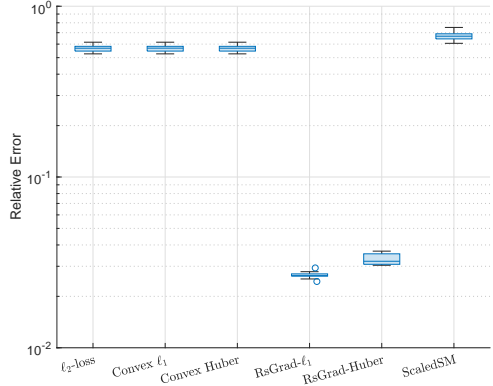}
		\caption{SNR=40}
		\label{fig:71}
	\end{subfigure}
	\hfill
	\begin{subfigure}[b]{0.45\textwidth}
		\centering
		\includegraphics[width=\textwidth]{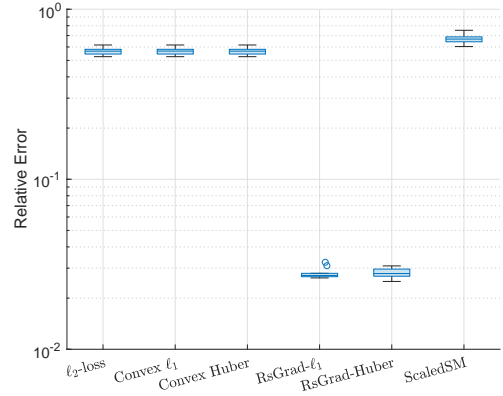}
		\caption{SNR=80}
		\label{fig:72}
	\end{subfigure}
	
	\begin{subfigure}[b]{0.45\textwidth}
		\centering
		\includegraphics[width=\textwidth]{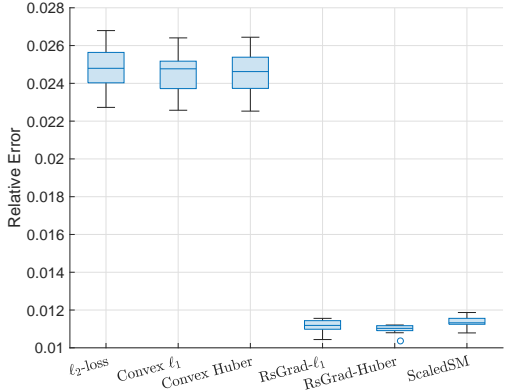}
		\caption{SNR=40}
		\label{fig:91}
	\end{subfigure}
	\hfill
	\begin{subfigure}[b]{0.45\textwidth}
		\centering
		\includegraphics[width=\textwidth]{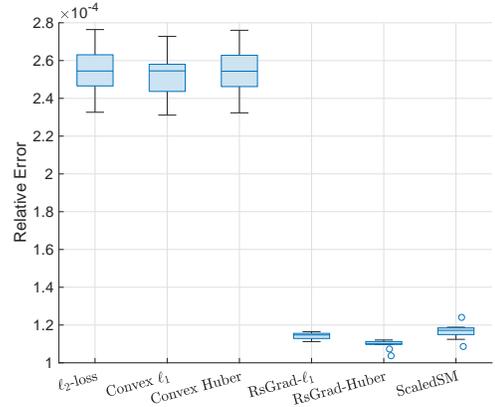}
		\caption{SNR=80}
		\label{fig:92}
	\end{subfigure}
	
	\caption{Accuracy comparisons of RsGrad, ScaledSM, robust convex models \citep{elsener2018robust} and convex $\ell_2$ loss \citep{candes2011tight} under Gaussian noise.  Top two figures:  $n=2.5rd_1$;  bottom two figures: $n=5rd_1$.}
	\label{fig:accu-Gaussian}
\end{figure}

\begin{figure}[t]
	\centering
	\begin{subfigure}[b]{0.45\textwidth}
		\centering
		\includegraphics[width=\textwidth]{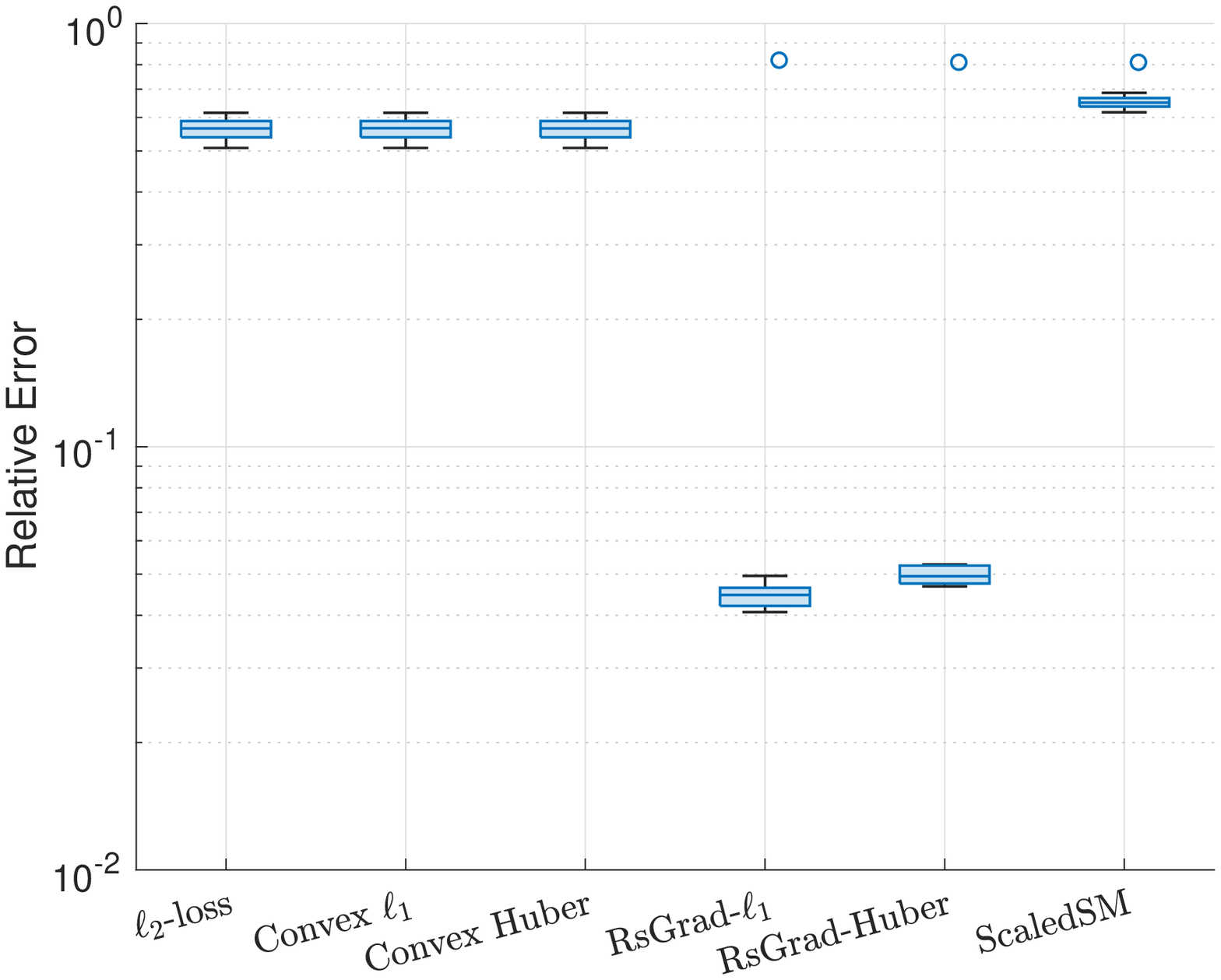}
		\caption{SNR=40}
		\label{fig:81}
	\end{subfigure}
	\hfill
	\begin{subfigure}[b]{0.45\textwidth}
		\centering
		\includegraphics[width=\textwidth]{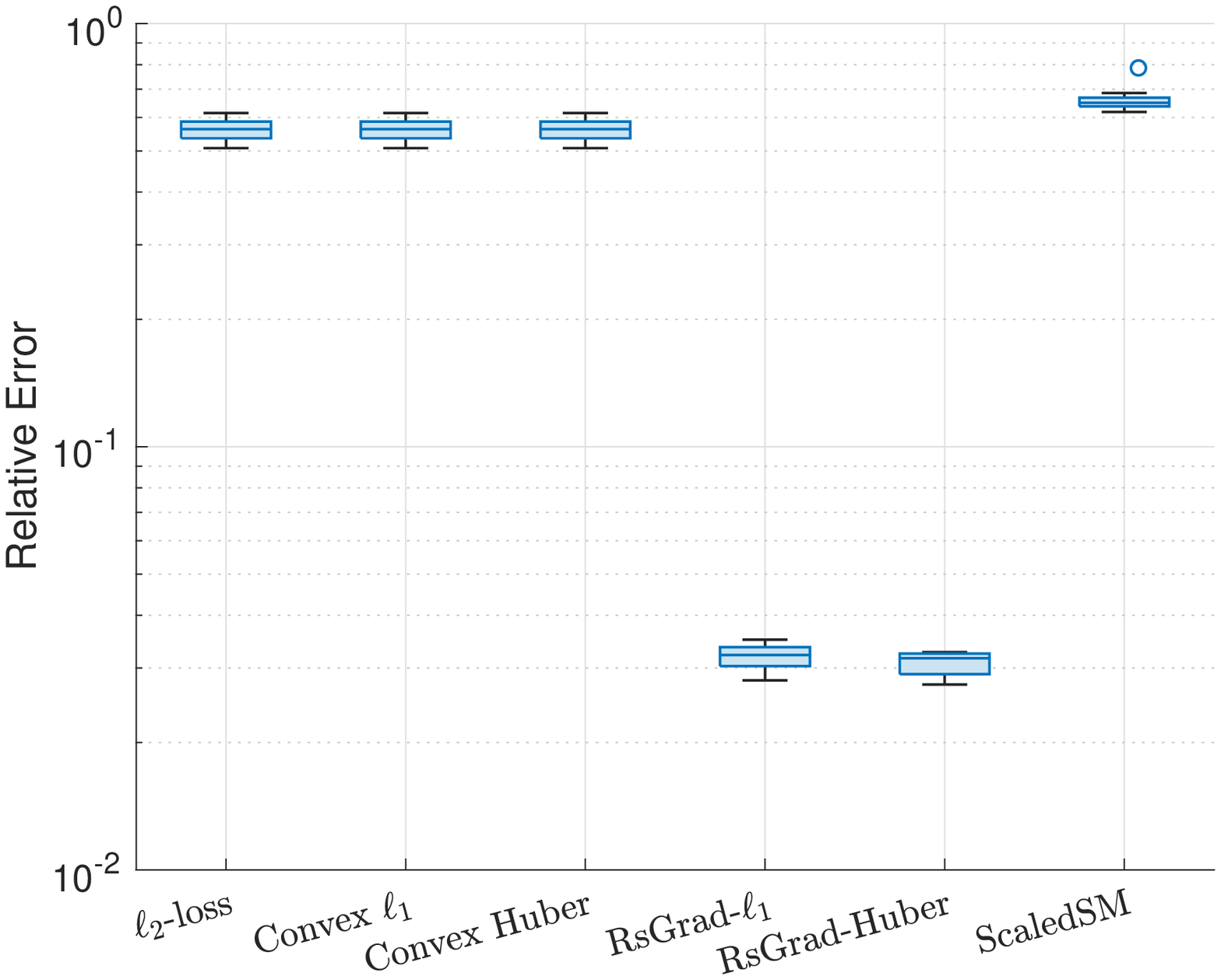}
		\caption{SNR=80}
		\label{fig:82}
	\end{subfigure}
	
	\begin{subfigure}[b]{0.45\textwidth}
		\centering
		\includegraphics[width=\textwidth]{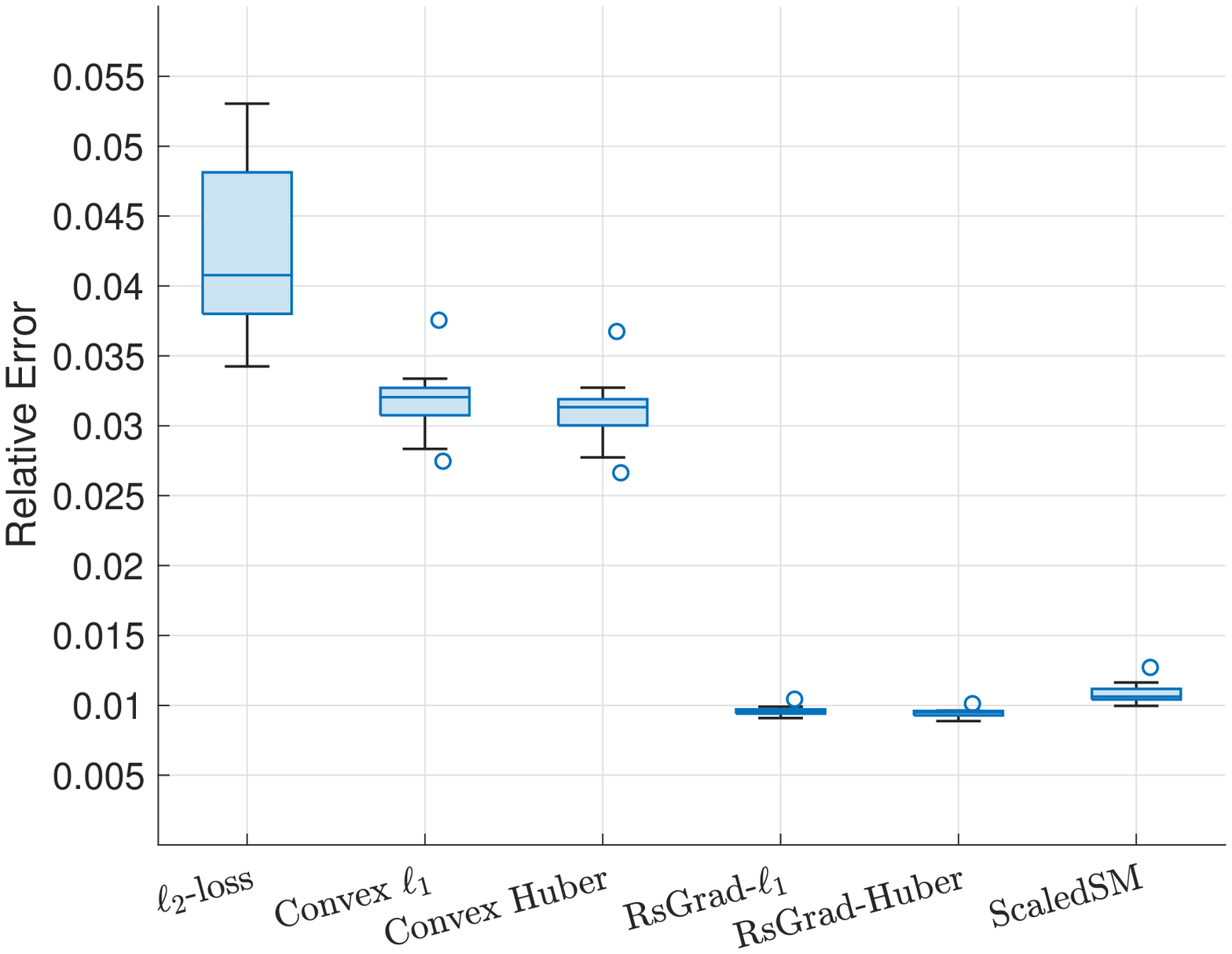}
		\caption{SNR=40}
		\label{fig:101}
	\end{subfigure}
	\hfill
	\begin{subfigure}[b]{0.45\textwidth}
		\centering
		\includegraphics[width=\textwidth]{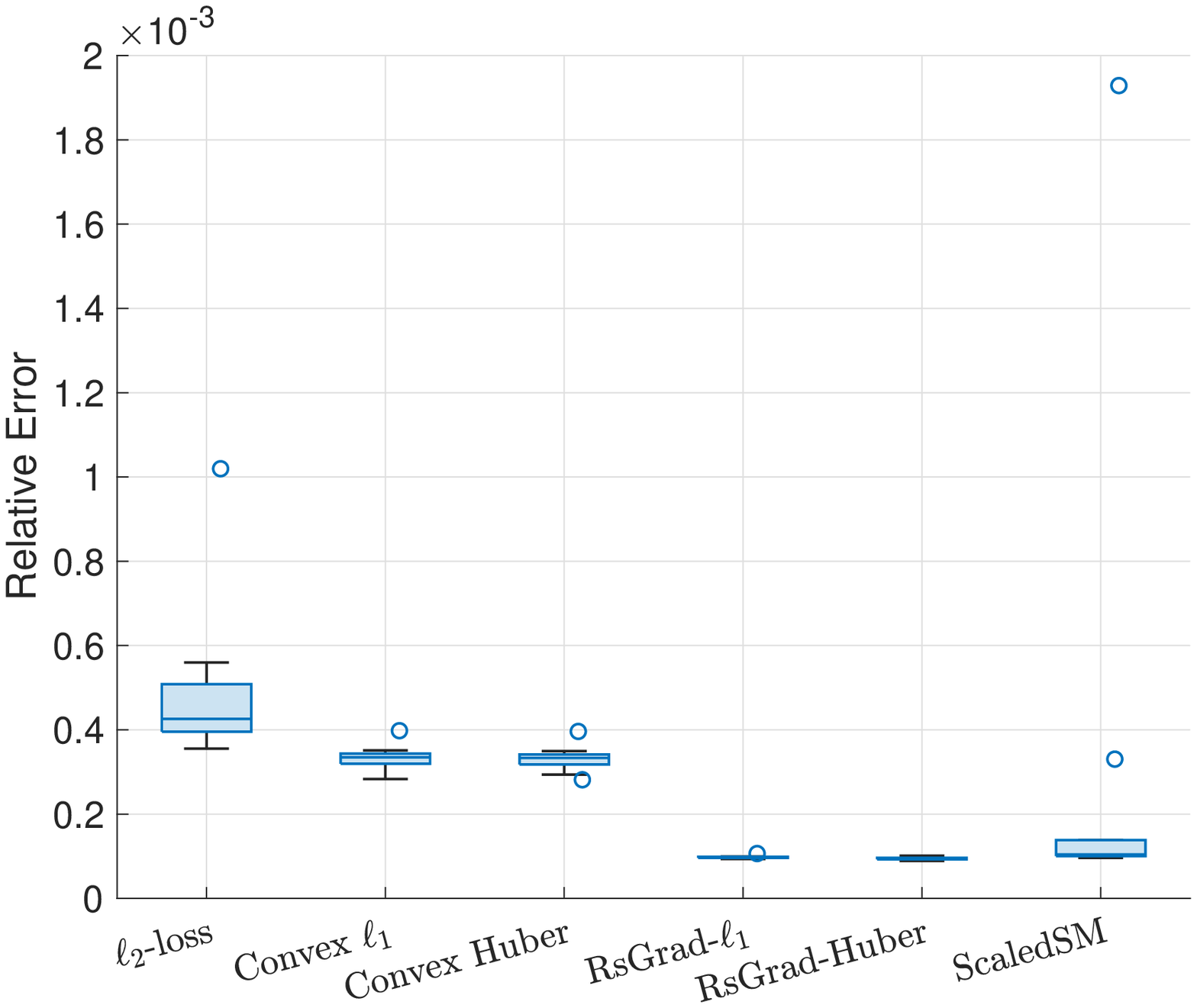}
		\caption{SNR=80}
		\label{fig:102}
	\end{subfigure}
	
	\caption{Accuracy comparisons of RsGrad, ScaledSM, robust convex models \citep{elsener2018robust} and convex $\ell_2$ loss \citep{candes2011tight} under Students' $t$-distribution noise with d.f.  $\nu=2$.   Top two figures:  $n=2.5rd_1$;  bottom two figures: $n=5rd_1$.}
	\label{fig:accu-t2}
\end{figure}

From Figure~\ref{fig:accu-Gaussian} and Figure~\ref{fig:accu-t2},  we observe that non-convex models achieve better accuracy, which may be attributed to more finely tuning parameters than convex approaches. Convex ones we compare with are also proven statistically optimal in theory and are computationally more demanding.  For instance,  in the case $n=5rd_1$,  convex $\ell_1$, convex Huber, convex $\ell_2$ need more than 1.5h, 1.5h, 3h to complete the ten repetitions,  respectively.  In sharp contrast,  RsGrad and ScaledSM only take around ten minutes.  When it comes to an even larger sample size like $n=10d_1r$,  the Matlab package on our computing platform needs 3h and 6h to complete the simulations (10 repetitions) for convex $\ell_1$ and convex Huber,  respectively,  whereas RsGrad and ScaledSM each takes about 20 minutes to do the job.

\section{Proofs of Sparse Linear Regression}
\label{proof:tec:vec}

In general, there are three key technical points, Proposition~\ref{prop:emp:vec}, Theorem~\ref{tecthm:vec-perturb} and Lemma~\ref{teclem:vec:l1exp}; see below.
\begin{proposition}
	\label{prop:emp:vec}
	Suppose function $f(\cdot)$ is given in Equation~\eqref{eq:vec-absloss:sparse}. Then there exist absolute constants $C,C_1>0$ such that
	\begin{align*}
		\left|f(\Bbeta+\Delta\Bbeta)-f(\Bbeta)-\EE\left[f(\Bbeta+\Delta\Bbeta)-f(\Bbeta) \right]\right|\leq C\sqrt{ns\ku\log(2d/s)}\ltwo{\Delta\Bbeta}
	\end{align*}
	holds for all $\Bbeta\in\RR^{d}$ and $\Delta\Bbeta\in\left\{\b\in\RR^{d}:\; \left|\text{supp}(\b)\right|\leq s\right\}$ with probability exceeding $1-\exp(-\frac{C_1^2s\log(2d/s)}{3})-3\exp\left(-\frac{\sqrt{n\log(2d/s)}}{\log n}\right)$.
\end{proposition}
Proposition~\ref{prop:emp:vec} bounds deviation of $f(\Bbeta+\Delta\Bbeta)-f(\Bbeta)$. Its proof is based on empirical process theories and is presented later in this section. Besides, we also need perturbation bound. Rough perturbation bound $\ltwo{\calH_{K}(\b)-\x}\leq 2\ltwo{\b-\x}$ is usually used in $\ell_2$ loss \citep{blumensath2009iterative,blumensath2010normalized} and is good enough to derive statiscal optimality in $\ell_2$ loss. However, the coefficient $2$ is too large for phase two convergence of $\ell_1$ loss. Thanksfully, work \cite{shen2017tight} provides a tight bound for hard thresholding perturbations.
\begin{theorem}[Tight Bound for Hard Thresholding \cite{shen2017tight}]
	Let $\b \in \RR^d$ be an arbitrary vector and $\x \in \RR^d$ be any $K$-sparse signal. For any $k \geq K$, we have the following bound:
	$$
	\left\|\mathcal{H}_k(\boldsymbol{b})-\boldsymbol{x}\right\|_2 \leq \sqrt{\nu}\|\boldsymbol{b}-\boldsymbol{x}\|_2, \quad \nu=1+\frac{\rho+\sqrt{(4+\rho) \rho}}{2}, \quad \rho=\frac{\min \{K, d-k\}}{k-K+\min \{K, d-k\}}.
	$$
	\label{tecthm:vec-perturb}
\end{theorem}
Note that to some extent larger hard thresholding parameter $k$ infers smaller upper bound coefficient $\nu$. Besides, Theorem~\ref{tecthm:vec-perturb} works for arbitrary $\b$ and puts no assumption on $\b$. It leads to we do not require good initializations for sparse vector recovery but need the practical thresholding parameter $\tilde{s}\geq \min\{cs,d\}$ to have convergence guarantees. We conjecture there are tighter bounds if with small enough $\ltwo{\b-\x}$ based on which we could choose $\tilde{s}=s$. We leave this for future study.
\begin{lemma}
	Suppose function $f(\cdot)$ is given in Equation~\eqref{eq:vec-absloss:sparse}. Suppose Assumptions \ref{assump:sensing operators:vec} and \ref{assump:heavy-tailed} hold. Then for any vector $\Bbeta\in\left\{\b\in\RR^d: \, \ltwo{\b-\Bbeta^*}\leq8\sqrt{\ku\kl^{-1}}\gamma\right\}$, we have \begin{align*}
		\EE f(\Bbeta)-\EE f(\Bbeta^*)\geq \frac{n}{6b_0}\ltwo{\Bbeta-\Bbeta^*}^2.
	\end{align*}
	For arbitrary $\Bbeta_1,\Bbeta_2\in\RR^d$, we have\begin{align*}
		\EE f(\Bbeta_1)-\EE f(\Bbeta_2)\leq nb_1^{-1}\ku\cdot\left[\ltwo{\Bbeta_1-\Bbeta_2}^2+2\ltwo{\Bbeta_1-\Bbeta_2}\ltwo{\Bbeta_2-\Bbeta^*}\right].
	\end{align*}
	\label{teclem:vec:l1exp}
\end{lemma}
Lemma~\ref{teclem:vec:l1exp} provides upper bound and lower bound of expectation of $f(\Bbeta_1)-f(\Bbeta_2)$ and $f(\Bbeta)-f(\Bbeta^*)$, respectively. Then we are ready to prove regularity proporties Lemma~\ref{lem:vec:sparse}.
\begin{proof}[Proof of Lemma~\ref{lem:vec:sparse}]
	We shall assume the event $$\bcalE:=\left\{\sup_{\Bbeta\in\RR^d,\; |\text{supp}(\Delta\Bbeta)|\leq 3\tilde{s}}\left|f(\Bbeta+\Delta\Bbeta)-f(\Bbeta)-\EE\left[f(\Bbeta+\Delta\Bbeta)-f(\Bbeta) \right]\right|\cdot\ltwo{\Delta\Bbeta}^{-1}\leq C\sqrt{n\tilde{s}\ku\log(2d/\tilde{s})}\right\}$$
	holds. Proposition~\ref{prop:emp:vec} proves $\PP(\bcalE)\geq 1-\exp(-\frac{C^2\tilde{s}\log(2d/\tilde{s})}{3})-3\exp\left(-\frac{\sqrt{n\log(2d/\tilde{s})}}{\log n}\right)$. Based on event $\bcalE$, we shall obtain the regularity properties. 
	
	First, analyze phase one properties, namely, when $\ltwo{\Bbeta-\Bbeta^*}\geq 8\sqrt{\kl^{-1}}\gamma$. As for $f(\Bbeta)-f(\Bbeta^*)$, under the event $\bcalE$, it has
	\begin{align*}
		f(\Bbeta)-f(\Bbeta^*)&\geq \EE f(\Bbeta)-\EE f(\Bbeta^*)-C\sqrt{n\tilde{s}\ku\log(2d/\tilde{s})}\ltwo{\Bbeta-\Bbeta^*}\\
		&=\EE\left[\lone{\X(\Bbeta-\Bbeta^*)-\bxi}-\lone{\bxi}\right]-C\sqrt{n\tilde{s}\ku\log(2d/\tilde{s})}\ltwo{\Bbeta-\Bbeta^*}\\
		&\geq \EE\lone{\X(\Bbeta-\Bbeta^*)} - 2\EE\lone{\bxi}-C\sqrt{n\tilde{s}\ku \log(2d/\tilde{s})}\ltwo{\Bbeta-\Bbeta^*}.
	\end{align*}
	Notice that $\left|\X_i^{\top}(\Bbeta-\Bbeta^*)\right|$ follows folded Gaussian distributions and then we have $\EE\lone{\X(\Bbeta-\Bbeta^*)}\geq n\sqrt{\frac{2}{\pi}}\sqrt{\kl}\ltwo{\Bbeta-\Bbeta^*}$. Thus, together with $\EE|\xi|=\gamma$, $\ltwo{\Bbeta-\Bbeta^*}\geq 8\sqrt{\kl^{-1}}\gamma$ and $n\geq C_1\ku\kl^{-1}\tilde{s}\log(2d/\tilde{s})$, it has
	\begin{align*}
		f(\Bbeta)-f(\Bbeta^*)\geq \frac{n}{4}\sqrt{\kl}\ltwo{\Bbeta-\Bbeta^*}.
	\end{align*}
	Then consider $\ltwo{\calP_{\Omega\cup\Pi\cup\Omega^*}(\G)}$. By definition of sub-gradient, it has
	\begin{align}
		\ltwo{\calP_{\Omega\cup\Pi\cup\Omega^*}(\G)}^2=\inp{\G}{\calP_{\Omega\cup\Pi\cup\Omega^*}(\G)}\leq f(\Bbeta+\calP_{\Omega\cup\Pi\cup\Omega^*}(\G))-f(\Bbeta).
		\label{eq29}
	\end{align}
	On the other hand, under the event $\bcalE$, it has
	\begin{align*}
		&{~~~~}f(\Bbeta+\calP_{\Omega\cup\Pi\cup\Omega^*}(\G))-f(\Bbeta)\\
		&\leq \EE f(\Bbeta+\calP_{\Omega\cup\Pi\cup\Omega^*}(\G))-\EE f(\Bbeta)+C\sqrt{n\tilde{s}\ku\log(2d/\tilde{s})}\ltwo{\calP_{\Omega\cup\Pi\cup\Omega^*}(\G)}\\
		&\leq \EE\lone{\X\calP_{\Omega\cup\Pi\cup\Omega^*}(\G)}+C\sqrt{n\tilde{s}\ku\log(2d/\tilde{s})}\ltwo{\calP_{\Omega\cup\Pi\cup\Omega^*}(\G)}.
	\end{align*}
	Then together with folded Gaussian distributions, it has $\EE\lone{\X\calP_{\Omega\cup\Pi\cup\Omega^*}(\G)}\leq n\sqrt{\frac{2}{\pi}}\sqrt{\ku}\ltwo{\calP_{\Omega\cup\Pi\cup\Omega^*}(\G)}$. More details of calculations can be found in Corollary~\ref{cor:l1expecation noiseless}. Furthermore, with $n\geq C_1\tilde{s}\log(2d/\tilde{s})$, it has\begin{align*}
		f(\Bbeta+\calP_{\Omega\cup\Pi\cup\Omega^*}(\G))-f(\Bbeta)\leq  n\sqrt{\ku}\ltwo{\calP_{\Omega\cup\Pi\cup\Omega^*}(\G)}.
	\end{align*}
	Then the above equation and Equation~\eqref{eq29} lead to
	\begin{align*}
		\ltwo{\calP_{\Omega\cup\Pi\cup\Omega^*}(\G)}\leq n\sqrt{\ku}.
	\end{align*}
	
	Secondly, analyze phase two properties, when $C_2b_0\sqrt{\frac{\ku}{\kl^2}\cdot\frac{\tilde{s}\log(2d/\tilde{s})}{n}}\leq\ltwo{\Bbeta-\Bbeta^*}\leq 8\sqrt{\kl^{-1}}\gamma$. With Lemma~\ref{teclem:vec:l1exp} and event $\bcalE$, we have
	\begin{align*}
		f(\Bbeta)-f(\Bbeta^*)&\geq \EE f(\Bbeta)-\EE f(\Bbeta^*) - C\sqrt{n\tilde{s}\ku\log(2d/\tilde{s})}\ltwo{\Bbeta-\Bbeta^*}\\
		&\geq \frac{n}{6b_0}\kl\ltwo{\Bbeta-\Bbeta^*}^2-C\sqrt{n\tilde{s}\ku\log(2d/\tilde{s})}\ltwo{\Bbeta-\Bbeta^*}\\
		&\geq \frac{n}{12b_0}\kl\ltwo{\Bbeta-\Bbeta^*}^2,
	\end{align*}
	where the last line is from $\ltwo{\Bbeta-\Bbeta^*}\geq C_2\sqrt{\frac{\tilde{s}\log(2d/\tilde{s})}{n}\cdot\frac{\ku}{\kl^2}}b_0$. On the other hand, we have
	\begin{align*}
		&{~~~~}f(\Bbeta+\frac{b_1}{2n\ku}\calP_{\Omega\cup\Pi\cup\Omega^*}(\G))-f(\Bbeta)\\
		&\leq \EE f(\Bbeta+\frac{b_1}{2n\ku}\calP_{\Omega\cup\Pi\cup\Omega^*}(\G))-\EE f(\Bbeta)+C_1\sqrt{n\tilde{s}\ku\log(2d/\tilde{s})}\cdot \frac{b_1}{2n\ku}\ltwo{\calP_{\Omega\cup\Pi\cup\Omega^*}(\G)}\\
		&\leq \frac{n\ku}{b_1}\left[ \left( \frac{b_1}{2n\ku}\right)^2\ltwo{\calP_{\Omega\cup\Pi\cup\Omega^*}(\G)}^2+2\frac{b_1}{2n\ku}\ltwo{\calP_{\Omega\cup\Pi\cup\Omega^*}(\G)}\ltwo{\Bbeta-\Bbeta^*}\right]\\
		&{~~~~~~~~~~~~~~~~~~~~~~~~~~~~~~~~~~~~~~~~~~~~~~~~~~~~~~~~~~~~~~~~~~~~~~~~}+C_1\sqrt{n\tilde{s}\ku\log(2d/\tilde{s})}\cdot\frac{b_1}{2n\ku}\ltwo{\calP_{\Omega\cup\Pi\cup\Omega^*}(\G)}\\
		&\leq \frac{b_1}{4n\ku}\ltwo{\calP_{\Omega\cup\Pi\cup\Omega^*}(\G)}^2+(2+C_1)\ltwo{\calP_{\Omega\cup\Pi\cup\Omega^*}(\G)}\ltwo{\Bbeta-\Bbeta^*},
	\end{align*}
	where the second equation is from  Lemma~\ref{teclem:vec:l1exp} and the last equation is from the phase two region constraint $\sqrt{\frac{\tilde{s}\log(2d/\tilde{s})}{\ku n}}b_1\leq C_2b_0\sqrt{\frac{\ku}{\kl^2}\cdot\frac{\tilde{s}\log(2d/\tilde{s})}{n}}\leq \ltwo{\Bbeta-\Bbeta^*}$. Meanwhile, by definition of sub-gradient, we obtain a lower bound for $f(\Bbeta+\frac{b_1}{2n\ku}\calP_{\Omega\cup\Pi\cup\Omega^*}(\G))-f(\Bbeta)$,
	\begin{align*}
		f(\Bbeta+\frac{b_1}{2n\ku}\calP_{\Omega\cup\Pi\cup\Omega^*}(\G))-f(\Bbeta)\geq\inp{\G}{\frac{b_1}{2n\ku}\calP_{\Omega\cup\Pi\cup\Omega^*}(\G) }=\frac{b_1}{2n\ku}\ltwo{\calP_{\Omega\cup\Pi\cup\Omega^*}(\G)}^2.
	\end{align*}
	Combine the above two equations and then the quadratic inequality leads to 
	\begin{align*}
		\ltwo{\calP_{\Omega\cup\Pi\cup\Omega^*}(\G)}\leq C_3\ku\frac{n}{b_1}\ltwo{\Bbeta-\Bbeta^*}.
	\end{align*}
\end{proof}
Then we are ready to prove convergence dynamics of IHT-$\ell_1$, Theorem~\ref{thm:vec:sparse}.
\begin{proof}[Proof of Theorem~\ref{thm:vec:sparse}]
	In phase one, for convenience, denote $D_l:= (1-c_1)^l\ltwo{\Bbeta_0-\Bbeta^*}$. Prove by induction and it holds obviously for $l=0$. We suppose $\ltwo{\Bbeta_{l}-\Bbeta^*}\leq D_l$ and we are going to prove $\ltwo{\Bbeta_{l+1}-\Bbeta^*}\leq D_{l+1}$. We have
	\begin{align*}
		\ltwo{\Bbeta_{l}-\eta_{l}\calP_{\Omega_{l}\cup\Pi_l\cup\Omega^*}(\G_l)-\Bbeta^*}^2&=\ltwo{\Bbeta_l-\Bbeta^*}^2-2\eta_{l}\inp{\Bbeta_l-\Bbeta^*}{\G_l}+\eta_{l}^2\ltwo{\calP_{\Omega_{l}\cup\Pi_l\cup\Omega^*}(\G_l)}^2\\
		&\leq \ltwo{\Bbeta_l-\Bbeta^*}^2-2\eta_{l}\left(f(\Bbeta_l)-f(\Bbeta^*)\right)+\eta_{l}^2\ltwo{\calP_{\Omega_{l}\cup\Pi_l\cup\Omega^*}(\G_l)}^2\\
		&\leq \ltwo{\Bbeta_l-\Bbeta^*}^2-\eta_{l}\frac{n}{2}\sqrt{\kl}\ltwo{\Bbeta_l-\Bbeta^*}+n^2\eta_{l}^2\ku,
	\end{align*}
	where the first line is due to $\text{supp}(\Bbeta_l-\Bbeta^*)\subseteq \Omega_{l}\cup\Pi_l\cup\Omega^*$ that $\inp{\Bbeta_l-\Bbeta^*}{\calP_{\Omega_{l}\cup\Pi_l\cup\Omega^*}(\G_l)}=\inp{\Bbeta_l-\Bbeta^*}{\G_{l}}$, the second line uses sub-gradient definition and the last line uses Lemma~\ref{lem:vec:sparse}.
	By induction, we have $\ltwo{\Bbeta_{l}-\Bbeta^*}\leq D_l$ and selection of stepsize says $\eta_l\in n^{-1}\sqrt{\frac{\kl}{\ku^2}}\cdot[\frac{1}{8}D_l,\frac{3}{8}D_l]$. Based on quadratic function properties Claim~\ref{claim:quadraticfct}, we have
	$$\ltwo{\Bbeta_l-\Bbeta^*}^2-\eta_{l}\frac{n}{2}\sqrt{\kl}\ltwo{\Bbeta_l-\Bbeta^*}\leq D_l^2- \eta_{l}\frac{n}{2}\sqrt{\kl} D_l.$$ Then insert stepsize selection and with quadratic function properties, it becomes
	\begin{align*}
		\ltwo{\Bbeta_{l}-\eta_{l}\calP_{\Omega_{l}\cup\Pi_l\cup\Omega^*}(\G_l)-\Bbeta^*}^2\leq \left(1-\frac{3}{64}\frac{\kl}{\ku}\right)\cdot D_{l}^2.
	\end{align*}
	It implies $$\ltwo{\Bbeta_{l}-\eta_{l}\calP_{\Omega_{l}\cup\Pi_l\cup\Omega^*}(\G_l)-\Bbeta^*}\leq \left(1-\frac{3}{128}\frac{\kl}{\ku}\right) D_{l}. $$
	Then with hard thresholding perturbation Theorem~\ref{tecthm:vec-perturb} and $\tilde{s}$ selection, we obtain\begin{align*}
		\ltwo{\Bbeta_{l+1}-\Bbeta^*}=\ltwo{\calH_{\tilde{s}}(\Bbeta_{l}-\eta_{l}\calP_{\Omega_{l}\cup\Pi_l\cup\Omega^*}(\G_l))-\Bbeta^*}\leq \left(1-\frac{1}{64}\frac{\kl}{\ku}\right)D_l\leq (1-c_1)D_l=D_{l+1},
	\end{align*}
	where constant $c_1\leq \frac{1}{64}\frac{\ku}{\kl}$. 
	
	Similarly, in phase two, based on Lemma~\ref{lem:vec:sparse}, we have
	\begin{align*}
		\ltwo{\Bbeta_{l}-\eta_{l}\calP_{\Omega_{l}\cup\Pi_l\cup\Omega^*}(\G_l)-\Bbeta^*}^2&=\ltwo{\Bbeta_l-\Bbeta^*}^2-2\eta_{l}\inp{\Bbeta_l-\Bbeta^*}{\G_l}+\eta_{l}^2\ltwo{\calP_{\Omega_{l}\cup\Pi_l\cup\Omega^*}(\G_l)}^2\\
		&\leq \ltwo{\Bbeta_l-\Bbeta^*}^2-2\eta_{l}\left(f(\Bbeta_l)-f(\Bbeta^*)\right)+\eta_{l}^2\ltwo{\calP_{\Omega_{l}\cup\Pi_l\cup\Omega^*}(\G_l)}^2\\
		&\leq \ltwo{\Bbeta_l-\Bbeta^*}^2-\eta_{l}\frac{n}{6b_0}\kl\ltwo{\Bbeta_l-\Bbeta^*}^2+C_3^2n^2b_1^{-2}\eta_{l}^2\ku^2\ltwo{\Bbeta_l-\Bbeta^*}^2.
	\end{align*}
	With stepsize $\eta_{l}\in\frac{\kl}{\ku^2}\cdot\frac{b_1^2}{b_0}\cdot[C_4,C_5]$, we have
	\begin{align*}
		\ltwo{\Bbeta_{l}-\eta_{l}\calP_{\Omega_{l}\cup\Pi_l\cup\Omega^*}(\G_l)-\Bbeta^*}^2\leq \left(1-\tilde{c}_2\frac{\kl^2}{\ku^2}\cdot\frac{b_1^2}{b_0^2}\right) \ltwo{\Bbeta_{l}-\Bbeta^*}^2.
	\end{align*}
	It leads to $$\ltwo{\Bbeta_{l}-\eta_{l}\calP_{\Omega_{l}\cup\Pi_l\cup\Omega^*}(\G_l)-\Bbeta^*}\leq \left(1-\tilde{c}_2\frac{\kl^2}{2\ku^2}\cdot\frac{b_1^2}{b_0^2}\right) \ltwo{\Bbeta_{l}-\Bbeta^*}. $$
	Then with $\tilde{s}\geq\min\left\{C_6s\cdot\frac{\ku^8}{\kl^8}\cdot\frac{b_0^8}{b_1^8},d \right\} $, perturbation bound Theorem~\ref{tecthm:vec-perturb} infers
	\begin{align*}
		\ltwo{\Bbeta_{l+1}-\Bbeta^*}&\leq \left(1+\sqrt{\frac{s}{\tilde{s}}+\sqrt{\frac{s}{\tilde{s}}}}\right)\cdot\left(1-\tilde{c}_2\frac{\kl^2}{2\ku^2}\cdot\frac{b_1^2}{b_0^2}\right) \ltwo{\Bbeta_{l}-\Bbeta^*}\\
		&\leq\left(1-c_2\frac{\kl^2}{\ku^2}\cdot\frac{b_1^2}{b_0^2}\right) \ltwo{\Bbeta_{l}-\Bbeta^*}.
	\end{align*}
\end{proof}

\begin{proof}[Proof of  Proposition~\ref{prop:emp:vec}, Vector Recovery Empirical Process]
	Denote $\VV_s:= \left\{\b\in\RR^{d}:\; \left|\text{supp}(\b)\right|\leq s\right\}$. Then consider
	\begin{align*}
		Z:=\sup_{\Bbeta\in\RR^d,\; \Delta\Bbeta\in \VV_s}\left| f(\Bbeta+\Delta\Bbeta)-f(\Bbeta)-\EE\left[f(\Bbeta+\Delta\Bbeta)-f(\Bbeta) \right]\right|\big/ \ltwo{\Delta\Bbeta}.
	\end{align*}
	With $f(\Bbeta):=\sum_{i=1}^{n}|\X_i^{\top}\Bbeta-Y_i|$, we have
	\begin{align*}
		\EE Z&\leq 2\EE\sup_{\Bbeta\in\RR^d,\; \Delta\Bbeta\in \VV_s}\left| \sum_{i=1}^{n}\tilde{\epsilon}_i\cdot \left( |\X_i^{\top}(\Bbeta+\Delta\Bbeta)-Y_i|-|\X_i^{\top}\Bbeta-Y_i|\right)\right|\big/ \ltwo{\Delta\Bbeta}\\
		&\leq 4\EE\sup_{\Delta\Bbeta\in \VV_s}\left| \sum_{i=1}^{n}\tilde{\epsilon}_i\cdot \X_i^{\top}\Delta\Bbeta\right|\big/ \ltwo{\Delta\Bbeta}\\
		&= 4\EE\sup_{\Delta\Bbeta\in \VV_s}\left| \left(\sum_{i=1}^{n}\tilde{\epsilon}_i\X_i\right)^{\top}\Delta\Bbeta\right|\big/ \ltwo{\Delta\Bbeta}\\
		&\leq 4\EE\sup_{|\Omega|\leq s}\left\|\calP_{\Omega}\left(\sum_{i=1}^{n}\tilde{\epsilon}_i\X_i\right) \right\|_{2}.
	\end{align*}
	where $\{\wt\epsilon_{i}\}_{i=1}^n$ is Rademacher sequence independent of $\X_1,\ldots,\X_n$ and operator $\calP_{\Omega}(\cdot)$ only keeps entries in $\Omega$. The first inequality is from Theorem \ref{Symmetrization of Expectation} and  the second inequality is from Theorem \ref{Contraction Theorem}. 
	
	Notice that $\sum_{i=1}^n\wt\epsilon_i\X_i\sim N(\boldsymbol{0}, \sum_{i=1}^{n}\bSigma_i)$ and $\op{\sum_{i=1}^{n}\bSigma_i}\leq n\ku$ and then by Corollary~\ref{teccol:maxsums}, we have $\EE\sup_{|\Omega|\leq s}\left\|\calP_{\Omega}\left(\sum_{i=1}^{n}\tilde{\epsilon}_i\X_i\right) \right\|_{2}\leq \left(\EE\sup_{|\Omega|\leq s}\left\|\calP_{\Omega}\left(\sum_{i=1}^{n}\tilde{\epsilon}_i\X_i\right) \right\|_{2}^2 \right)^{1/2}\leq C\sqrt{ns\ku\log(2d/s)}$. Thus, $\EE Z$ could be upper bounded with $$\EE Z\leq C\sqrt{ns\ku\log(2d/s)}.$$
	With similar analyses in matrix case Section~\ref{proof:thm:empirical}, we have\begin{align*}
		&~~~\left\|\max_{i=1,\ldots,n}\sup_{\Bbeta\in\RR^d,\; \Delta\Bbeta\in \VV_s}|\X_i^{\top}(\Bbeta+\Delta\Bbeta)-Y_i|-|\X_i^{\top}\Bbeta-Y_i|-\EE\left[|\X_i^{\top}(\Bbeta+\Delta\Bbeta)-Y_i|-|\X_i^{\top}\Bbeta-Y_i| \right]\right\|_{\Psi_{1}}\\
		&\leq C_1\sqrt{\ku s}\log n,
	\end{align*}
	and
	\begin{align*}
		\EE\left(|\X_i^{\top}(\Bbeta+\Delta\Bbeta)-Y_i|-|\X_i^{\top}\Bbeta-Y_i| \right)^2\big/ \ltwo{\Delta\Bbeta}^2\leq \ku.
	\end{align*}
	Thus with Theorem~\ref{tecthm:orlicz norm empirical}, we have
	\begin{align*}
		\PP\left(Z\geq C\sqrt{ns\ku\log(2d/s)}\right)\leq\exp\left(-\frac{C^2s\log(2d/s)}{3}\right)+3\exp\left(-\frac{\sqrt{n\log(2d/s)}}{\log n}\right),
	\end{align*}
	which completes the proof.
\end{proof}

\begin{proof}[Proof of Lemma~\ref{teclem:vec:l1exp}]
	Notice that $\EE f(\Bbeta)-\EE f(\Bbeta^*)$ has the following expansion \begin{align*}
		\EE f(\Bbeta)-\EE f(\Bbeta^*)&=\sum_{i=1}^{n} \EE\left|\xi_i-\inp{\X_i}{\Bbeta-\Bbeta^*}\right|-\EE\left|\xi_{i}\right|\\
		&=\sum_{i=1}^{n} \EE\left[\EE\left[\left|\xi-\inp{\X_i}{\Bbeta-\Bbeta^*}\right|-\left|\xi\right|\bigg|\X_i\right]\right].
	\end{align*}
	To bound the conditional expectation $\EE\left[\left|\xi-\inp{\X_i}{\Bbeta-\Bbeta^*}\right|-\left|\xi\right|\bigg|\X_i\right]$, we shall first analyze $\EE|\xi-t_0|$, with $t_0\in\RR$,
	\begin{equation}
		\begin{split}
			\EE|\xi-t_0| &= \int_{s\geq t_0}(s-t_0)\,dH_{\xi}(s) + \int_{s< t_0}(t_0-s)\,dH_{\xi}(s)\\
			&=2\int_{s\geq t_0}(s-t_0)\,dH_{\xi}(s) + \int_{-\infty}^{+\infty}(t_0-s)\,dH_{\xi}(s)\\
			&=2\int_{s\geq t_0}(1-H_{\xi}(s))\,ds + t_0 - \int_{-\infty}^{\infty}s \,dH_{\xi}(s),
		\end{split}
		\label{eq:l1condiexp-vec}
	\end{equation}
	where the last line is due to \begin{align*}
		\int_{s\geq t_0}(s-t_0)\,dH_{\xi}(s)&=-\int_{s\geq t_0} (s-t_0)\, d(1-H_\xi(s))\\
		&=-(s-t_0)(1-H_\xi(s))\bigg|_{t_0}^{+\infty}+\int_{s\geq t_0}\left(1-H_\xi
		(s)\right)\, ds\\
		&=\lim_{s\to+\infty}(s-t_0)(1-H_\xi(s))+\int_{s\geq t_0}\left(1-H_\xi
		(s)\right)\, ds,
	\end{align*} whose first term vanishes
	\begin{align*}
		0\leq\lim_{s\to+\infty}(s-t_0)(1-H_\xi(s))\leq\lim_{s\to+\infty} (s-t_0)\cdot\PP(\xi>s)\leq\lim_{s\to+\infty} (s-t_0)\cdot\frac{\EE|\xi|^{1+\eps}}{s^{1+\eps}}=0.
	\end{align*}
	Set $t_0 = \inp{\X_i}{\Bbeta-\Bbeta^*}$ and then Equation~\eqref{eq:l1condiexp-vec} becomes
	\begin{align*}
		\EE\left[|\xi-\inp{\X_i}{\Bbeta-\Bbeta^*}| \bigg|\X_i\right]= 2\int_{s\geq \inp{\X_i}{\Bbeta-\Bbeta^*}}(1-H_{\xi}(s))\,ds + \inp{\X_i}{\Bbeta-\Bbeta^*} - \int_{-\infty}^{+\infty}s dH_{\xi}(s).
	\end{align*}
	With $t_0=0$, it becomes
	\begin{align*}
		\EE|\xi| = 2\int_{s\geq 0}(1-H_{\xi}(s))\,ds - \int_{-\infty}^{+\infty}s dH_{\xi}(s).
	\end{align*}
	Thus, by taking difference of the above two equations, we have the conditional expectation $\EE\left[\left|\xi-\inp{\X_i}{\Bbeta-\Bbeta^*}\right|-\left|\xi\right|\bigg|\X_i\right]$,
	\begin{align*}
		\EE\left[\left|\xi-\inp{\X_i}{\Bbeta-\Bbeta^*}\right|-\left|\xi\right|\bigg|\X_i\right]
		&=-2\int_{0}^{\inp{\X_i}{\Bbeta-\Bbeta^*}}(1-H_{\xi}(s))\,ds + \inp{\X_i}{\Bbeta-\Bbeta^*}\\
		&=2\int_{0}^{\inp{\X_i}{\Bbeta-\Bbeta^*}}(H_{\xi}(s)-H_{\xi}(0))\,ds\\
		&=2\int_{0}^{\inp{\X_i}{\Bbeta-\Bbeta^*}}\int_{0}^{s}h_{\xi}(w)\, dw\, ds,
	\end{align*}
	where $H_{\xi}(0)=1/2$ is used. Note that $\inp{\X_i}{\Bbeta-\Bbeta^*}\sim N(0, (\Bbeta-\Bbeta^*)^{\top}\bSigma_i(\Bbeta-\Bbeta))$. Denote $z_i:= \inp{\X_i}{\Bbeta-\Bbeta^*}$, $\sigma_i^2:= (\Bbeta-\Bbeta^*)^{\top}\bSigma_i(\Bbeta-\Bbeta)$ and let $f_{z_i}(\cdot)$ be density of $z_i$. Take expectation over $\X_i$ on each side of the above equation and then we obtain,
	\begin{align*}
		\EE \left[\left|\xi-\inp{\X_i}{\Bbeta-\Bbeta^*}\right|-\left|\xi\right|\right]=2\int_{-\infty}^{+\infty}\int_{0}^{t}\int_{0}^{s}h_{\xi}(w) f_{z_i}(t)\, dw\, ds\, dt.
	\end{align*}
	Note that with Assumption~\ref{assump:heavy-tailed}, we have
	\begin{align*}
		\EE \left[\left|\xi-\inp{\X_i}{\Bbeta-\Bbeta^*}\right|-\left|\xi\right|\right]&\geq2\int_{-\sigma_i}^{+\sigma_i}\int_{0}^{t}\int_{0}^{s}h_{\xi}(w) f_{z_i}(t)\, dw\, ds\, dt\\
		&\geq \frac{2}{b_0}\int_{-\sigma_i}^{+\sigma_i}\int_{0}^{t}s f_{z_i}(t)\, ds\, dt\\
		&=\frac{1}{b_0}\int_{-\sigma_i}^{+\sigma_i}t^2 f_{z_i}(t)\, dt\\
		&=\frac{\sigma_i^2}{b_0}\int_{-1}^{1} t^2\cdot\frac{1}{\sqrt{2\pi}}\exp(-t^2/2)\, dt\geq \frac{\sigma_i^2}{6b_0},
	\end{align*}
	where the last line uses $\int_{-1}^{1} t^2\cdot\frac{1}{\sqrt{2\pi}}\exp(-t^2/2)\, dt\geq\frac{1}{6}$. Besides, Assumption~\ref{assump:sensing operators:vec} implies $\sigma_i^2\geq \kl\ltwo{\Bbeta-\Bbeta^*}^2$. Thus we have\begin{align*}
		\EE f(\Bbeta)-\EE f(\Bbeta^*)|&=\sum_{i=1}^{n} \EE \left[\left|\xi-\inp{\X_i}{\Bbeta-\Bbeta^*}\right|-\left|\xi\right|\right]\\
		&\geq \frac{n}{6b_0}\kl\ltwo{\Bbeta-\Bbeta^*}^2.
	\end{align*}
	On the other hand for arbitrary $\Bbeta_1,\Bbeta_2\in\RR^d$, based on Equation~\eqref{eq:l1condiexp-vec} it has
	\begin{align*}
		&{~~~~}\EE\left[\left|\xi-\inp{\X_i}{\Bbeta_1-\Bbeta^*}\right|-\left|\xi-\inp{\X_i}{\Bbeta_2-\Bbeta^*}\right|\bigg|\X_i\right]\\
		&=2\int_{\inp{\X_i}{\Bbeta_2-\Bbeta^*}}^{\inp{\X_i}{\Bbeta_1-\Bbeta^*}}\int_{0}^{s}h_{\xi}(w)\, dw\, ds\\
		&\leq\frac{2}{b_1}\left|\int_{\inp{\X_i}{\Bbeta_2-\Bbeta^*}}^{\inp{\X_i}{\Bbeta_1-\Bbeta^*}}\int_{0}^{s}\, dw\, ds\right|\\
		&=b_1^{-1}\cdot\left|\inp{\X_i}{\Bbeta_1-\Bbeta_2}^2+2\inp{\X_i}{\Bbeta_1-\Bbeta_2}\inp{\X_i}{\Bbeta_2-\Bbeta^*}\right|,
	\end{align*}
	where the third line is from noise density condition, Assumption~\ref{assump:heavy-tailed}. Then take expectation over $\X_i$ on each side of the above equation and it becomes
	\begin{align*}
		\EE\left[\left|\xi-\inp{\X_i}{\Bbeta_1-\Bbeta^*}\right|-\left|\xi-\inp{\X_i}{\Bbeta_2-\Bbeta^*}\right|\right]&\leq b_1^{-1}\ku\cdot\left[\ltwo{\Bbeta_1-\Bbeta_2}^2+2\ltwo{\Bbeta_1-\Bbeta_2}\ltwo{\Bbeta_2-\Bbeta^*}\right],
	\end{align*}
	where Holder's inequality is used. Sum over $i=1,\dots,n$ and it is
	\begin{align*}
		\EE f(\Bbeta_1)-\EE f(\Bbeta_2)\leq nb_1^{-1}\ku\cdot\left[\ltwo{\Bbeta_1-\Bbeta_2}^2+2\ltwo{\Bbeta_1-\Bbeta_2}\ltwo{\Bbeta_2-\Bbeta^*}\right],
	\end{align*}
	which completes the proof.
\end{proof}
\begin{lemma}[Proposition~E.1 of \cite{bellec2018slope}]
	Let $g_1, \ldots, g_p$ be zero-mean Gaussian random variables with variance at most $\sigma^2$. Denote by $\left(g_1^{\sharp}, \ldots, g_p^{\sharp}\right)$ be a non-increasing rearrangement of $\left(\left|g_1\right|, \ldots,\left|g_p\right|\right)$. Then
	$$
	\mathbb{P}\left(\frac{1}{s \sigma^2} \sum_{j=1}^s\left(g_j^{\sharp}\right)^2>t \log \left(\frac{2 p}{s}\right)\right) \leq\left(\frac{2 p}{s}\right)^{1-\frac{3 t}{8}}
	$$
	for all $t>0$ and $s \in\{1, \ldots, p\}$.
	\label{teclem:maxsums}
\end{lemma}
\begin{corollary}[Corollary of Lemma~\ref{teclem:maxsums}]
	Under same conditions of Lemma~\ref{teclem:maxsums}, we have\begin{align*}
		\EE\sum_{i=1}^{s}\left(g_i^{\sharp}\right)^2\leq Cs\log(p/s)\sigma^2,
	\end{align*}
	where $C$ is some absolute constant.
	\label{teccol:maxsums}
\end{corollary}
\begin{proof}
	We have,
	\begin{align*}
		&{~~~~}\EE\sum_{i=1}^{s}\left(g_i^{\sharp}\right)^2\\
		&= \EE\left[\sum_{i=1}^{s}\left(g_i^{\sharp}\right)^2\cdot 1_{ \sum_{i=1}^{s}\left(g_i^{\sharp}\right)^2\leq 8s\log(2p/s)\sigma^2} \right]+\sum_{j=8}^{+\infty} \EE\left[\sum_{i=1}^{s}\left(g_i^{\sharp}\right)^2\cdot 1_{ js\log(2p/s)\sigma^2\leq \sum_{i=1}^{s}\left(g_i^{\sharp}\right)^2\leq (j+1)s\log(2p/s)\sigma^2} \right]\\
		&\leq 8s\log(2p/s)\sigma^2+\sum_{j=8}^{+\infty} (j+1)s\log(2p/s)\sigma^2\cdot \PP\left(\sum_{i=1}^{s}\left(g_i^{\sharp}\right)^2> js\log(2p/s)\sigma^2 \right)\\
		&\leq 8s\log(2p/s)\sigma^2+ s\log(2p/s)\sigma^2\cdot\sum_{j=8}^{+\infty} (j+1)\left(\left(\frac{s}{2p}\right)^{\frac{1}{4}}\right)^{j},
	\end{align*}
	where the last line is from Lemma~\ref{teclem:maxsums} and $1-\frac{3}{8}t\leq-\frac{1}{4}t$ when $t\geq8$. Also, note that $\left(\frac{s}{2p}\right)^{\frac{1}{4}}\leq 0.5^{\frac{1}{4}}\leq 0.85$ and then with mathematical analyses, we have
	\begin{align*}
		\sum_{j=8}^{+\infty} (j+1)\left(\left(\frac{s}{2p}\right)^{\frac{1}{4}}\right)^{j}\leq \sum_{j=1}^{+\infty} (j+1)\cdot (0.85)^{j}\leq \frac{1}{(1-0.85)^2}\leq 49.
	\end{align*}
	Thus there exist some absolute constant $C>0$, such that
	\begin{align*}
		\EE\sum_{i=1}^{s}\left(g_i^{\sharp}\right)^2\leq Cs\log(2p/s)\sigma^2.
	\end{align*}
\end{proof}

\section{Proof of Proposition~\ref{prop:main}}
\label{proof:prop:main}
\subsection*{Phase One} For convenience, denote $d_{l}:=\left(1-c\frac{\muc^2}{\Lc^2}\right)^{l}\cdot\fro{\M_0-\M^*}$. Prove by induction. When $l=0$, $\fro{\M_{0}-\M^*}\leq d_{0} $ is obvious. Suppose we already have $\fro{\M_l-\M^*}\leq d_{l}$ and we are going to prove $\fro{\M_{l+1}-\M^*}\leq d_{l+1}$. 

Notice that $\M_{l+1}=\text{SVD}_r(\M_{l}-\eta_{l}\mathcal{P}_{\TT_l}(\G_l))$ and we shall use Lemma~\ref{teclem:perturbation} to bound the distance between $\M_{l+1}$ and the ground truth matrix $\M^*$. First consider $\Vert \M_{l}-\eta_{l}\mathcal{P}_{\mathbb{T}_{l}}(\G_{l})-\M^{*}\Vert_{\mathrm{F}}^{2}$,
\begin{equation}\label{eq1}
	\begin{split}
		&\Vert \M_{l}-\eta_{l}\mathcal{P}_{\mathbb{T}_{l}}(\G_{l})-\M^{*}\Vert_{\mathrm{F}}^{2}\\
		&{~~~~~~~~~}= \fro{\M_{l}-\M^{*}}^{2} - 2\eta_{l}\langle \M_{l}-\M^{*}, \mathcal{P}_{\mathbb{T}_{l}}(\G_{l})\rangle + \eta_{l}^{2}\fro{ \mathcal{P}_{\mathbb{T}_{l}}(\G_{l})}^{2}.
	\end{split}
\end{equation} By definition of sub-gradient, the second term can be bounded in the way of
\begin{equation}
	\begin{split}
		\inp{\M_l-\M^*}{\mathcal{P}_{\mathbb{T}_{l}}(\G_{l}) }&=\inp{\M_l-\M^*}{\G_{l}}-\inp{\M_l-\M^*}{\mathcal{P}_{\mathbb{T}_{l}}^{\perp}(\G_l)}\\
		&\geq f(\M_{l}) - f(\M^{*}) -\inp{\M_l-\M^*}{\mathcal{P}_{\mathbb{T}_{l}}^{\perp}(\G_l)}.\\
	\end{split}
	\label{eq2-2}
\end{equation}
Notice that $ \langle \M_{l}-\M^{*}, \mathcal{P}_{\mathbb{T}_{l}}^{\perp}(\mathbf{G}_{l})\rangle $ has the following expansion, see Lemma 4.1 of \cite{wei2016guarantees} for more detailed calculations,
\begin{align*}
	\inp{\M_l-\M^*}{\mathcal{P}_{\mathbb{T}_{l}}^{\perp}(\G_l)}&=\inp{ \calP_{\TT_l}^{\perp}(\M_{l}-\M^{*})}{\mathbf{G}_{l}}\\
	&=-\inp{\calP_{\TT_l}^{\perp}(\M^{*})}{\mathbf{G}_{l}}\\
	&=\inp{(\U\U^{\top}-\U_l\U_l^{\top})(\M_l-\M^*)(\I-\V_l\V_l^\top)}{\G_{l}}.
\end{align*}
Besides, the term $(\U\U^{\top}-\U_l\U_l^{\top})(\M_l-\M^*)(\I-\V_l\V_l^\top)=-(\U\U^{\top}-\U_l\U_l^{\top})\M^*(\I-\V_l\V_l^\top)$ has rank at most $r$, due to $\text{rank}(\M^*)\leq r$. Therefore by properties of partial Frobenius norm Lemma~\ref{teclem:partial F norm}, we have
\begin{align*}
	\inp{\M_l-\M^\ast}{\mathcal{P}_{\mathbb{T}_{l}}^{\perp}(\G_{l})}&=\inp{(\U\U^{\top}-\U_l\U_l^{\top})(\M_l-\M^*)(\I-\V_l\V_l^\top)}{\G_{l}}\\
	&\leq\fro{(\U\U^{\top}-\U_l\U_l^{\top})(\M_l-\M^*)(\I-\V_l\V_l^{\top})}\fror{\G_{l}}.
\end{align*}
Furthermore, Lemma~\ref{teclem:perturbation} shows perturbation bound $\op{\U\U^{\top}-\U_l\U_l^\top}\leq \frac{8\fro{\M_l-\M^*}}{\sigma_{r}}$, so the above equation becomes $$\inp{\M_l-\M^\ast}{\mathcal{P}_{\mathbb{T}_{l}}^{\perp}(\G_{l})}\leq 8\frac{1}{\sigma_r}\fro{\M_l-\M^\ast}^2\fror{\G_l}.$$
Also, Condition~\ref{assump:two-phase} assumes $\fror{\G_l}\leq \Lc$ and then the above equation becomes $$\inp{\M_l-\M^\ast}{\mathcal{P}_{\mathbb{T}_{l}}^{\perp}(\G_{l})}\leq 8\frac{\Lc}{\sigma_r}\fro{\M_l-\M^\ast}^2.$$
Thus Equation~\eqref{eq2-2} has the lower bound \begin{align*}
	\inp{\M_l-\M^*}{\mathcal{P}_{\mathbb{T}_{l}}(\G_{l}) }\geq f(\M_{l}) - f(\M^{*})-8\frac{\Lc}{\sigma_r}\fro{\M_l-\M^\ast}^2.
\end{align*}
Then analyze last term of Equation~\eqref{eq1}, $\fro{\calP_{\TT_l}(\G_l)}^2$. Though $$\calP_{\TT_l}(\G_l)=\U_l\U_l^{\top}\G_l+\G_l\V_l\V_l^{\top}-\U_l\U_l^{\top}\V_l\V_l^{\top}$$ is not $\text{SVD}_r(\G_l)$, since $\calP_{\TT_l}(\G_l)$ could be written into sum of two rank at most $r$ matrices, $\fro{\calP_{\TT_l}(\G_l)}$ can be upper bounded by $\fro{\text{SVD}_r(\G_l)}$ up to constant scales. Lemma~\ref{teclem:projected subgradient norm} shows $$\fro{\calP_{\TT_l}(\G_l)}\leq\sqrt{2}\fror{\G_l}.$$
Then with phase one regularity conditions in Condition~\ref{assump:two-phase}, $$\fror{\G_l}\leq\Lc,\quad f(\M_{l}) - f(\M^{*})\geq \muc \fro{\M_l-\M^*},$$Equation~\eqref{eq1} could be  bounded in the way of
\begin{equation}
	\begin{split}
		&\fro{\M_l-\eta_l\mathcal{P}_{\TT_l}(\G_l)-\M^*}^2\\
		&{~~~~~~~~}\leq \fro{\M_{l}-\M^*}^2-2\eta_l\muc\fro{\M_l-\M^*}+16\eta_{l}\frac{\Lc}{\sigma_{r}}\fro{\M_l-\M^*}^2+2\eta_l^2\Lc^{2}\\
		&{~~~~~~~~}\leq \fro{\M_{l}-\M^*}^2-\eta_l\muc\fro{\M_l-\M^*}+2\eta_l^2\Lc^{2},
	\end{split}
	\label{eq:prop1-eq1}
\end{equation}
where the last line uses the initialization condition and the induction that $\Vert \M_{l} - \M^{*}\Vert_{\mathrm{F}}\leq d_l\leq \fro{\M_0-\M^*}\leq\frac{1}{16}\frac{\muc}{\Lc}\sigma_{r}$.
\begin{claim}
	For quadratic function $q(x)=x^2-ax:\RR\to\RR$ with parameter $a>0$ and for any $x_2\geq x_1\geq 0$, if $q(x_2)>0$, then we have $$q(x_1)\leq q(x_2).$$
	\label{claim:quadraticfct}
\end{claim}
\begin{proof}[Proof of Claim~\ref{claim:quadraticfct}]
	On the one hand, if $q(x_1)\leq 0$, then $q(x_1)\leq0<q(x_2)$. On the other hand, if $q(x_1)\geq 0$, then it infers $x_1>a$. Notice that $q(\cdot)$ is an increasing function when restricted to $(a/2,+\infty)$. Thus, the relation $x_1\leq x_2$ implies $q(x_1)\leq q(x_2)$.
\end{proof}
Regard $\fro{\M_{l}-\M^*}^2-\eta_l\muc\fro{\M_l-\M^*}$ as quadratic function of $\fro{\M_{l}-\M^*}$. By induction, one already has $\fro{\M_{l}-\M^{*}}\leq d_{l}$ and by stepsize selection, we know $\eta_l$ satisfies $$ \frac{1}{5}d_{l}\muc\Lc^{-2}\leq\eta_{l}=(1-c\muc^2/\Lc^2)^{l}\cdot\eta_{0}\leq \frac{3}{10}d_{l}\muc\Lc^{-2}.$$ Thus, the quadratic function is positive at point $d_l$, namely $d_l^2-\eta_{l}\muc d_l>0$. Then Claim~\ref{claim:quadraticfct} implies
\begin{align*}
	\fro{\M_{l}-\M^*}^2-\eta_l\muc\fro{\M_l-\M^*}\leq d_l^2-\eta_{l}\muc d_l.
\end{align*}
As a consequence, Equation~\eqref{eq:prop1-eq1} could be bounded with
\begin{equation}
	\begin{split}
		\fro{\M_l-\eta_l\mathcal{P}_{\TT_l}(\G_l)-\M^*}^2&\leq d_l^2-\eta_{l}\muc d_l+2\eta_{l}^2\Lc^2\\
		&\leq (1-\frac{3}{25}\frac{\muc^2}{\Lc^2})d_{l}^2,
	\end{split}
	\label{eq2}
\end{equation}
where the last line is from $ \frac{1}{5}d_{l}\muc\Lc^{-2}\leq\eta_{l}\leq \frac{3}{10}d_{l}\muc\Lc^{-2}$. Take square root on each side of Equation~\eqref{eq2},
\begin{align}
	\fro{\M_l-\eta_l\mathcal{P}_{\TT_l}(\G_l)-\M^*}\leq(1-\frac{3}{50}\frac{\muc^2}{\Lc^2})d_{l}<d_0\leq\sigma_{r}/8.
	\label{eq2-1}
\end{align} Then apply matrix perturbation bound Lemma~\ref{teclem:perturbation},
\begin{equation}
	\begin{split}
		\fro{\M_{l+1}-\M^*}
		&= \Vert  {\rm SVD}_{r}(\M_{l} - \eta_{l}\mathcal{P}_{\mathbb{T}_{l}}(\mathbf{G}_{l}))- \M^{*}\Vert_{\mathrm{F}}\\
		&\leq \fro{\M_l-\eta_l\mathcal{P}_{\TT_l}(\G_l)-\M^*}+\frac{40}{\sigma_{r}}\fro{\M_l-\eta_l\mathcal{P}_{\TT_l}(\G_l)-\M^*}^{2},
	\end{split}
	\label{eq3}
\end{equation}
Insert Equation~\eqref{eq2-1} into Equation~\eqref{eq3} :
\begin{align*}
	\fro{\M_{l+1}-\M^*}&\leq (1-\frac{3}{50}\frac{\muc^{2}}{\Lc^{2}})d_{l}  + \frac{40}{\sigma_{r}}(1-\frac{3}{50}\frac{\muc^{2}}{\Lc^{2}})^2d_{l}^{2}\\
	&\leq (1-\frac{1}{25}\frac{\muc^{2}}{\Lc^{2}})d_{l},
\end{align*}
which uses $\frac{40}{\sigma_{r}}d_{l}\leq\frac{40}{\sigma_{r}}d_{0} \leq\frac{1}{50}\frac{\muc^{2}}{\Lc^{2}}$. That is $$ \fro{\M_{l+1}-\M^*}\leq (1-\frac{1}{25}\frac{\muc^{2}}{\Lc^{2}})d_{l}\leq \left(1-c\frac{\muc^2}{\Lc^2}\right)d_l=d_{l+1},$$
with constant $c\leq\frac{1}{25}$. Hence, analyses of phase one convergence finish.
\subsection*{Phase Two} Consider the second phase, when $\taus\leq\fro{\M_l-\M^*}<\tauc$. The derivation is slightly similar to phase one case,
\begin{align*}
	&{~~~~}\fro{\M_l-\eta_l\mathcal{P}_{\TT_l}(\G_l)-\M^*}^2\\
	&=\fro{\M_l-\M^*}^{2} - 2\eta_l\langle \M_{l}-\M^{*}, \mathcal{P}_{\mathbb{T}_{l}}(\mathbf{G}_{l})\rangle + \eta_l^{2}\Vert \mathcal{P}_{\mathbb{T}_{l}}(\mathbf{G}_{l})\Vert_{\mathrm{F}}^{2}\\
	&\leq \fro{\M_l-\M^*}^{2} - 2\eta_l(f(\M_{l}) - f(\M^{*})) + 2\eta_l\frac{\Ls}{\sigma_{r}}\Vert \M_{l} - \M^{*}\Vert_{\mathrm{F}}^{3} + 2\eta_l^{2}\fror{\G_l}^2\\
	&\leq \fro{\M_l-\M^*}^{2} - 2\eta_l\mus\fro{\M_l-\M^*}^{2}+ 2\eta_l\frac{\Ls}{\sigma_{r}}\fro{\M_l-\M^*}^{3} + 2\eta_l^{2}\Ls^2\fro{\M_l-\M^*}^2\\
	&\leq(1-\frac{7}{4}\eta_l\mus)\fro{\M_l-\M^*}^{2}+2\eta_l^{2}\Ls^2\fro{\M_l-\M^*}^2,
\end{align*}
where the fourth line follows from Condition~\ref{assump:two-phase} that $f(\M_l)-f(\M^*)\geq\mus\fro{\M_l-\M^*}^2$, $\fror{\G_l}\leq \Ls\fro{\M_l-\M^*}$ and the last line uses $\fro{\M_l-\M^*}\leq\fro{\M_0-\M^*}\leq \frac{1}{8}\frac{\mus}{\Ls}\sigma_r$.
Take stepsize $ \frac{1}{8}\cdot\frac{\mus}{\Ls^2}\leq\eta_l \leq \frac{3}{4}\cdot\frac{\mus}{\Ls^2}$ and one has
\begin{align*}
	\fro{\M_l-\eta\calP_{\TT_l}(\G_l)-\M^*}^2 \leq (1-\frac{1}{8}\frac{\mus^2}{\Ls^2}) \fro{\M_{l} -\M^{*}}^2.
\end{align*}
Take the square root of the above equation and it arrives at \begin{align*}
	\fro{\M_l-\eta\calP_{\TT_l}(\G_l)-\M^*}\leq (1-\frac{1}{16}\frac{\mus^2}{\Ls^2}) \fro{\M_{l} -\M^{*}}<\fro{\M_0-\M^*}\leq\sigma_{r}/8.
\end{align*}
Then use Lemma~\ref{teclem:perturbation} to cope with the matrix perturbations,
\begin{align*}
	\Vert \M_{l+1}-\M^{*}\Vert_{\mathrm{F}}&= \Vert  {\rm SVD}_{r}(\M_{l} - \eta_{l}\mathcal{P}_{\mathbb{T}_{l}}(\mathbf{G}_{l}))- \M^{*}\Vert_{\mathrm{F}}\\
	&\leq \fro{\M_l-\eta\calP_{\TT_l}(\G_l)-\M^*}+\frac{40}{\sigma_{r}}\fro{\M_l-\eta\calP_{\TT_l}(\G_l)-\M^*}^2\\
	&\leq 	(1-\frac{1}{16}\frac{\mus^2}{\Ls^2}) \fro{\M_{l} -\M^{*}}+\frac{40}{\sigma_{r}}(1-\frac{1}{8}\frac{\mus^2}{\Ls^2})^{2} \fro{\M_{l} -\M^{*}}^2\\
	&\leq (1-\frac{1}{32}\frac{\mus^2}{\Ls^2}) \fro{\M_{l} -\M^{*}},
\end{align*}
which uses the initialization condition $40\fro{\M_{l} -\M^{*}}\leq40\fro{\M_{0} -\M^{*}} \leq\frac{1}{32}\cdot\frac{\mus^2}{\Ls^2}\sigma_{r}$. It finishes the proof.

\section{Proofs of Low-rank Linear Regression}
\label{sec:proof-app}
\subsection{Proof of Theorem~\ref{thm:Gaussian-l1}}\label{proof:lem:Gaussian-l1} The following lemma establishes the two-phase regularity conditions of the absolute loss under Gaussian noise. 

\begin{lemma}\label{lem:Gaussian-l1}
	Assume $\xi_1,\dots, \xi_n \stackrel{i.i.d.}{\sim} N(0,\sigma^2)$ and $\{{\rm vec}(\X_i)\}_{i=1}^n$ satisfy Assumption~\ref{assump:sensing operators:vec}.  There exist absolute constants $C_1, C_2,  C_3,  c_1>0$ such that if $n\geq C_1rd_1\ku\kl^{-1}$ and $\BB_1$, $\BB_2$ are given by
	\begin{align*}
		&\BB_1:=\left\{\M\in\MM_r: \|\M-\M^{\ast}\|_{\rm F}\geq \sqrt{\kl^{-1}}\sigma\right\},\\
		&\BB_2:=\left\{\M\in\MM_r: C_2\sigma\left(\frac{\ku}{\kl^2}\cdot\frac{rd_1}{n}\right)^{1/2}\leq \|\M-\M^{\ast}\|_{\rm F}\leq \sqrt{\kl^{-1}}\sigma\right\},
	\end{align*}
	then with probability at least $1-\exp(-c_1rd_1)-3\exp(-\sqrt{n}/\log n)$, the absolute loss $f(\M)=\sum_{i=1}^n |Y_i-\langle \M,  \X_i\rangle |$ satisfies Condition~\ref{assump:two-phase}:
	\begin{enumerate}[(1)]
		\item the rank-$r$ restricted \textbf{two-phase sharpness} with respect to $\M^*$,
		$$
		f(\M)-f(\M^{\ast})\geq 
		\begin{cases}
			\frac{n}{12}\sqrt{\kl} \|\M-\M^{\ast}\|_{\rm F}, & \textrm{ for }\ \  \M\in\BB_1; \\
			\frac{n}{12\sigma}\kl \|\M-\M^{\ast}\|_{\rm F}^2, & \textrm{ for }\ \  \M\in\BB_2;
		\end{cases}
		$$
		
		\item the rank-r restricted \textbf{two-phase sub-gradient bound} with respect to $\M^{\ast}$,
		$$
		\|\G\|_{\rm F,  r}\leq 
		\begin{cases}
			2n\sqrt{\ku}, & \textrm{ for }\ \ \M\in\BB_1;\\
			C_3n\sigma^{-1}\ku\|\M-\M^{\ast}\|_{\rm F},& \textrm{ for }\ \ \M\in\BB_2,
		\end{cases}
		$$
		where $\G\in\partial f(\M)$ is any sub-gradient.
	\end{enumerate}
\end{lemma}

By putting together Lemma~\ref{lem:Gaussian-l1} and  Proposition~\ref{prop:main}, we  immediately obtain the convergence and statistical performance of RsGrad Algorithm~\ref{alg:RsGrad}.  The proof is a combination of Proposition~\ref{prop:main} and Lemma~\ref{lem:Gaussian-l1},  and is hence omitted.  

\begin{proof}[Proof of Lemma~\ref{lem:Gaussian-l1}]
	We shall proceed assuming the event $\bcalE = \{\sup_{\substack{\Delta\M\in\MM_{2r} \\ \M\in\RR^{d_1\times d_2} }}|f(\M+\Delta\M)-f(\M) - \EE(f(\M+\Delta\M)-f(\M))|\cdot\fro{\Delta\M}^{-1}\leq C_2\sqrt{nd_1r\ku} \}$ holds. Specifically, Theorem \ref{thm:empirical process} proves $$\PP(\bcalE)\geq1-\exp(-cd_1r/4)-3\exp(-\sqrt{n}/\log n). $$
	First consider upper bound of partial Frobenius norm for the sub-gradient $\fror{\G}$ in phase one, where $\G\in \partial f(\M)$ and $\M$ has rank at most $r$. According to event $\bcalE$, it has
	\begin{align*}
		&{~~~~}f(\M+\text{SVD}_r(\G))-f(\M)\\
		&\leq \EE f(\M+\text{SVD}_r(\G))-\EE f(\M)+C_2\sqrt{nd_1r\ku}\fro{\text{SVD}_r(\G)}\\
		&=\EE \sum_{i=1}^{n}\vert Y_i-\inp{\M +\text{SVD}_r(\G)}{\X_i}\vert - \EE\sum_{i=1}^{n}\vert Y_i-\inp{\M}{\X_i}\vert +C_2\sqrt{nd_1r\ku}\fro{\text{SVD}_r(\G)}.
	\end{align*}
	Then by triangular inequality, we have
	\begin{align*}
		f(\M+\text{SVD}_r(\G))-f(\M)\leq \sum_{i=1}^{n}\EE\vert\inp{\text{SVD}_r(\G)}{\X_i}\vert+C_2\sqrt{nd_1r\ku}\fro{\text{SVD}_r(\G)}.
	\end{align*}
	Moreover, Corollary~\ref{cor:l1expecation noiseless} proves $\sum_{i=1}^{n}\EE\vert\inp{\text{SVD}_r(\G)}{\X_i}\vert\leq n\sqrt{\ku}\fro{\text{SVD}_r(\G)}$. Thus, together with sample size condition $n\geq c\sqrt{\ku\kl^{-1}}d_1r\geq cd_1r$, we obtain \begin{align*}
		f(\M+\text{SVD}_r(\G))-f(\M)\leq 2n\sqrt{\ku}\fro{\text{SVD}_r(\G)}.
	\end{align*} Then, recall definition of partial Frobenius norm $\fro{\text{SVD}_r(\G)}=\fror{\G}$; see Lemma~\ref{teclem:partial F norm}, and it leads to 
	\begin{align}
		f(\M+\text{SVD}_r(\G))-f(\M)\leq2n\sqrt{\ku}\fror{\G}.
		\label{eq:31}
	\end{align}
	On the other hand, by definition of sub-gradient, we have $$f(\M+\text{SVD}_r(\G))-f(\M)\geq \inp{\text{SVD}_r(\G)}{\G}=\fro{\text{SVD}_r(\G)}^2=\fror{\G}^2,$$ and together with Equation~\eqref{eq:31}, we have $$\fror{\G}\leq 2n\sqrt{\ku},$$ which implies $\Lc=2n\sqrt{\ku}$.
	
	Next consider the lower bound for $f(\M) - f(\M^*)$. We discuss the two phases respectively, namely, phase one when $\fro{\M-\M^*}\geq \sqrt{\kl^{-1}}\sigma$ and phase two when $C\sqrt{\frac{\ku}{\kl^2}\cdot\frac{d_1 r}{n}}\sigma\leq\fro{\M-\M^*}< \sqrt{\kl^{-1}}\sigma$. 
	
	\noindent\textit{Case 1:}
	When $\Vert \mathbf{M}-\mathbf{M}^{*}\Vert_{\mathrm{F}}\geq\sqrt{\kl^{-1}}\sigma=\tauc$, under event $\bcalE$, we have 
	\begin{align*}
		f(\mathbf{M})-f(\mathbf{M}^*)&\geq\mathbb{E}\left[f(\mathbf{M})-f(\mathbf{M}^*) \right] - C_2\sqrt{nd_1r\ku}\Vert \mathbf{M}-\mathbf{M}^{*}\Vert_{\mathrm{F}}\\
		&\geq n\sqrt{2/\pi}\frac{\kl}{\sqrt{\kl\Vert \mathbf{M}-\mathbf{M}^{*}\Vert_{\mathrm{F}}^{2}+\sigma^{2}}+\sigma}\Vert \mathbf{M}-\mathbf{M}^{*}\Vert_{\mathrm{F}}^{2} - C_2\sqrt{nd_1r\ku}\Vert \mathbf{M}-\mathbf{M}^{*}\Vert_{\mathrm{F}}\\
		&\geq n\sqrt{2/\pi}\frac{\sqrt{\kl}}{\sqrt{2}+1}\Vert \mathbf{M}-\mathbf{M}^{*}\Vert_{\mathrm{F}} -C_2\sqrt{nd_1r\ku}\Vert \mathbf{M}-\mathbf{M}^{*}\Vert_{\mathrm{F}}\\
		&\geq \frac{1}{12}n\sqrt{\kl}\Vert \mathbf{M}-\mathbf{M}^{*}\Vert_{\mathrm{F}},
	\end{align*}
	where the second line uses $\EE f(\M)-\EE f(\M^*)$, see Lemma~\ref{teclem:l1expectation}, the penultimate line is from $\Vert \mathbf{M}-\mathbf{M}^{*}\Vert_{\mathrm{F}}\geq\sqrt{\kl^{-1}}\sigma$ and the last line uses $n\geq Cd_1r\ku\kl^{-1}$ for some absolute constant $C>0$. It shows with $\taus=\sqrt{\kl^{-1}}\sigma$,  $\muc=\frac{1}{12}n\sqrt{\kl}$.

	\noindent\textit{Case 2:}
	When $\taus=C\sqrt{\frac{\ku}{\kl^2}\cdot\frac{d_1 r}{n}}\sigma\leq\Vert \mathbf{M}-\mathbf{M}^{*}\Vert_{\mathrm{F}}<\sqrt{\kl^{-1}}\sigma=\tauc$, under event $\bcalE$, we have \begin{equation*}
		\begin{split}
			f(\mathbf{M})-f(\mathbf{M}^*)&\geq\mathbb{E}\left[f(\mathbf{M})-f(\mathbf{M}^*) \right] - C_2\sqrt{nd_1r\ku}\Vert \mathbf{M}-\mathbf{M}^{*}\Vert_{\mathrm{F}}\\
			&\geq n\sqrt{2/\pi}\frac{\kl}{\sqrt{\kl\Vert \mathbf{M}-\mathbf{M}^{*}\Vert_{\mathrm{F}}^{2}+\sigma^{2}}+\sigma}\Vert \mathbf{M}-\mathbf{M}^{*}\Vert_{\mathrm{F}}^{2} - C_2\sqrt{nd_1r\ku}\Vert \mathbf{M}-\mathbf{M}^{*}\Vert_{\mathrm{F}}\\
			&\geq\frac{1}{6}\frac{n}{\sigma}\kl\Vert \mathbf{M}-\mathbf{M}^{*}\Vert_{\mathrm{F}}^{2} -C_2\sqrt{nd_1r\ku}\Vert \mathbf{M}-\mathbf{M}^{*}\Vert_{\mathrm{F}}\\
			&\geq \frac{1}{12}\frac{n}{\sigma}\kl\Vert \mathbf{M}-\mathbf{M}^{*}\Vert_{\mathrm{F}}^{2},
		\end{split}
	\end{equation*}
	where the penultimate line uses $\sqrt{\kl}\Vert \mathbf{M}-\mathbf{M}^{*}\Vert_{\mathrm{F}}< \sigma$ and the last line is from $\fro{\M-\M^*}\geq C\sqrt{\frac{\ku}{\kl^2}\cdot\frac{d_1 r}{n}}\sigma$. It shows $\mus=\frac{1}{12}\frac{n}{\sigma}\kl$ under $\taus=C\sqrt{\frac{\ku}{\kl^2}\cdot\frac{d_1 r}{n}}\sigma$.
	
	Finally, the following Lemma \ref{lemma:upperboundsubgradient:gaussian} proves under the event $\bcalE$, $\Ls=Cn\sigma^{-1}\ku$ and thus proof of Lemma~\ref{lem:Gaussian-l1} completes.
\end{proof}

\begin{lemma}[Upper bound for sub-gradient]\label{lemma:upperboundsubgradient:gaussian}
	Let $\M\in\MM_r$ satisfy $\fro{\M-\M^*}\geq \taus:=\sqrt{\frac{d_1r}{n}\cdot\frac{\ku}{\kl^2}}\sigma$. Let $\G\in\partial f(\M)$ be any sub-gradient. Under the event $\bcalE$, we have $\fror{\G}\leq Cn\sigma^{-1}\ku\fro{\M-\M^*}$ for some absolute constant $C>0$.
\end{lemma}
\begin{proof}[Proof of Lemma~\ref{lemma:upperboundsubgradient:gaussian}]
	Take $\M_1 = \M + \frac{\sigma}{2n\ku}\text{SVD}_r(\G)$, where $\text{SVD}_r(\G)$ is the best rank $r$ approximation of $\G$. So $\rank(\M-\M_1) \leq r$. Then by Lemma~\ref{teclem:l1expectation}, we have
	\begin{align*}
		\EE f(\M_{1})-\EE f(\M)\leq& n\sqrt{2/\pi}\sqrt{\sigma^2+\ku\fro{\M_1-\M^*}^2}-n\sqrt{2/\pi}\sqrt{\sigma^2+\kl\fro{\M-\M^*}^2}\\
		=&\sqrt{2/\pi}n\frac{\ku\Vert \M_{1}-\M^{*}\Vert_{\mathrm{F}}^{2} -\kl \Vert \M-\M^{*}\Vert_{\mathrm{F}}^{2} }{\sqrt{\sigma^{2} + \ku\Vert \M_{1}-\M^{*}\Vert_{\mathrm{F}}^{2} } + \sqrt{\sigma^2 + \kl\Vert \M-\M^{*}\Vert_{\mathrm{F}}^{2} } }\\
		=&\sqrt{2/\pi}n\frac{\kl\Vert \M_{1}-\M\Vert_{\mathrm{F}}^{2}+2\kl\langle \M-\M^{*}, \M_{1}-\M\rangle+(\ku-\kl)\fro{\M_1-\M^*}^2 }{\sqrt{\sigma^{2} + \ku\Vert \M_{1}-\M^{*}\Vert_{\mathrm{F}}^{2} } + \sqrt{\sigma^2 +\kl \Vert \M-\M^{*}\Vert_{\mathrm{F}}^{2} } }\\
		\leq&\frac{n}{\sigma}\ku\left( \Vert \M_{1}-\M\Vert_{\mathrm{F}}^{2}+2\Vert \M-\M^{*}\Vert_{\mathrm{F}} \Vert\M_{1}-\M\Vert_{\mathrm{F}}+\fro{\M-\M^*}^2\right),
	\end{align*}
	Besides, event $\bcalE$ shows
	\begin{align*}
		f(\M_{1}) - f(\M)\leq&\EE f(\M_{1})-\EE f(\M)+C_1\sqrt{ nd_1r\ku}\fro{\M_1-\M}\\
		\leq&\EE f(\M_{1})-\EE f(\M)+C_1\frac{n}{\sigma}\ku\fro{\M-\M^*} \fro{\M_1-\M}\\
		\leq& \frac{n}{\sigma}\ku\Vert \M_{1}-\M\Vert_{\mathrm{F}}^{2} + (2+C_1) \frac{n}{\sigma}\ku\Vert \M-\M^{*}\Vert_{\mathrm{F}} \Vert\M_{1}-\M\Vert_{\mathrm{F}} +\frac{n}{\sigma}\ku\Vert \M-\M^{*}\Vert_{\mathrm{F}}^{2},
	\end{align*}
	where the second inequality is from the condition $\fro{\M-\M^*}\geq\taus\geq \sqrt{\frac{d_1r}{n\ku}}\sigma$. Then insert $\M_1 = \M + \frac{\sigma}{2n\ku}\text{SVD}_r(\G)$ into the equation and then we obtain
	\begin{equation}\label{eq11}
		\begin{split}
			f(\M+\frac{\sigma}{2n\ku}\text{SVD}_{r}(\G))&-f(\M)\\&\leq \frac{\sigma}{4n\ku} \Vert\G\Vert_{\mathrm{F,r}}^{2} + (1+C_1/2) \Vert \M-\M^{*}\Vert_{\mathrm{F}} \Vert\G\Vert_{\mathrm{F,r}}+\frac{n}{\sigma}\ku\Vert \M-\M^{*}\Vert_{\mathrm{F}}^{2}.
		\end{split}
	\end{equation}
	On the other hand, by the definition of sub-gradient, we have
	\begin{align}\label{eq12}
		f(\M+\frac{\sigma}{2n\ku}\text{SVD}_{r}(\G))-f(\M)\geq\inp{\G}{\frac{\sigma}{2n\ku}\text{SVD}_{r}(\G) }= \frac{\sigma}{2n\ku}\fror{\G}^2.
	\end{align}
	Combine Equation~\eqref{eq11} with Equation~\eqref{eq12} and by solving the quadratic inequality we get $$\fror{\G}\leq Cn\sigma^{-1}\ku\fro{\M-\M^*}.$$
\end{proof}

\subsection{Proof of Theorem~\ref{thm:heavytail-l1}}
The following lemma establishes the two-phase regularity conditions of the absolute loss under heavy-tailed noise.

\begin{lemma}\label{lem:heavytail-l1}
	Assume $\{\xi_i\}_{i=1}^n$ and $\{{\rm vec}(\X_i)\}_{i=1}^n$ satisfy Assumptions~\ref{assump:heavy-tailed} and \ref{assump:sensing operators:vec}, respectively.  There exist absolute constants $C_1, C_2, C_3,c_1>0$ such that if $n\geq C_1rd_1\ku\kl^{-1}$ and $\BB_1$, $\BB_2$ are given by
	\begin{align*}
		&\BB_1:=\left\{\M\in\MM_r: \|\M-\M^{\ast}\|_{\rm F}\geq 8\sqrt{\kl^{-1}}\gamma\right\},\\
		&\BB_2:=\left\{\M\in\MM_r: C_2b_0\left(\frac{\ku}{\kl^2}\cdot\frac{rd_1}{n}\right)^{1/2}\leq\|\M-\M^{\ast}\|_{\rm F}< 8\sqrt{\kl^{-1}}\gamma\right\},
	\end{align*}
	then with probability at least $1-\exp(-c_1rd_1)-3\exp(-\sqrt{n}/\log n)$, the absolute loss $f(\M)=\sum_{i=1}^n|Y_i-\langle \M, \X_i\rangle|$ satisfies Condition~\ref{assump:two-phase}:
	\begin{enumerate}[(1)]
		\item the rank-$r$ restricted \textbf{two-phase sharpness} with respect to $\M^{\ast}$,
		$$
		f(\M)-f(\M^{\ast})\geq 
		\begin{cases}
			\frac{n}{4}\sqrt{\kl} \|\M-\M^{\ast}\|_{\rm F}, & \textrm{ for }\ \  \M\in\BB_1; \\
			\frac{n}{12b_0}\kl \|\M-\M^{\ast}\|_{\rm F}^2, & \textrm{ for }\ \  \M\in\BB_2;
		\end{cases}
		$$
		
		\item the rank-r restricted \textbf{two-phase sub-gradient bound} with respect to $\M^{\ast}$,
		$$
		\|\G\|_{\rm F,  r}\leq 
		\begin{cases}
			2n\sqrt{\ku}, & \textrm{ for }\ \ \M\in\BB_1;\\
			C_3nb_1^{-1}\ku\|\M-\M^{\ast}\|_{\rm F},& \textrm{ for }\ \ \M\in\BB_2,
		\end{cases}
		$$
		where $\G\in\partial f(\M)$ is any sub-gradient.
	\end{enumerate}
\end{lemma}

The proof of Theorem~\ref{thm:heavytail-l1} is a combination of Proposition~\ref{prop:main} and Lemma~\ref{lem:heavytail-l1}. Thus we only prove Lemma~\ref{lem:heavytail-l1} here. 

\begin{proof}[Proof of Lemma~\ref{lem:heavytail-l1}]
	Suppose the event $\bcalE = \{\sup_{\substack{\Delta\M\in\MM_{2r} \\ \M\in\RR^{d_1\times d_2} }}|f(\M+\Delta\M)-f(\M) - \EE(f(\M+\Delta\M)-f(\M))|\cdot\fro{\Delta\M}^{-1}\leq C_2\sqrt{nd_1r\ku} \}$ holds. Specifically, Theorem \ref{thm:empirical process} proves $$\PP(\bcalE)\geq1-\exp(-cd_1r/4)-3\exp(-\sqrt{n}/\log n). $$ 
	The $\Lc=2n\sqrt{\ku}$ proof is the same as the one in the proof of Lemma~\ref{lem:Gaussian-l1}. Then consider lower bound of $f(\M)-f(\M^*)$.
	
	\noindent\textit{Case 1:} When $\fro{\M-\M^*}\geq8\sqrt{\kl^{-1}}\gamma=:\tauc$, by event $\bcalE$, we have
	\begin{align*}
		f(\M)-f(\M^*)&\geq\EE f(\M)-\EE f(\M^*)-C_2\sqrt{nd_1r\ku}\fro{\M-\M^*}\\
		&=\sum_{i=1}^{n}\EE\left(\left|\inp{\M-\M^*}{\X_{i}}-\xi_{i}\right|-\left|\xi_{i}\right|\right)-C_2\sqrt{nd_1r\ku}\fro{\M-\M^*}\\
		&\geq \sum_{i=1}^{n}\EE\left|\inp{\M-\M^*}{\X_{i}}\right|-2\sum_{i=1}^{n}\EE|\xi_i|-C_2\sqrt{nd_1r\ku}\fro{\M-\M^*}\\
		&\geq n\sqrt{\frac{2\kl}{\pi}}\fro{\M-\M^*}-2n\gamma-C_2\sqrt{nd_1r\ku}\fro{\M-\M^*},
	\end{align*}
	where the third line is from triangular inequality and the last line is from Corollary~\ref{cor:l1expecation noiseless}. Then combine with $\fro{\M-\M^*}\geq 8\sqrt{\kl^{-1}}\gamma$ and $n\geq cd_1r\ku\kl^{-1}$, and $f(\M)-f(\M^*)$ could be further bounded with 
	\begin{align*}
		f(\M)-f(\M^*)\geq\frac{n}{4}\sqrt{\kl}\fro{\M-\M^*},
	\end{align*}
	which verifies $\muc = \frac{n}{4}\sqrt{\kl}$ under $\tauc=8\sqrt{\kl^{-1}}\gamma$.
	
	\noindent\textit{Case 2:} Then consider the second phase when $C_2\sqrt{\frac{\ku}{\kl^2}\cdot\frac{d_1r}{n}}b_0 =\taus\leq \fro{\M-\M^*}< \tauc=8\sqrt{\kl^{-1}}\gamma$. To bound expectation of $f(\M)-f(\M^*)$, we first consider the conditional expectation $\EE\left[f(\M)-f(\M^*)\bigg|\X_1,\dots,\X_n\right]$. Notice for all $t_0\in\RR$,
	\begin{equation}
		\begin{split}
			\EE|\xi-t_0| &= \int_{s\geq t_0}(s-t_0)\,dH_{\xi}(s) + \int_{s< t_0}(t_0-s)\,dH_{\xi}(s)\\
			&=2\int_{s\geq t_0}(s-t_0)\,dH_{\xi}(s) + \int_{-\infty}^{+\infty}(t_0-s)\,dH_{\xi}(s)\\
			&=2\int_{s\geq t_0}(1-H_{\xi}(s))\,ds + t_0 - \int_{-\infty}^{\infty}s \,dH_{\xi}(s),
		\end{split}
		\label{eq:l1condiexp}
	\end{equation}
	where the last line is due to \begin{align*}
		\int_{s\geq t_0}(s-t_0)\,dH_{\xi}(s)&=-\int_{s\geq t_0} (s-t_0)\, d(1-H_\xi(s))\\
		&=-(s-t_0)(1-H_\xi(s))\bigg|_{t_0}^{+\infty}+\int_{s\geq t_0}\left(1-H_\xi
		(s)\right)\, ds\\
		&=\lim_{s\to+\infty}(s-t_0)(1-H_\xi(s))+\int_{s\geq t_0}\left(1-H_\xi
		(s)\right)\, ds,
	\end{align*} whose first term vanishes
	\begin{align*}
		0\leq\lim_{s\to+\infty}(s-t_0)(1-H_\xi(s))\leq\lim_{s\to+\infty} (s-t_0)\cdot\PP(\xi>s)\leq\lim_{s\to+\infty} (s-t_0)\cdot\frac{\EE|\xi|^{1+\eps}}{s^{1+\eps}}=0.
	\end{align*}
	Set $t_0 = \inp{\X_i}{\M-\M^*}$ and then Equation~\eqref{eq:l1condiexp} becomes
	\begin{align*}
		\EE\left[\left|\xi-\inp{\X_i}{\M-\M^*}\right|\big| \X_i\right] = 2\int_{s\geq \inp{\X_i}{\M-\M^*}}(1-H_{\xi}(s))\,ds + \inp{\X_i}{\M-\M^*} - \int_{-\infty}^{+\infty}s dH_{\xi}(s).
	\end{align*}
	In the meantime, with $t_0 = 0$, Equation~\eqref{eq:l1condiexp} becomes $$\EE\left|\xi\right|=2\int_{s\geq 0}(1-H_{\xi}(s))\, ds-\int_{-\infty}^{+\infty} s\, dH_{\xi}(s).$$ Take the difference of the above two equations and it yields
	\begin{align}
		\EE\left[\left|\xi-\inp{\X_i}{\M-\M^*}\right|\big| \X_i\right]-\EE\left|\xi\right| = 2\int_{0}^{\inp{\X_i}{\M-\M^*}}H_{\xi}(s)ds -  \inp{\X_i}{\M-\M^*}.
		\label{eq14}
	\end{align}
	Note that $\langle \mathbf{X}_{i}, \mathbf{M}-\mathbf{M}^{*}\rangle\sim N(0,\text{vec}(\M-\M^*)^{\top}\bSigma_i\text{vec}(\M-\M^*)) $ and denote $z_i:=\langle \mathbf{X}_{i}, \mathbf{M}-\mathbf{M}^{*}\rangle$, $\sigma_{i}^2:=\text{vec}(\M-\M^*)^{\top}\bSigma_i\text{vec}(\M-\M^*)$. Note that Assumption~\ref{assump:sensing operators:vec}, conditions of sensing matrices, implies \begin{align*}
		\kl\fro{\M-\M^*}^2\leq\sigma_{i}^2\leq\ku\fro{\M-\M^*}^2
	\end{align*}
	Denote the density of $z_i$ as $f_{z_i}(\cdot)$. Take expectation of $z_i$ on both sides of Equation~\eqref{eq14}
	\begin{align*}
		\mathbb{E}\left[\left|\xi_i-\inp{\X_i}{\M-\M^*}\right|-\left|\xi_i\right|\right]
		&=2\int_{-\infty}^{+\infty}\int_{0}^{t}\left(H_{\xi}(s)-0.5\right)f_{z_i}(t)\,ds\, dt\\
		&=2\int_{-\infty}^{+\infty}\int_{0}^{t}\int_{0}^{s}h_{\xi}(w)f_{z_i}(t)\,dw\,ds\,dt.
	\end{align*}
	Notice that by conditions on sensing matrices Assumption~\ref{assump:sensing operators:vec}, we have $\sigma_i\leq\sqrt{\ku}\fro{\M-\M^*}\leq 8\sqrt{\ku\kl^{-1}}\gamma$. Hence, when $w\in(-\sigma_i,\sigma_i)$, by Assumption~\ref{assump:heavy-tailed}, we have $h_{\xi}(w)\geq b_0^{-1}$. Thus the above equation could be further bounded with
	\begin{align*}
		\mathbb{E}\left[\left|\xi_i-\inp{\X_i}{\M-\M^*}\right|-\left|\xi_i\right|\right]&\geq2\int_{-\sigma_{i}}^{\sigma_{i}}\int_{0}^{t}f_{z_i}(t)b_0^{-1}s\,ds\,dt\\
		&=b_0^{-1}\int_{-\sigma_{i}}^{\sigma_{i}}t^2f_{z_i}(t)\,dt = b_0^{-1}\sigma_{i}^2\int_{-1}^{1}t^2\cdot \frac{1}{\sqrt{2\pi}}e^{-t^2/2}\,dt\\
		&\geq\frac{1}{6b_0}\sigma_{i}^2,
	\end{align*}
	where the last line is from $\int_{-1}^1t^2\cdot\frac{1}{\sqrt{2\pi}}e^{-t^2/2}\,dt\geq 1/6$. Therefore, sum over $i=1.\dots,n$ and then it has
	\begin{align*}
		\EE f(\M)- \EE f(\M^*) = \sum_{i=1}^n  \mathbb{E}\left[\left|\xi_i-\inp{\X_i}{\M-\M^*}\right|-\left|\xi_i\right|\right] \geq \frac{1}{6b_0}\sum_{i=1}^n\sigma_{i}^2\geq \frac{n}{6b_0}\kl\fro{\M-\M^*}^2.
	\end{align*}
	Invoke event $\bcalE$,
	\begin{align*}
		f(\M)-f(\M^*)&\geq \EE[f(\M)-f(\M^*)] - C\sqrt{nd_1r}\fro{\M-\M^*}\\
		&\geq \frac{n}{6b_{0}}\kl \fro{\M-\M^*}^2 - C\sqrt{nd_1r\ku}\fro{\M-\M^*}\\
		&\geq\frac{n}{12b_0}\kl\fro{\M-\M^*}^2,
	\end{align*}
	where the last inequality uses $\fro{\M-\M^{*}}\geq C_2\sqrt{\frac{d_1r}{n}\cdot \frac{\ku}{\kl^2}}b_{0}$. This proves $\mus=\frac{n}{12b_0}\kl$.
	
	Finally, from the following lemma, we see that $\Ls \leq C_3nb_1^{-1}\ku$ when $\bcalE$ holds.
	And this finishes the proof.
\end{proof}

\begin{lemma}[Upper bound for sub-gradient]\label{lemma:upperboundsubgradient:heavytail}
	Let $\M\in\MM_r$ satisfy $\fro{\M-\M^*}\geq \taus:=C_1\sqrt{\frac{d_1r}{n}\cdot\frac{\ku}{\kl^2}}b_1$. Let $\G\in\partial f(\M)$ be any sub-gradient. Under event $\bcalE$, we have $$\fror{\G}\leq C_2nb_{1}^{-1}\ku\fro{\M-\M^*},$$ where $C_1,C_2>0$ are absolute constants.
\end{lemma}
\begin{proof}[Proof of Lemma~\ref{lemma:upperboundsubgradient:heavytail}]
	Take $\M_1 = \M + \frac{b_1}{2n\ku}\text{SVD}_r(\G)$, where $\text{SVD}_r(\G)$ is the best rank $r$ approximation of $\G$. We finish the proof via combining lower bound and upper bound of $f(\M_{1})-f(\M)$. 
	
	First consider $\EE f(\M_1)-\EE f(\M)$.
	Use the same notation as the above proof and similarly,
	\begin{align*}
		&\EE\left[\left|\xi_i-\inp{\X_i}{\M_1-\M^*}\right|- \left|\xi_i-\inp{\X_i}{\M-\M^*}\right|\bigg|\X_i\right]\\
		=&2\int_{\langle \X_{i}, \M-\M^{*}\rangle}^{\langle \X_{i}, \M-\M^{*}\rangle+\langle \X_{i}, \M_{1}-\M\rangle}\int_{0}^{\xi} h_{\xi}(x)\,dx\,d\xi\\
		\leq& 2b_1^{-1}\left|\int_{\langle \X_{i}, \M-\M^{*}\rangle}^{\langle \X_{i}, \M-\M^{*}\rangle+\langle \X_{i}, \M_{1}-\M\rangle}\xi \,d\xi\right|\\
		=&b_{1}^{-1}\left|(\langle \X_{i}, \M-\M^{*}\rangle+\langle \X_{i}, \M_{1}-\M\rangle )^2 - (\langle \X_{i}, \M-\M^{*}\rangle)^2\right|,
	\end{align*}
	where the inequality follows from the upper bound for $h_{\xi}(x)$. Then take expectation with respect to $\X_i$ on each side and sum up over $i$ :
	\begin{align*}
		\EE f(\M_{1}) - \EE f(\M) =& \sum_{i=1}^{n}\EE \left[\left|\xi_i-\inp{\X_i}{\M_1-\M^*}\right|- \left|\xi_i-\inp{\X_i}{\M-\M^*}\right|\bigg|\X_i\right]\\
		\leq&b_1^{-1}\sum_{i=1}^{n} \EE\left|(\langle \X_{i}, \M-\M^{*}\rangle+\langle \X_{i}, \M_{1}-\M\rangle )^2 - (\langle \X_{i}, \M-\M^{*}\rangle)^2\right|\\
		\leq&nb_1^{-1}\ku\left[ \Vert \M_{1} -\M\Vert_{\mathrm{F}}^{2} + 2\Vert \M_{1} -\M\Vert_{\mathrm{F}} \Vert \M -\M^{*}\Vert_{\mathrm{F}}\right],
	\end{align*}
	where the last inequality uses Holder's inequality. Then, bound $f(\M_{1})-f(\M)$ under event $\bcalE$, we have
	\begin{align*}
		f(\M_{1}) - f(\M)\leq&\EE f(\M_{1})-\EE f(\M)+C_1\sqrt{ nd_1r\ku}\fro{\M_1-\M}\\
		\leq&\EE f(\M_{1})-\EE f(\M)+C_1nb_1^{-1}\ku\fro{\M-\M^*}\fro{\M_1-\M}\\
		\leq& nb_1^{-1}\ku \Vert \M_{1}-\M\Vert_{\mathrm{F}}^{2} + (2+C_1) nb_1^{-1}\ku\Vert \M-\M^{*}\Vert_{\mathrm{F}} \Vert\M_{1}-\M\Vert_{\mathrm{F}},
	\end{align*}
	where the second inequality is from the condition $\fro{\M-\M^*}\geq\taus= C_2\sqrt{\frac{d_1r}{n}\cdot\frac{\ku}{\kl^2}}b_0\geq C_2\sqrt{\frac{d_1r}{n}}\sqrt{\ku^{-1}}b_1$. Insert $\M_1 = \M + \frac{b_1}{2n\ku}\text{SVD}_r(\G)$ and it becomes
	\begin{equation}
		\label{eq18}
		\begin{split}
			f(\M+\frac{b_1}{2n\ku}\text{SVD}_{r}(\G))&-f(\M)\\&\leq \frac{b_1}{4n\ku} \Vert\G\Vert_{\mathrm{F,r}}^{2} + (1+C_1/2) \Vert \M-\M^{*}\Vert_{\mathrm{F}} \Vert\G\Vert_{\mathrm{F,r}}.
		\end{split}
	\end{equation}
	On the other hand, by the definition of sub-gradient, $f(\M_{1})-f(\M)$ has lower bound
	\begin{align}\label{eq17}
		f(\M+\frac{b_1}{2n\ku}\text{SVD}_{r}(\G))-f(\M)\geq\inp{\frac{b_1}{2n\ku}\text{SVD}_r(\G)}{\G}= \frac{b_1}{2n\ku}\fror{\G}^2.
	\end{align}
	Combine Equation~\eqref{eq18} with Equation~\eqref{eq17} and then solve the quadratic inequality which leads to $$\fror{\G}\leq C_2nb_{1}^{-1}\ku\fro{\M-\M^*}.$$
\end{proof}

\subsection{Proof of Theorem~\ref{thm:heavytail-huber}}
\label{proof:lem:heavytail-huber}

The following lemma establishes the two-phase regularity conditions of Huber loss under heavy-tailed noise. 

\begin{lemma}\label{lem:heavytail-huber}
	Assume $\{\xi_i\}_{i=1}^n$ and $\{{\rm vec}(\X_i)\}_{i=1}^n$ satisfy Assumptions~\ref{assump:heavytail-huber} and \ref{assump:sensing operators:vec}, respectively. There exist absolute constants $C_1,C_2, C_3,c_1>0$ such that if $n\geq C_1rd_1\ku\kl^{-1}$ and define $\BB_1$, $\BB_2$ to be
	\begin{align*}
		&\BB_1:=\left\{\M\in\MM_r: \|\M-\M^{\ast}\|_{\rm F}\geq 8\sqrt{\kl^{-1}}\gamma+2\sqrt{\kl^{-1}}\delta\right\},\\
		&\BB_2:=\left\{\M\in\MM_r: C_2b_0\left(\frac{\ku}{\kl^2}\cdot\frac{rd_1}{n}\right)^{1/2} \leq\|\M-\M^{\ast}\|_{\rm F}<8\sqrt{\kl^{-1}}\gamma+2\sqrt{\kl^{-1}}\delta \right\},
	\end{align*}
	then with probability at least $1-\exp(-c_1rd_1)-3\exp(-\sqrt{n}/\log n)$, the Huber loss $f(\M)=\sum_{i=1}^n \rho_{H,\delta}(Y_i-\langle \M,  \X_i\rangle)$ satisfies Condition~\ref{assump:two-phase}:
	\begin{enumerate}[(1)]
		\item the rank-$r$ restricted \textbf{two-phase sharpness} with respect to $\M^{\ast}$,
		$$
		f(\M)-f(\M^{\ast})\geq 
		\begin{cases}
			\frac{\delta n}{2}\sqrt{\kl} \|\M-\M^{\ast}\|_{\rm F}, & \textrm{ for }\ \  \M\in\BB_1; \\
			\frac{\delta n}{3b_0}\kl \|\M-\M^{\ast}\|_{\rm F}^2, & \textrm{ for }\ \  \M\in\BB_2;
		\end{cases}
		$$
		
		\item the  rank-r restricted \textbf{two-phase sub-gradient bound} with respect to $\M^{\ast}$,
		$$
		\|\G\|_{\rm F,  r}\leq 
		\begin{cases}
			4\delta n\sqrt{\ku}, & \textrm{ for }\ \ \M\in\BB_1;\\
			C_3\delta nb_1^{-1}\ku\|\M-\M^{\ast}\|_{\rm F},& \textrm{ for }\ \ \M\in\BB_2,
		\end{cases}
		$$
		where $\G\in\partial f(\M)$ is any sub-gradient.
	\end{enumerate}
	
\end{lemma}

The proof of Theorem~\ref{thm:heavytail-huber} is a combination of Proposition~\ref{prop:main} and Lemma~\ref{lem:heavytail-huber}. Thus we only prove Lemma~\ref{lem:heavytail-huber} here. 

\begin{proof}[Proof of Lemma~\ref{lem:heavytail-huber}]
	We shall proceed assuming the event $\bcalE= \{\sup_{\substack{\Delta\M\in\MM_{2r} \\ \M\in\RR^{d_1\times d_2} }}|f(\M+\Delta\M)-f(\M) - \EE(f(\M+\Delta\M)-f(\M))|\cdot\fro{\Delta\M}^{-1}\leq C_2\delta\sqrt{nd_1r\ku} \}$ holds. Specifically, Theorem \ref{thm:empirical process} proves $$\PP(\bcalE)\geq1-\exp(-cd_1r/4)-3\exp(-\sqrt{n}/\log n). $$
	First consider upper bound of sub-gradient $\G\in f(\M)$, where matrix $\M$ has rank at most $r$. By event $\bcalE$ and Lipschitz $2\delta$ continuity of $\rho_{H,\delta}(\cdot)$, we have
	\begin{align*}
		&{~~~~}f(\M+\text{SVD}_r(\G)) - f(\M)\\
		&= \sum_{i=1}^n\rho_{H,\delta}(\xi_i - \inp{\X_i}{\M+\text{SVD}_r(\G)}) -\rho_{H,\delta}(\xi_i - \inp{\X_i}{\M})\\
		&\leq\EE\sum_{i=1}^n\rho_{H,\delta}(\xi_i - \inp{\X_i}{\M+\text{SVD}_r(\G)}) -\rho_{H,\delta}(\xi_i - \inp{\X_i}{\M})+C_2\delta\sqrt{nd_1r\ku}\fro{\text{SVD}_r(\G)}\\
		&\leq \EE\sum_{i=1}^n 2\delta|\inp{\X_i}{\text{SVD}_r(\G)}|+C_2\delta\sqrt{nd_1r\ku}\fro{\text{SVD}_r(\G)}.
	\end{align*}
	Note that Corollary~\ref{cor:l1expecation noiseless} bounds $\left| \inp{\X_i}{\text{SVD}_r(\G)}\right|\leq n\sqrt{2/\pi}\sqrt{\ku}\fro{\text{SVD}_r(\G)}=n\sqrt{2/\pi}\sqrt{\ku}\fror{\G}$ and $n\geq cd_1r$. Thus the above equation could be upper bounded with \begin{align*}
		f(\M+\text{SVD}_r(\G)) - f(\M)\leq4\delta n\sqrt{\ku}\fro{\text{SVD}_r(\G)}=4\delta n\sqrt{\ku}\fror{\G}
	\end{align*} On the other hand, by definition of sub-gradient, it has
	\begin{align*}
		f(\M+\text{SVD}_r(\G)) - f(\M)\geq\inp{\text{SVD}_r(\G)}{\G}=\fror{\G}^2.
	\end{align*}
	Thus the above two equations lead to $$\fror{\G}^2\leq f(\M+\text{SVD}_r(\G)) - f(\M)\leq 4\delta n\sqrt{\ku}\fror{\G},$$
	which implies $\fror{\G}\leq4\delta n\sqrt{\ku}$, namely, $\Lc=4\delta n\sqrt{\ku}$.
	
	Then consider the lower bound of $f(\M)-f(\M^*)$.
	
	\noindent\textit{Case 1:}  When $\fro{\M-\M^*}\geq 8\gamma\sqrt{\kl^{-1}}+2\delta\sqrt{\kl^{-1}}$, use the inequality $2\delta|x| - \delta^2\leq\rho_{H,\delta}(x)\leq 2\delta\vert x\vert$. Under the event $\bcalE$, we have
	\begin{align*}
		f(\M)-f(\M^*) &\geq \EE f(\M)-\EE f(\M^*)-C_2\delta\sqrt{nd_1r\ku}\fro{\M-\M^*}\\
		&= \EE\sum_{i=1}^n\rho_{H,\delta}(\xi_i - \inp{\X_i}{\M-\M^*}) -\rho_{H,\delta}(\xi_i) - C_2\sqrt{nd_1r\ku}\fro{\M-\M^*}\\
		&\geq 2\delta \EE\sum_{i=1}^n\left(|\xi_i - \inp{\X_i}{\M-\M^*}| - |\xi_i|\right) - n\delta^2- C_2\delta\sqrt{nd_1r\ku}\fro{\M-\M^*}\\
		&\geq2\delta\sqrt{\frac{2}{\pi}}\sum_{i=1}^{n}\EE\left|\inp{\X_i}{\M-\M^*}\right| - 4\delta \sum_{i=1}^{n}\EE|\xi_i| - n\delta^2- C_2\delta\sqrt{nd_1r\ku}\fro{\M-\M^*}\\
		&\geq 2n\delta\sqrt{\frac{2\kl}{\pi}}\fro{\M-\M^*}-4\delta n\gamma-n\delta^2- C_2\delta\sqrt{nd_1r\ku}\fro{\M-\M^*},
	\end{align*}
	where the last line is from Corollary~\ref{cor:l1expecation noiseless} and $\EE|\xi_{i}|=\gamma$. Then with $\fro{\M-\M^*}\geq 8\gamma\sqrt{\kl^{-1}}+2\delta\sqrt{\kl^{-1}}$ and sample size $n\geq cd_1r\ku\kl^{-1}$, it has
	\begin{align*}
		f(\M)-f(\M^*) \geq \frac{1}{2}n\delta\sqrt{\kl}\fro{\M-\M^*}.
	\end{align*}
	It shows with $\tauc=8\gamma\sqrt{\kl^{-1}} + 2\delta\sqrt{\kl^{-1}}$, we have $\muc=\delta n\sqrt{\kl}/2$.
	
	\noindent\textit{Case 2: } When $C_2b_0\sqrt{\frac{\ku}{\kl^2}\cdot\frac{d_1r}{n}}\leq\fro{\M-\M^*}\leq 8\gamma\sqrt{\kl^{-1}}+2\delta\sqrt{\kl^{-1}}$,
	first consider expectation conditioned on $\X_i$. For any $z\in\RR$, it has
	\begin{equation}\label{eq19}
		\begin{split}
			&{~~~~~}\EE\left[\rho_{H,\delta}(\xi_{i}-z)\right]\\
			&= \int_{\delta+z}^{+\infty} 2\delta(x - z)-\delta^2 \,dH_{\xi}(x) + \int_{-\infty}^{z-\delta} 2\delta(z- x)-\delta^2 dH_{\xi}(x) + \int_{z-\delta}^{z+\delta}\left(z -x \right)^{2}dH_{\xi}(x).
		\end{split}
	\end{equation}
	Take derivative over variable $z$,
	\begin{align*}
		&{~~~~~}\frac{d}{d\, z}\EE\left[\rho_{H,\delta}(\xi_{i}-z)\right]\\ &=-\delta^2 h_{\xi}(\delta+z)-2\delta\int_{\delta+z}^{+\infty}\, dH_{\xi}(x)+\delta^2h_\xi(z-\delta)+2\delta\int_{-\infty}^{z-\delta}\, dH_{\xi}(x)\\
		&{~~~~~}+\delta^2h_{\xi}(z+\delta)-\delta^2h_{\xi}(z-\delta)-2\int_{z-\delta}^{z+\delta}(x-z)\, dH_{\xi}(x)\\
		&=-2\delta+2\delta H_{\xi}(z+\delta)+2\delta H_{\xi}(z-\delta)-2(x-z)H_{\xi}(x)\bigg|_{z-\delta}^{z+\delta}+2\int_{z-\delta}^{z+\delta} H_{\xi_{i}}(x)\, dx\\
		&=2\int_{z-\delta}^{z+\delta}H_{\xi}(x)\,dx - 2\delta,
	\end{align*}
	and take second order derivative,
	\begin{align*}
		\frac{d^2}{dz^2}\EE\left[\rho_{H,\delta}(\xi_{i}-z)\right] &= 2(H_{\xi}(z+\delta) - H_{\xi}(z-\delta)).
	\end{align*}
	Assumption~\ref{assump:heavytail-huber} guarantees $\frac{d}{d\, z}\EE\left[\rho_{H,\delta}(\xi_{i}-z)\right]\bigg|_{z=0} = 0$. Therefore the Taylor expansion at $0$ is
	\begin{align*}
		\EE\left[\rho_{H,\delta}(\xi_{i}-z)\right] = \EE\left[\rho_{H,\delta}(\xi_{i})\right] + 2\int_{0}^{z}(H_{\xi}(y+\delta) - H_{\xi}(y-\delta))y\,dy.
	\end{align*}
	Insert $z =\inp{\X_i}{\M-\M^*}\sim N(0,\text{vec}(\M-\M^*)^{\top}\bSigma_i\text{vec}(\M-\M^*))$ into the above equation,
	\begin{align*}
		\EE\left[\rho_{H,\delta}(\xi_{i}- \inp{\X_i}{\M-\M^*})\bigg| \X_i\right] - \EE\left[\rho_{H,\delta}(\xi_{i})\bigg| \X_i\right] = 2\int_{0}^{z}(H_{\xi}(y+\delta) - H_{\xi}(y-\delta))y\,dy.
	\end{align*}
	Denote  density of $\inp{\X_i}{\M-\M^*}$ by $f_{z_i}(\cdot)$ and denote its variance by $\sigma_{i}^2:=\text{vec}(\M-\M^*)^{\top}\bSigma_i\text{vec}(\M-\M^*)$. Take expectation w.r.t. $\X_i$ and we have 
	\begin{align*}
		\EE\left[\rho_{H,\delta}(\xi_{i}- \inp{\X_i}{\M-\M^*})\right] - \EE\left[\rho_{H,\delta}(\xi_{i})\right] &= 2\int_{-\infty}^{\infty}\int_{0}^t(H_{\xi}(y+\delta) - H_{\xi}(y-\delta))yf_{z_i}(t)\,dy\,dt\\
		&\geq 2\int_{-\sigma_{i}}^{\sigma_{i}}\int_{0}^tf_{z_i}(t)(H_{\xi}(y+\delta) - H_{\xi}(y-\delta))y\,dy\,dt.
	\end{align*}
	Notice that the conditions on sensing matrices Assumption~\ref{assump:sensing operators:vec} and phase-two region  imply $\sigma_i\leq\sqrt{\ku}\fro{\M-\M^*}\leq 8\gamma\sqrt{\ku\kl^{-1}}+2\delta\sqrt{\ku\kl^{-1}}$. Thus for $y\in(-\sigma_i,\sigma_i)$, by Assumption~\ref{assump:heavytail-huber}, it has $H_{\xi}(y+\delta)-H_{\xi}(y-\delta)\geq2\delta b_0^{-1}$. Hence, the above equation could be lower bounded with \begin{align*}
		\EE\left[\rho_{H,\delta}(\xi_{i}- \inp{\X_i}{\M-\M^*})\right] - \EE\left[\rho_{H,\delta}(\xi_{i})\right] &\geq 4\delta b_0^{-1}\int_{-\sigma_{i}}^{\sigma_{i}}t^2f_{z_i}(t)\,dt \\
		&= 4\delta b_0^{-1}\sigma_{i}^2\int_{-1}^{1} t^2 \frac{1}{\sqrt{2\pi}}\exp(-t^2/2)\,dt\\
		&\geq \frac{2}{3}\delta b_{0}^{-1}\sigma_{i}^2,
	\end{align*}
	where the last line uses $\int_{-1}^{1} t^2 \frac{1}{\sqrt{2\pi}}\exp(-t^2/2)\,dt\geq \frac{1}{6}$.
	Conditions on sensing matrices Assumption~\ref{assump:sensing operators:vec} also imply $\sigma_i^2\geq \kl\fro{\M-\M^*}^2$. Therefore, $\EE f(\M)-\EE f(\M^*)$ has the following lower bound $$\EE f(\M)-\EE f(\M^*)=\sum_{i=1}^{n}\EE \left[\rho_{H,\delta}(\xi_{i}- \inp{\X_i}{\M-\M^*})\right] - \EE\left[\rho_{H,\delta}(\xi_{i})\right]\geq\frac{2}{3}n\delta b_0^{-1}\kl\fro{\M-\M^*}^2.$$
	Based on event $\bcalE$, we obtain
	\begin{align*}
		f(\M)-f(\M^*) &\geq \EE[f(\M)-f(\M^*)] - C_2\delta\sqrt{nd_1r\ku}\fro{\M-\M^*}.
	\end{align*}
	Thus when $\fro{\M-\M^*}\geq b_0C_2\sqrt{\frac{d_1 r}{n}\cdot\frac{\ku}{\kl^2}}$, we have $$f(\M)-f(\M^*)\geq \frac{1}{3}\delta n b_{0}^{-1}\kl\fro{\M-\M^*}^2.$$ It proves under $\taus=b_0C_2\sqrt{\frac{d_1 r}{n}\cdot\frac{\ku}{\kl^2}} $, it has $\mus= \frac{1}{3}\delta n b_0^{-1}\kl$. The following lemma shows when $\bcalE$ holds, it has $\Ls=C_2\delta\ku\fro{\M-\M^*}$ and the the proof completes.
\end{proof}

\begin{lemma}[Upper bound for sub-gradient]\label{lemma:upperboundsubgradient:huber}
	Let $\M\in\MM_r$ satisfy $\fro{\M-\M^*}\geq \taus:=c\sqrt{\frac{d_1r}{n}\cdot\frac{\ku}{\kl^2}}b_0$. Let $\G\in\partial f(\M)$ be the sub-gradient.  Under the event $\bcalE$, we have $\fror{\G}\leq C_2\delta\fro{\M-\M^*}$ for some absolute constant $C_2>0$.
\end{lemma}
\begin{proof}
	Take $\M_1 = \M + \frac{b_1}{4n\delta\kl}\text{SVD}_r(\G)$. First consider $\EE f(\M_1)-\EE f(\M)$. Continue from Equation~\eqref{eq19},
	\begin{align*}
		&~~~~\EE \left[\rho_{H,\delta}(\xi_i - \inp{\X_i}{\M_1-\M^*})\bigg|\X_i \right] - \EE\left[\rho_{H,\delta}(\xi_i - \inp{\X_i}{\M-\M^*}) \bigg|\X_i\right] \\
		&= 2\int_{\inp{\X_i}{\M-\M^*}}^{\inp{\X_i}{\M_1-\M^*}}[H_{\xi}(y+\delta) - H_{\xi}(y-\delta)]y\,dy\\
		&\leq 4\delta b_1^{-1}\left|\int_{\inp{\X_i}{\M-\M^*}}^{\inp{\X_i}{\M_1-\M^*}}y\,dy\right|\\
		&= 2\delta b_{1}^{-1}\left|\inp{\X_i}{\M_1-\M}^2 + 2\inp{\X_i}{\M_1-\M}\inp{\X_i}{\M-\M^*}\right|.
	\end{align*}
	Then take expectation over $\X_i$ on each side and sum up over $i$:
	\begin{align*}
		\EE f(\M_1)-\EE f(\M)
		&\leq 2\delta nb_{1}^{-1}\ku[\fro{\M_1-\M}^2 + 2\fro{\M_1-\M}\fro{\M-\M^*}],
	\end{align*}
	which uses assumptions of sensing operators. Then invoke event $\bcalE$,  $f(\M_1)-f(\M)$ could be bounded with:
	\begin{align*}
		&{~~~~}f(\M_1)-f(\M)\\&\leq\EE f(\M_1)-\EE f(\M) + C_1\delta\sqrt{nd_1r\ku}\fro{\M_1-\M}\\
		&\leq 2\delta n b_1^{-1}\ku\fro{\M_1-\M}^2 +4\delta n b_1^{-1}\ku\fro{\M_1-\M}\fro{\M-\M^*}+C_1\delta\sqrt{nd_1r\ku}\fro{\M_1-\M}\\
		&\leq2\delta n b_1^{-1}\ku\fro{\M_1-\M}^2 + (C_1+4)\delta n b_1^{-1}\ku\fro{\M_1-\M}\fro{\M-\M^*},
	\end{align*}
	where the last inequality uses $\fro{\M-\M^*}\geq\taus=c\sqrt{\frac{d_1r}{n}\cdot \frac{\ku}{\kl^2}}b_0\geq c\sqrt{\frac{d_1r}{n}\cdot\frac{1}{\ku}}b_1$. Insert $\M_1=\M+\frac{b_1}{4n\delta\ku}\text{SVD}_r(\G)$ and then it is
	\begin{equation}
		\begin{split}
			f(\M+\frac{b_1}{4n\delta\ku}\text{SVD}_r(\G))-f(\M)\leq \frac{b_1}{8n\delta\ku}\fror{\G}^2 + \frac{C_1+4}{4}\fror{\G}\fro{\M-\M^*}.
		\end{split}
		\label{eq22}
	\end{equation}
	On the other hand, from the definition of the sub-gradient, we have 
	\begin{align}\label{eq23}
		f(\M+\frac{b_1}{4n\delta\ku}\text{SVD}_r(\G)) - f(\M)\geq\inp{\G}{\frac{b_1}{4n\delta\ku}\text{SVD}_r(\G) }= \frac{b_1}{4n\delta\ku}\fror{\G}^2.
	\end{align}
	Combine Equation~\eqref{eq22} with Equation~\eqref{eq23} and then the quadratic inequality of $\fror{\G}$ gives $$\fror{\G}\leq C_2\delta n b_1^{-1}\ku\fro{\M-\M^*}.$$
\end{proof}

\subsection{Proof of Theorem~\ref{thm:heavytail-quantile}}
\label{proof:lem:heavytail-quantile}
The two-phase regularity of the quantile loss is provided by the following lemma.

\begin{lemma}\label{lem:heavytail-quantile}
	Suppose Assumption~\ref{assump:heavy-tailed quantile} holds.  There exist absolute constants $C_1, C_2, C_3,c_1>0$ such that if $n\geq C_1rd_1\max\{\delta^2,(1-\delta)^2\}\ku\kl^{-1}$ abd $\BB_1$, $\BB_2$ are given by \begin{align*}
		&\BB_1:=\left\{\M\in\MM_r:\fro{\M-\M^*}\geq 8\sqrt{\kl^{-1}}\gamma\right\},\\
		&\BB_2:=\left\{\M\in\MM_r:C_2\max\{\delta,1-\delta\}b_0\Big(\frac{\ku}{\kl^2}\cdot\frac{rd_1}{n}\Big)^{1/2}\leq\fro{\M-\M^*}<8\sqrt{\kl^{-1}}\gamma\right\},
	\end{align*}
	then with probability over $1-\exp(-c_1rd_1)-3\exp(-\sqrt{n}/\log n)$, the quantile loss $f(\M)=\sum_{i=1}^n \rho_{Q,\delta}(Y_i-\langle \M, \X_i\rangle)$ satisfies Condition~\ref{assump:two-phase}:
	\begin{enumerate}[(1)]
		\item rank-$r$ restricted \textbf{two-phase sharpness} with respect to $\M^*$,$$
		f(\M)-f(\M^{\ast})\geq 
		\begin{cases}
			\frac{n}{8}\sqrt{\kl} \|\M-\M^{\ast}\|_{\rm F}, & \textrm{ for }\ \  \M\in\BB_1; \\
			\frac{n}{24b_0}\kl \|\M-\M^{\ast}\|_{\rm F}^2, & \textrm{ for }\ \  \M\in\BB_2;
		\end{cases}
		$$
		\item rank-$r$ restricted \textbf{two-phase sub-gradient bound} with respect to $\M^*$, $$
		\|\G\|_{\rm F,  r}\leq 
		\begin{cases}
			n\sqrt{\ku}, & \textrm{ for }\ \ \M\in\BB_1;\\
			C_3nb_1^{-1}\ku\|\M-\M^{\ast}\|_{\rm F},& \textrm{ for }\ \ \M\in\BB_2,
		\end{cases}
		$$
		where $\G\in\partial f(\M)$ is any sub-gradient.
	\end{enumerate}
\end{lemma}

By combining Proposition~\ref{prop:main} with Lemma~\ref{lem:heavytail-quantile},  we get the convergence dynamics and statistical accuracy of RsGrad Algorithm~\ref{alg:RsGrad} for quantile loss, namely, Theorem~\ref{thm:heavytail-quantile}.

\begin{proof}[Proof of Lemma~\ref{lem:heavytail-quantile}]
	Notice that quantile loss function $\rho_{Q,\delta}(\cdot)$ is Lipschitz continuous with $\max\{\delta,1-\delta\}$ and satisfies $f(x_1+x_2)\leq f(x_1)+f(x_2)$. 
	
	Suppose the event $\bcalE = \{\sup_{\substack{\Delta\M\in\MM_{2r} \\ \M\in\RR^{d_1\times d_2} }}|f(\M+\Delta\M)-f(\M) - \EE(f(\M+\Delta\M)-f(\M))|\cdot\fro{\Delta\M}^{-1}\leq C_2\max\{\delta,1-\delta\}\sqrt{nd_1r\ku} \}$ holds. Specifically, Theorem \ref{thm:empirical process} proves $$\PP(\bcalE)\geq1-\exp(-cd_1r/4)-3\exp(-\sqrt{n}/\log n). $$
	First consider upper bound of sub-gradient $\G\in f(\M)$, where matrix $\M$ has rank at most $r$. By event $\bcalE$, we have 
	\begin{align*}
		&~~~~f(\M+\text{SVD}_r(\G))-f(\M)\\
		&\leq \EE f(\M+\text{SVD}_r(\G))-\EE f(\M)+C_2\max\{\delta,1-\delta\}\sqrt{nd_1r\ku}\fro{\text{SVD}_r(\G)}\\
		&=\sum_{i=1}^{n}\EE\left[\rho_{Q,\delta}(\xi_i - \inp{\X_i}{\M+\text{SVD}_r(\G)}) -\rho_{Q,\delta}(\xi_i - \inp{\X_i}{\M}) \right]+C_2\max\{\delta,1-\delta\}\sqrt{nd_1r\ku}\fror{\G}\\
		&\leq\sum_{i=1}^{n}\EE\rho_{Q,\delta}(-\inp{\X_i}{\text{SVD}_r(\G)})+C_2\max\{\delta,1-\delta\}\sqrt{nd_1r\ku}\fror{\G}\\
		&\leq \sqrt{\frac{1}{2\pi}}n\sqrt{\ku}\fror{\G}+C_2\max\{\delta,1-\delta\}\sqrt{nd_1r\ku}\fror{\G}\\
		&\leq n\sqrt{\ku}\fror{\G},
	\end{align*}
	where the last line uses $n\geq cd_1r\cdot\max\{\delta^2,(1-\delta)^2\}$. On the other hand, by definition of sub-gradient, we have $$f(\M+\text{SVD}_r(\G))-f(\M)\geq\inp{\G}{\text{SVD}_r(\G)}=\fror{\G}^2.$$
	Then combine the above two equations and it leads to $$\fror{\G}\leq n\sqrt{\ku}.$$
	Then consider the lower bound of $f(\M)-f(\M^*)$.
	
	\noindent\textit{Case 1:} When $\fro{\M-\M^*}\geq8\sqrt{\kl^{-1}}\gamma$. Notice that quantile loss function satisfies triangle inequality, $\rho_{Q,\delta}(x_1+x_2)\leq\rho_{Q,\delta}(x_1)+\rho_{Q,\delta}(x_2)$ with which we have
	\begin{align*}
		&{~~~~}f(\M)-f(\M^*)\\
		&\geq \EE f(\M)-\EE f(\M^*)-C_2\max\{\delta,1-\delta\}\sqrt{nd_1r\ku}\fro{\M-\M^*}\\
		&= \sum_{i=1}^n\EE\left[\rho_{Q,\delta}(\xi_i - \inp{\X_i}{\M-\M^*}) -\rho_{Q,\delta}(\xi_i)\right]-C_2\max\{\delta,1-\delta\}\sqrt{nd_1r\ku}\fro{\M-\M^*}\\
		&\geq \sum_{i=1}^n\EE\rho_{Q,\delta}(-\inp{\X_i}{\M-\M^*})-\sum_{i=1}^{n}\EE\left[\rho_{Q,\delta}(\xi_{i})+\rho_{Q,\delta}(-\xi_{i}) \right]-C_2\max\{\delta,1-\delta\}\sqrt{nd_1r\ku}\fro{\M-\M^*}\\
		&\geq \sqrt{\frac{1}{2\pi}}n\sqrt{\kl}\fro{\M-\M^*}-n\gamma-C_2\max\{\delta,1-\delta\}\sqrt{nd_1r\ku}\fro{\M-\M^*}.
	\end{align*}
	Thus under the condition $\fro{\M-\M^*}\geq 8\sqrt{\kl^{-1}}\gamma$ and $n\geq Cd_1r\ku\kl^{-1}\cdot\max\{\delta^2,(1-\delta)^2\}$, it has $$f(\M)-f(\M^*)\geq \frac{n}{8}\sqrt{\kl}\fro{\M-\M^*},$$
	which verifies with $\taus=8\sqrt{\kl^{-1}}\gamma$, $\muc=\frac{1}{8}n\sqrt{\kl}$.
	
	\noindent\textit{Case 2:} Then consider the second phase where $c\sqrt{\frac{d_1 r}{n}}b_{0}\leq\fro{\M-\M^*}\leq8\sqrt{\kl^{-1}}\gamma$. Notice for all fixed $t_0\in\RR$, 
	\begin{align*}
		\EE\rho_{Q,\delta}(\xi-t_0) &= \delta\int_{s\geq t_0}(s-t_0)\,dH_{\xi}(s) + (1-\delta)\int_{s< t_0}(t_0-s)\,dH_{\xi}(s)\\
		&=\int_{s\geq t_0}(s-t_0)\,dH_{\xi}(s) +(1-\delta) \int_{-\infty}^{+\infty}(t_0-s)\,dH_{\xi}(s)\\
		&=\int_{s\geq t_0}(1-H_{\xi}(s))\,ds + (1-\delta)t_0 - (1-\delta)\int_{-\infty}^{\infty}s \,dH_{\xi}(s),
	\end{align*}
	where the last line is same as the absolute loss case. Set $t_0 = \inp{\X_i}{\M-\M^*}$ and it becomes
	\begin{align*}
		&{~~~~~}\EE\left[\rho_{Q,\delta}(\xi-\inp{\X_i}{\M-\M^*})\bigg| \X_i\right] \\
		&= \int_{s\geq \inp{\X_i}{\M-\M^*}}(1-H_{\xi}(s))\,ds + (1-\delta)\inp{\X_i}{\M-\M^*} - (1-\delta)\int_{-\infty}^{+\infty}s \,dH_{\xi}(s).
	\end{align*}
	With $t_0=0$, it becomes \begin{align*}
		\EE\left[\rho_{Q,\delta}(\xi)\bigg| \X_i\right] = \int_{s\geq 0}(1-H_{\xi}(s))\,ds - (1-\delta)\int_{-\infty}^{+\infty}s \,dH_{\xi}(s).
	\end{align*}The above two equations lead to 
	\begin{align}
		\EE\left[\rho_{Q,\delta}(\xi-\inp{\X_i}{\M-\M^*})\bigg| \X_i\right]- \EE\left[\rho_{Q,\delta}(\xi)\bigg| \X_i\right]= \int_{0}^{\inp{\X_i}{\M-\M^*}}H_{\xi}(s)\,ds -\delta \inp{\X_i}{\M-\M^*}.
		\label{eq24}
	\end{align}
	Note that $\langle \mathbf{X}_{i}, \mathbf{M}-\mathbf{M}^{*}\rangle\sim N(0,\text{vec}(\M-\M^*)^{\top}\bSigma_i\text{vec}(\M-\M^*)) $. Denote $z_i:=\langle \mathbf{X}_{i}, \mathbf{M}-\mathbf{M}^{*}\rangle$ and $\sigma_i^2:= \text{vec}(\M-\M^*)^{\top}\bSigma_i\text{vec}(\M-\M^*)$. Let $f_{z_i}(\cdot)$ be the density of $z_i$. Take expectation of $\X_i$ on both sides of Equation~\eqref{eq24}
	\begin{align*}
		\mathbb{E}\left[\rho_{Q,\delta}(\xi-\inp{\X_i}{\M-\M^*})-\rho_{Q,\delta}(\xi) \right]
		&=\int_{-\infty}^{+\infty}\int_{0}^{t}\left(H_{\xi}(s)-\delta\right)f_{z_i}(t)\,ds\, dt\\
		&=\int_{-\infty}^{+\infty}\int_{0}^{t}\int_{0}^{s}h_{\xi}(w)f_{z_i}(t)\,dw\,ds\,dt
	\end{align*}
	Notice that sensing matrices conditions imply $\sigma_i\leq \sqrt{\ku}\fro{\M-\M^*}\leq 8\sqrt{\ku\kl^{-1}}\gamma$. Then when $w\in(-\sigma_i,\sigma_i)$, under Assumption~\ref{assump:heavy-tailed quantile}, it has $h_{\xi}(w)\geq b_0^{-1}$. Thus it has the following upper bound,
	\begin{align*}
		\mathbb{E}\left[\rho_{Q,\delta}(\xi-\inp{\X_i}{\M-\M^*})-\rho_{Q,\delta}(\xi) \right]&\geq b_{0}^{-1}\int_{-\sigma_{i}}^{\sigma_{i} }\int_{0}^{t}f_{z_i}(t)s\,ds\,dt\\
		&=\frac{1}{2}b_0^{-1}\int_{-\sigma_i}^{\sigma_i}t^2f_{z_i}(t)\,dt =\frac{1}{2} b_0^{-1}\sigma_i^2\int_{-1}^{1}t^2\cdot \frac{1}{\sqrt{2\pi}}e^{-t^2/2}\,dt\\
		&\geq\frac{1}{12b_0}\sigma_{i}^2,
	\end{align*}
	and the last line is from $\int_{-1}^{1}t^2\cdot\frac{1}{\sqrt{2}}e^{-t^2/2}dt\geq1/6$. Therefore, \begin{align*}
		\EE f(\M)-\EE f(\M^*)=\sum_{i=1}^n \mathbb{E}\left[\rho_{Q,\delta}(\xi-\inp{\X_i}{\M-\M^*})-\rho_{Q,\delta}(\xi) \right]\geq\frac{1}{12b_0}\sum_{i=1}^{n}\sigma_i^2\geq\frac{n}{12b_0}\kl\fro{\M-\M^*}^2.
	\end{align*}
	Invoke event $\bcalE$, \begin{align*}
		f(\M)-f(\M^*)&\geq \EE f(\M)-\EE f(\M^*) - C_2\max\{\delta,1-\delta\}\sqrt{nd_1r\ku}\fro{\M-\M^*}\\
		&\geq \frac{n}{12b_0}\kl\fro{\M-\M^*}^2-C_2\max\{\delta,1-\delta\}\sqrt{nd_1r\ku}\fro{\M-\M^*}\\
		&\geq \frac{n}{24b_0}\kl\fro{\M-\M^*}^2,
	\end{align*}
	where the last inequality uses $\fro{\X-\X^*}\geq c\max\{\delta,1-\delta\}\sqrt{\frac{d_1 r}{n}\cdot\frac{\ku}{\kl^2}}b_0$. This proves the lower bound in the second phase and shows $\mus=\frac{n}{24b_0}\kl$.
	
	Finally, from the following lemma, we see that under $\bcalE$, $\Ls\leq C_2nb_1^{-1}$. And this finishes the proof of the lemma.
\end{proof}

\begin{lemma}[Upper bound for sub-gradient]\label{lemma:upperboundsubgradient:quantile}
	Let $\M\in\MM_r$ satisfy $\fro{\M-\M^*}\geq \taus:=c\max\{\delta,1-\delta\}\sqrt{\frac{d_1 r}{n}\cdot\frac{\ku}{\kl^2}}b_1$. Let $\G\in\partial f(\M)$ be the sub-gradient.  Under the event $\bcalE$, we have $\fror{\G}\leq Cnb_1^{-1}\fro{\M-\M^*}$ for some absolute constant $C>0$.
\end{lemma}
\begin{proof}
	Take $\M_1 = \M + \frac{b_1}{2n}\text{SVD}_r(\G)$, where $\text{SVD}_r(\G)$ is the best rank $r$ approximation of $\G$ and then $\rank(\M_1) \leq2r$. First consider $\EE f(\M_1)-\EE f(\M)$.
	Use the same notation as in the above proof and similarly,
	\begin{align*}
		&\EE\left[\rho_{Q,\delta}(\xi-\inp{\X_i}{\M_1-\M^*})-\rho_{Q,\delta}(\xi-\inp{\X_i}{\M-\M^*})\bigg|\X_i\right]\\
		=&\int_{\langle \X_{i}, \M-\M^{*}\rangle}^{\langle \X_{i}, \M-\M^{*}\rangle+\langle \X_{i}, \M_{1}-\M\rangle}\int_{0}^{\xi} h_{\xi}(x)\,dx\,d\xi\\
		\leq& b_{1}^{-1}\left|\int_{\langle \X_{i}, \M-\M^{*}\rangle}^{\langle \X_{i}, \M-\M^{*}\rangle+\langle \X_{i}, \M_{1}-\M\rangle}\xi \,d\xi\right|\\
		=&\frac{1}{2b_1} \left|(\langle \X_{i}, \M-\M^{*}\rangle+\langle \X_{i}, \M_{1}-\M\rangle )^2 - (\langle \X_{i}, \M-\M^{*}\rangle)^2\right|,
	\end{align*}
	where the inequality follows from the upper bound for $h_{\xi}(x)\leq b_{1}^{-1}$. Then take expectation over $\X_i$ on each side and sum up over $i=1,\dots,n$,
	\begin{align*}
		\EE f(\M_{1}) - \EE f(\M) =& \sum_{i=1}^{n}\EE\left[\rho_{Q,\delta}(\xi-\inp{\X_i}{\M_1-\M^*})-\rho_{Q,\delta}(\xi-\inp{\X_i}{\M-\M^*})\right]\\
		\leq&\frac{1}{2b_{1}} \sum_{i=1}^{n}\EE\left|(\langle \X_{i}, \M-\M^{*}\rangle+\langle \X_{i}, \M_{1}-\M\rangle )^2 - (\langle \X_{i}, \M-\M^{*}\rangle)^2\right|\\
		\leq&\frac{n}{2b_{1}} \ku\left| \Vert \M_{1} -\M\Vert_{\mathrm{F}}^{2} + 2\Vert \M_{1} -\M\Vert_{\mathrm{F}} \Vert \M -\M^{*}\Vert_{\mathrm{F}}\right|.
	\end{align*}
	Bound $f(\M_{1})-f(\M)$ under event $\bcalE$:
	\begin{align*}
		f(\M_{1}) - f(\M)\leq&\EE f(\M_{1})-\EE f(\M)+C_1\max\{\delta,1-\delta\}\sqrt{ nd_1r\ku}\fro{\M_1-\M}\\
		\leq&\EE f(\M_{1})-\EE f(\M)+C_1nb_{1}^{-1}\ku\Vert \M-\M^{*}\Vert_{\mathrm{F}}\fro{\M_1-\M}\\
		\leq& \frac{n}{2b_1}\ku \Vert \M_{1}-\M\Vert_{\mathrm{F}}^{2} + (0.5+C_1) nb_{1}^{-1}\ku\Vert \M-\M^{*}\Vert_{\mathrm{F}} \Vert\M_{1}-\M\Vert_{\mathrm{F}},
	\end{align*}
	where the second inequality is from the condition $\fro{\M-\M^*}\geq c\max\{\delta,1-\delta\}\sqrt{\frac{d_1r}{n}\cdot\frac{\ku}{\kl^2}}b_{1}\geq c\max\{\delta,1-\delta\}\sqrt{\frac{d_1r}{n}\cdot\frac{1}{\kl}}b_{1}$. Since $\M_1 = \M + \frac{b_1}{2n\ku}\text{SVD}_r(\G)$, it becomes
	\begin{equation}
		\label{eq28}
		\begin{split}
			&f(\M+\frac{b_1}{2n\ku}\text{SVD}_{r}(\G))-f(\M)\\&{~~~~~~~~~~~~~~~~~~~}\leq \frac{b_1}{8n\ku} \Vert\G\Vert_{\mathrm{F,r}}^{2} + (1/2+C_1) \Vert \M-\M^{*}\Vert_{\mathrm{F}} \Vert\G\Vert_{\mathrm{F,r}}.
		\end{split}
	\end{equation}
	On the other hand, by the definition of sub-gradient, $f(\M_{1})-f(\M)$ has lower bound
	\begin{align}
		f(\M+\frac{b_1}{2n\ku}\text{SVD}_{r}(\G))-f(\M)\geq \frac{b_1^2}{2n\ku}\fror{\G}^2.
		\label{eq27}
	\end{align}
	Combine Equation~\eqref{eq28}, Equation~\eqref{eq27} and then solve the quadratic inequality which leads to $$\fror{\G}\leq Cnb_{1}^{-1}\ku\fro{\M-\M^*}.$$
\end{proof}

\section{Proof of  Initializations}
\label{proof:initialization}
\begin{proof}[Proof of Theorem~\ref{thm:init-general-known}]
	To start with, consider
	\begin{align*}
		&\op{\frac{1}{n}\sum_{i=1}^{n}Y_{i}\text{mat}(\bSigma_i^{-1}\text{vec}(\X_i)) -\M^*}\\
		&{~~~~~~~~~}=\op{\frac{1}{n}\sum_{i=1}^{n}\left(\xi_i+\inp{\X_i}{\M^*}\right)\text{mat}(\bSigma_i^{-1}\text{vec}(\X_i)) -\M^* }\\
		&{~~~~~~~~~}\leq\underbrace{\op{\frac{1}{n}\sum_{i=1}^{n}\inp{\X_i}{\M^*}\text{mat}(\bSigma_i^{-1}\text{vec}(\X_i)) -\M^* }}_{B_1}+\underbrace{\op{\frac{1}{n}\sum_{i=1}^{n}\xi_{i}\text{mat}(\bSigma_i^{-1}\text{vec}(\X_i)}}_{B_2}.
	\end{align*}
	\paragraph*{Analysis of $B_1$}
	First show $\inp{\X_i}{\M^*}\text{mat}(\bSigma_i^{-1}\text{vec}(\X_i))$ is an unbiased estimator of $\M^*$. Or equivalently, $\inp{\X_i}{\M^*}\bSigma_i^{-1}\text{vec}(\X_i)$ is an unbiased estimator of $\text{vec}(\M^*)$,
	\begin{align*}
		\EE\inp{\X_i}{\M^*}\bSigma_i^{-1}\text{vec}(\X_i)&=\EE \text{vec}(\X_i)^{\top}\text{vec}(\M^*)\bSigma_i^{-1}\text{vec}(\X_i)\\
		&=\EE\bSigma_i^{-1}\text{vec}(\X_i)\text{vec}(\X_i)^{\top}\text{vec}(\M^*)\\
		&=\text{vec}(\M^*),
	\end{align*}
	where the second line switches vector and scalar. Then we shall use Bernstein Theorem~\ref{thm:Bernstein Ineq} to bound $B_1$. Note that \begin{align}
		\inp{\X_i}{\M^*}\sim N(0,\text{vec}(\M^*)^{\top}\bSigma_i\text{vec}(\M^*)),\quad \bSigma_i^{-1}\text{vec}(\X_i)\sim N(0,\bSigma_i^{-1}),
		\label{eq:proof-init-1}
	\end{align}
	and it implies
	\begin{align*}
		\left\|\op{\inp{\X_i}{\M^*}\text{mat}(\bSigma_i^{-1}\text{vec}(\X_i)) -\M^* }\right\|_{\Psi_{1}}&\leq\op{\inp{\X_i}{\M^*}}_{\Psi_{2}}\left\|\op{\text{mat}(\bSigma_i^{-1}\text{vec}(\X_i))}\right\|_{\Psi_{2}}\\
		&\leq c\sqrt{\text{vec}(\M^*)^{\top}\bSigma_i\text{vec}(\M^*) }\cdot\op{\bSigma_i^{-\frac{1}{2}}}\\
		&\leq c\sqrt{\ku\kl^{-1}}\fro{\M^*}.
	\end{align*}
	Besides, we have
	\begin{align*}
		&{~~~~}\op{\EE\left(\inp{\X_i}{\M^*}\text{mat}(\bSigma_i^{-1}\text{vec}(\X_i)) -\M^* \right)\left(\inp{\X_i}{\M^*}\text{mat}(\bSigma_i^{-1}\text{vec}(\X_i)) -\M^*\right)^{\top}}\\
		&\leq \EE\op{\left(\inp{\X_i}{\M^*}\text{mat}(\bSigma_i^{-1}\text{vec}(\X_i)) -\M^* \right)\left(\inp{\X_i}{\M^*}\text{mat}(\bSigma_i^{-1}\text{vec}(\X_i)) -\M^*\right)^{\top}}\\
		&\leq \op{\M^*}^2 +2\op{\M^*}\EE\op{\inp{\X_i}{\M^*}\text{mat}(\bSigma_i^{-1}\text{vec}(\X_i))}+\EE\inp{\X_i}{\M^*}^2\op{\text{mat}(\bSigma_i^{-1}\text{vec}(\X_i))}^2\\
		&\leq 2\op{\M^*}^2+2\EE\inp{\X_i}{\M^*}^2\op{\text{mat}(\bSigma_i^{-1}\text{vec}(\X_i))}^2.
	\end{align*}
	Note that $\op{\M^*}=\sigma_{1}$ and by $ab\leq\frac{a^2}{2c}+\frac{cb^2}{2}$, we have \begin{align*}
		&{~~~~~}\EE\inp{\X_i}{\M^*}^2\op{\text{mat}(\bSigma_i^{-1}\text{vec}(\X_i))}^2\\
		&\leq \frac{d_1\kl^{-1}}{2\ku\fro{\M^*}^2}\EE\inp{\X_i}{\M^*}^4+\frac{\ku\fro{\M^*}^2}{2d_1\kl^{-1}}\EE\op{\text{mat}(\bSigma_i^{-1}\text{vec}(\X_i))}^4
	\end{align*}
	Equation~\eqref{eq:proof-init-1} suggests $\EE\inp{\X_i}{\M^*}^4\leq 3\ku^2\fro{\M^*}^4$ and $\EE\op{\text{mat}(\bSigma_i^{-1}\text{vec}(\X_i))}^4\leq c\kl^{-2}d_1^2$. Therefore we have \begin{align*}
		\EE\inp{\X_i}{\M^*}^2\op{\text{mat}(\bSigma_i^{-1}\text{vec}(\X_i))}^2\leq \tilde{c}d_1\ku\kl^{-1} \fro{\M^*}^2.
	\end{align*}
	Thus, we have
	\begin{align*}
		&{~~~~}\op{\EE\left(\inp{\X_i}{\M^*}\text{mat}(\bSigma_i^{-1}\text{vec}(\X_i)) -\M^* \right)\left(\inp{\X_i}{\M^*}\text{mat}(\bSigma_i^{-1}\text{vec}(\X_i)) -\M^*\right)^{\top}}\\
		&\leq 2\fro{\M^*}^2+ 2\tilde{c}d_1\ku\kl^{-1} \fro{\M^*}^2,
	\end{align*}
	and it verifies
	\begin{align*}
		&{~~~~~}\Lambda_{\max}\left( \sum_{i=1}^{n}\left(\inp{\X_i}{\M^*}\text{mat}(\bSigma_i^{-1}\text{vec}(\X_i)) -\M^* \right)\left(\inp{\X_i}{\M^*}\text{mat}(\bSigma_i^{-1}\text{vec}(\X_i)) -\M^*\right)^{\top} \right)/n\\
		&\leq 2\fro{\M^*}^2+ 2\tilde{c}d_1\ku\kl^{-1} \fro{\M^*}^2\leq cd_1\ku\kl^{-1} \fro{\M^*}^2
	\end{align*}
	Similarly, we have
	\begin{align*}
		&{~~~~~}\Lambda_{\max}\left( \sum_{i=1}^{n}\left(\inp{\X_i}{\M^*}\text{mat}(\bSigma_i^{-1}\text{vec}(\X_i)) -\M^* \right)^{\top}\left(\inp{\X_i}{\M^*}\text{mat}(\bSigma_i^{-1}\text{vec}(\X_i)) -\M^*\right) \right)/n\\
		&\leq 2\fro{\M^*}^2+ 2\tilde{c}d_1\ku\kl^{-1} \fro{\M^*}^2\leq cd_1\ku\kl^{-1} \fro{\M^*}^2
	\end{align*}
	Thus, Bernstein Theorem~\ref{thm:Bernstein Ineq} implies the following equation
	\begin{align*}
		B_1&\leq c_1\sqrt{ d_1\ku\kl^{-1} \fro{\M^*}^2}\sqrt{\frac{\log d_1}{n}}+c_1\frac{\log d_1}{n}\log\left(\frac{c\sqrt{\ku\kl^{-1}}\fro{\M^*}}{\sqrt{cd_1\ku\kl^{-1} \fro{\M^*}^2}}\right)\\
		&\leq c_0\sqrt{\ku\kl^{-1}}\sqrt{\frac{d_1 r\log d_1}{n}}\sigma_{1}
	\end{align*}
	holds with probability over $1-c_1d_1^{-1}$.
	\paragraph*{Analysis of $B_2$} Notice that given $\{\xi_{i}\}_{i=1}^n$ with probability exceding $1-c_2\exp(-d_1)$, we have,
	\begin{align*}
		B_2\leq c_3\sqrt{\kl^{-1}}\frac{\sqrt{d_1}}{n}\left(\sum_{i=1}^{n}\xi_{i}^2\right)^{1/2}.
	\end{align*}
	Then by Lemma~\ref{teclem:bound of second moment heavy tailed}, we have with probability exceeding $1-c_3(\log d_1)^{-1}$, the following holds
	\begin{align*}
		\sum_{i=1}^{n}\xi_{i}^2\leq  (2n\gamma_1\log d_1)^{\frac{2}{1+\eps}}.
	\end{align*}
	Thus, we have the upper bound for $B_2$,
	\begin{align*}
		B_2\leq c_4\sqrt{\kl^{-1}}\frac{\sqrt{d_1}}{n^{\frac{\eps}{1+\eps}}}(\log d_1)^{\frac{1}{1+\eps}}\gamma_1^{\frac{1}{1+\eps}}.
	\end{align*}
	Combine $B_1$ and $B_2$ and then we obtain
	\begin{align*}
		\op{\frac{1}{n}\sum_{i=1}^{n}Y_{i}\text{mat}(\bSigma_i^{-1}\text{vec}(\X_i)) -\M^*}&\leq c_0\sqrt{\ku\kl^{-1}}\sqrt{\frac{d_1 r\log d_1}{n}}\sigma_{1}+c_4\sqrt{\kl^{-1}}\frac{\sqrt{d_1}}{n^{\frac{\eps}{1+\eps}}}(\log d_1)^{\frac{1}{1+\eps}}\gamma_1^{\frac{1}{1+\eps}}
	\end{align*}
	holds with probability exceeding $1-c_1d_1^{-1}-c_2\exp(-d_1)-c_3(\log d_1)^{-1}$. Then with perturbation theories Lemma~\ref{teclem:perturbation} and sample size condition $$n\geq C\max\left\{ \ku\kl^{-1}\kappa^2d_1r^2\log d_1, \left(\kl^{-1}d_1r\right)^{\frac{1+\eps}{2\eps}}\cdot\gamma_1^{\frac{1}{\eps}}\sigma_r^{-\frac{1+\eps}{\eps}}(\log d_1)^{\frac{1}{\eps}}\right\},$$ we have
	\begin{align*}
		\op{\M_0-\M^*}&=\op{\text{SVD}_r\left(\frac{1}{n}\sum_{i=1}^{n}Y_{i}\text{mat}(\bSigma_i^{-1}\text{vec}(\X_i))\right) -\M^*}\\
		&\leq \tilde{c}_0\sqrt{\ku\kl^{-1}}\sqrt{\frac{d_1 r\log d_1}{n}}\sigma_{1}+\tilde{c}_4\sqrt{\kl^{-1}}\frac{\sqrt{d_1}}{n^{\frac{\eps}{1+\eps}}}(\log d_1)^{\frac{1}{1+\eps}}\gamma_1^{\frac{1}{1+\eps}},
	\end{align*}
	which implies
	\begin{align*}
		\fro{\M_0-\M^*}\leq\sqrt{2r}\op{\M_0-\M^*}\leq c\sigma_r,
	\end{align*}
	and it completes the proof.
\end{proof}

\section{Technical Lemma}\label{proof:teclem}
The following norm relations are used as facts,
\begin{align*}
	\fro{\A\B}\leq\fro{\A}\op{\B},\quad \left|\inp{\C}{\D}\right|\leq\fro{\C}\fro{\D},
\end{align*}
where $\A,\B,\C,\D$ are matrices with appropriate dimensions.
\begin{lemma}[Partial Frobenius Norm, Lemma 15 of \cite{tong2021accelerating}]
	Define partial Frobenius norm of matrix $\M\in\mathbb{R}^{d_1\times d_2}$ to be $$\fror{\M}:=\sqrt{\sum_{i=1}^{r}\sigma_{i}^{2}(\M)}=\fro{{\rm SVD}_r (\M)}.$$ For any $\M, \hat{\M}\in\mathbb{R}^{d_1\times d_2}$ with $\operatorname{rank}(\hat{\M})\leq r$, one has \begin{align*}
		\vert\langle \M,\hat{\M}\rangle\vert\leq\fror{\M}\fro{\hat\M}.
	\end{align*}
	For any $\M\in\mathbb{R}^{d_1\times d_2}$, $\V\in\mathbb{R}^{d_2\times r}$, one has\begin{align*}
		\fro{\M\V}\leq \fror{\M}\op{\V}.
	\end{align*}
	\label{teclem:partial F norm}
\end{lemma}
Note that in proofs of main theorems, we actually need to bound $\fro{\calP_{\TT_l}(\G_l)}$ rather than the term $\fror{\G_l}$. While, $\fror{\G_l}$ provides a satisfactory upper bound for $\fro{\calP_{\TT_l}(\G_l)}$, see the followed lemma.
\begin{lemma}(Frobenius norm of projected sub-gradient)\label{teclem:projected subgradient norm}
	Let $\M\in\RR^{d_1\times d_2}$ be a matrix with rank at most $r$. Suppose $\mathbf{G}\in \partial f(\M)$ is the sub-gradient at $\M$ and $\mathbb{T}$ is the tangent space for $\MM_r$ at $\M$. Then we have
	$$
	\Vert \mathcal{P}_{\mathbb{T}}(\mathbf{G})\Vert_{\mathrm{F}}\leq\sqrt{2}\fror{\G}.
	$$
\end{lemma}
\begin{proof}
	Let $\M=\U\bSigma\V^\top$ be SVD of $\M$ with orthogonal matrices $\U\in\RR^{d_1\times r}$, $\V\in\RR^{d_2\times r}$. Notice that 
	\begin{align*}
		\fro{\calP_{\TT}(\G)}^2 = \fro{\U\U^T\G}^2 + \fro{(\I - \U\U^T)\G\V\V^T}^2,
	\end{align*}
	and $\text{rank}(\U\U^{\top})=\text{rank}(\V\V^{\top})=r$. Then by properties of partial Frobenius norm Lemma~\ref{teclem:partial F norm}, one has
	\begin{align*}
		\fro{\U\U^T\G} = \fro{\U^T\G} &\leq \fror{\G} \op{\U}=\fror{\G},\\
		\fro{(\I - \U\U^T)\G\V\V^T} &\leq \fror{\G},
	\end{align*}
	which implies $ \fro{\calP_{\TT}(\G)}^2\leq 2\fror{\G}^2$.
\end{proof}
The following lemma is critical in establishing convergence of Riemannian optimizations.
\begin{lemma}(Matrix Perturbation)
	Suppose matrix $\M \in\mathbb{R}^{d_1\times d_2}$ has rank $r$ and has singular value decomposition  $\M=\mathbf{U}\boldsymbol{\Sigma}\mathbf{V}^{\top}$ where $\boldsymbol{\Sigma}=\operatorname{diag}\{\sigma_{1},\sigma_{2},\cdots,\sigma_{r}\}$ and $\sigma_{1}\geq\sigma_{2}\geq\cdots\geq\sigma_{r}> 0$. Then for any $\hat{\M}\in\mathbb{R}^{d\times d}$ satisfying $\Vert\hat{\M}-\M\Vert_{\mathrm{F}}<\sigma_{r}/8$, with $\hat{\mathbf{U}}\in\mathbb{R}^{d_1\times r}$ and $\hat{\mathbf{V}}\in\RR^{d_2\times r}$ the left and right singular vectors of $r$ largest singular values, we have
	\begin{align*}
		\Vert \hat{\mathbf{U}}\hat{\mathbf{U}}^{\top} -\mathbf{U}\mathbf{U}^{\top}\Vert&\leq \frac{8}{\sigma_{r}}\Vert \hat{\mathbf{M}}-\mathbf{M}\Vert,\quad \Vert \hat{\mathbf{V}}\hat{\mathbf{V}}^{\top} -\mathbf{V}\mathbf{V}^{\top}\Vert\leq \frac{8}{\sigma_{r}}\Vert \hat{\mathbf{M}}-\mathbf{M}\Vert,\\
		\Vert\operatorname{SVD}_{r}(\hat{\M})-\M\Vert&\leq\Vert\mathbf{\hat{\M}-\M}\Vert+40\frac{\Vert \hat{\M}-\M\Vert^2}{\sigma_{r}},\\
		\Vert \operatorname{SVD}_{r}(\hat{\M})-\M\Vert_{\mathrm{F}}&\leq\Vert\mathbf{\hat{\M}-\M}\Vert_{\mathrm{F}}+40\frac{\Vert \hat{\M}-\M\Vert\Vert \hat{\M}-\M\Vert_{\mathrm{F}}}{\sigma_{r}}.
	\end{align*}
	\label{teclem:perturbation}
\end{lemma}
\begin{proof}
	See Section~\ref{proof:teclem:perturbation}.
\end{proof}

\begin{lemma}
	Suppose random matrix $\X\in\RR^{d_1\times d_2}$ follows mean zero multivariate Gaussian distribution $N(\boldsymbol{0},\bSigma)$ where $\bSigma\in\RR^{d_1d_2\times d_1d_2}$ satisfies \begin{align*}
		\lambda_{\max}(\bSigma)\leq \ku.
	\end{align*}
	Then its operator norm could be upper bounded by
	\begin{align*}
		\PP\left(\op{\X}\geq t+c\sqrt{\ku(d_1+d_2)}\right)\leq\exp(-t^2),
	\end{align*}
	where $c>0$ is some constant. It implies $\left\| \op{\X}\right\|_{\Psi_{2}}\leq c_1\sqrt{d_1\ku}$ and $\left\| \op{\X}\right\|_{\Psi_{1}}\leq c_2\sqrt{d_1\ku}$, with some constants $c_1,c_2>0$.
	\label{teclem:randmatrxnorm}
\end{lemma}
\begin{proof}
	The proof follows $\eps$-net arguments. Corollary~4.2.13 of work \cite{vershynin2018high} shows that take $\eps=\frac{1}{4}$ and then there exists an $\eps$-net $\calN$ of sphere $\SS^{d_1-1}$ and an $\eps$-net $\calM$ of sphere $\SS^{d_2-1}$ with cardinalities $$|\calN|\leq 9^{d_1},\quad |\calM|\leq 9^{d_2}.$$ For any fixed unit vector $\u\in\calN$, $\v\in\calM$, consider $\u^{\top}\X\v$,
	\begin{align*}
		\u^{\top}\X\v&=\text{vec}(\X)^{\top}(\u\otimes\v)\sim N(0,(\u\otimes\v)^{\top}\bSigma(\u\otimes\v)),
	\end{align*}
	where $\otimes$ is Kronecker product. Besides, it has $\ltwo{\u\otimes\v}=1$ and $(\u\otimes\v)^{\top}\bSigma(\u\otimes\v)\leq\ku$. Thus we have, for any $t>1$,
	\begin{align*}
		\PP\left(\u^{\top}\X\v\geq t\right)\leq \exp\left(-\frac{t^2}{2\ku}\right).
	\end{align*}
	Take the union over $\u\in\calN$ and $\v\in\calM$ and then it leads to
	\begin{align*}
		\PP\left(\max_{\u\in\calN,\, \v\in\calM}\u^{\top}\X\v\geq t\right)\leq  9^{d_1+d_2}\exp\left(-\frac{t^2}{2\ku}\right).
	\end{align*}
	On the other hand, by Theorem~4.4.5 of \cite{vershynin2018high}, we have $\op{\X}\leq 2\max_{\u\in\calN,\, \v\in\calM}\u^{\top}\X\v$, by which we obtain,
	\begin{align*}
		\PP\left(\op{\X}\geq 2t\right)\leq \PP\left(\max_{\u\in\calN,\, \v\in\calM}\u^{\top}\X\v\geq t\right)\leq  9^{d_1+d_2}\exp\left(-\frac{t^2}{2\ku}\right).
	\end{align*}
	It implies $\PP\left(\op{\X}\geq t+c\sqrt{\ku(d_1+d_2)}\right)\leq\exp(-t^2)$ and the proof completes.
\end{proof}

\begin{lemma}(Expectation of absolute loss under Gaussian noise)  Suppose the noise term follows Gaussian distribution $\xi_{1},\cdots,\xi_{n}\stackrel{i.i.d.}{\sim} N(0,\sigma^{2})$ and Gaussian sensing operator vectorization $\text{vec}(\X_i)$ satisfy Assumption~\ref{assump:sensing operators:vec}. Take the absolute value loss function $$f(\M)=\sum_{i=1}^{n}|Y_i-\inp{\M}{\X_i}|,$$
	and then for any matrix $\M\in\RR^{d_1\times d_2}$, we have
	\begin{align*}
		&\mathbb{E}[f(\mathbf{M}^{*})]= n\sqrt{\frac{2}{\pi}}\sigma,\\
		n\sqrt{\frac{2}{\pi}}\sqrt{\sigma^2+\kl\fro{\M-\M^*}^2}\leq&\mathbb{E}[f(\mathbf{M})]\leq n\sqrt{\frac{2}{\pi}}\sqrt{\sigma^{2}+\ku\Vert \mathbf{M}-\mathbf{M}^{*}\Vert_{\mathrm{F}}^{2}}.
	\end{align*}
	\label{teclem:l1expectation}
\end{lemma}
\begin{proof}
	Note that $f(\mathbf{M})=\sum_{i=1}^{n}\vert Y_{i}-\langle \mathbf{X}_{i}, \mathbf{M}\rangle\vert=\sum_{i=1}^{n}\vert \xi_{i} - \langle \mathbf{X}_{i}, \mathbf{M}-\mathbf{M}^{*}\rangle\vert$, $f(\M^{*})=\sum_{i=1}^{n}\vert Y_{i}-\langle \mathbf{X}_{i}, \mathbf{M}^{*}\rangle\vert=\sum_{i=1}^{n}\vert \xi_{i} \vert$.
	
	First calculate the conditional expectation $\mathbb{E}[\vert \xi_{i}-\langle \mathbf{X}_{i}, \mathbf{M}-\mathbf{M}^{*}\rangle\vert \big\lvert \mathbf{X}_{i}]$. Notice that conditioned on $\mathbf{X}_{i}$, $\xi_{i}-\langle \mathbf{X}_{i}, \mathbf{M}-\mathbf{M}^{*}\rangle$ is normally distributed with mean $-\langle \mathbf{X}_{i}, \mathbf{M}-\mathbf{M}^{*}\rangle$ and variance $\sigma^{2}$. Then $\vert \xi_{i}-\langle \mathbf{X}_{i}, \mathbf{M}-\mathbf{M}^{*}\rangle\vert$ has a folded normal distribution.
	\begin{equation*}
		\begin{split}
			\mathbb{E}&\left[\vert \xi_{i}-\langle \mathbf{X}_{i}, \mathbf{M}-\mathbf{M}^{*}\rangle\vert \big\vert \mathbf{X}_{i}\right]\\ &= \sigma\sqrt{\frac{2}{\pi}}\exp(-\frac{\langle \mathbf{X}_{i}, \mathbf{M}-\mathbf{M}^{*}\rangle^{2}}{2\sigma^{2}})-\langle \mathbf{X}_{i}, \mathbf{M}-\mathbf{M}^{*}\rangle[1-2\phi(\frac{\langle \mathbf{X}_{i}, \mathbf{M}-\mathbf{M}^{*}\rangle}{\sigma})],
		\end{split}
	\end{equation*}
	where $\phi(\cdot)$ is the normal cumulative distribution function. For each $i=1,2,\cdots,n$, one has
	\begin{equation*}
		\begin{split}
			\mathbb{E}\vert \xi_{i}-\langle \mathbf{X}_{i}, \mathbf{M}-\mathbf{M}^{*}\rangle\vert &=\mathbb{E}\left[\mathbb{E}[\vert \xi_{i}-\langle \mathbf{X}_{i}, \mathbf{M}-\mathbf{M}^{*}\rangle\vert \big\lvert \mathbf{X}_{i}] \right]\\
			&=\sigma\sqrt{\frac{2}{\pi}}\mathbb{E}\left[\exp\left(-\frac{\langle \mathbf{X}_{i}, \mathbf{M}-\mathbf{M}^{*}\rangle^{2}}{2\sigma^{2}}\right)\right]\\
			&{~~~~~~~~~~~~~}-\mathbb{E}\left[\langle \mathbf{X}_{i}, \mathbf{M}-\mathbf{M}^{*}\rangle\cdot \left[1-2\phi(\frac{\langle \mathbf{X}_{i}, \mathbf{M}-\mathbf{M}^{*}\rangle}{\sigma})\right]\right].
		\end{split}
	\end{equation*}
	Notice that $\langle \mathbf{X}_{i}, \mathbf{M}-\mathbf{M}^{*}\rangle\sim N(0, \text{vec}(\M-\M^*)^{\top}\bSigma_i\text{vec}(\M-\M^*))$ and then after integration calculation it is \begin{equation*}
		\mathbb{E}\vert \xi_{i}-\langle \mathbf{X}_{i}, \mathbf{M}-\mathbf{M}^{*}\rangle\vert = \sqrt{\frac{2}{\pi}}\sqrt{\sigma^{2}+\text{vec}(\M-\M^*)^{\top}\bSigma_i\text{vec}(\M-\M^*)}.
	\end{equation*}
	Specifically, when $\M=\M^*$, it leads to $\mathbb{E}\vert \xi\vert = \sqrt{\frac{2}{\pi}}\sigma$.
	Thus, sum over $i=1,\dots,n$, we have 
	\begin{align*}
		\mathbb{E}f(\mathbf{\mathbf{M}}^{*})&=n\mathbb{E}\vert \xi\vert=n\sqrt{\frac{2}{\pi}}\sigma,\\
		\mathbb{E}f(\mathbf{M})&=\sum_{i=1}^n\mathbb{E}\vert \xi_{i}-\langle \mathbf{X}_{i}, \mathbf{M}-\mathbf{M}^{*}\rangle\vert\\
		&=\sum_{i=1}^n\sqrt{\frac{2}{\pi}}\sqrt{\sigma^{2}+\text{vec}(\M-\M^*)^{\top}\bSigma_i\text{vec}(\M-\M^*)}.
	\end{align*}
	Then by conditions of sensing operators Assumption~\ref{assump:sensing operators:vec}, we obtain
	\begin{align*}
		n\sqrt{\frac{2}{\pi}}\sqrt{\sigma^2+\kl\fro{\M-\M^*}^2}\leq \mathbb{E}f(\mathbf{M})\leq n\sqrt{\frac{2}{\pi}}\sqrt{\sigma^2+\ku\fro{\M-\M^*}^2}.
	\end{align*}
\end{proof}

\begin{corollary}
	Under same conditions of Lemma~\ref{teclem:l1expectation}, we have \begin{align*}
		n\sqrt{\frac{2}{\pi}}\sqrt{\kl}\fro{\M-\M^*}\leq\EE\sum_{i=1}^{n}\left|\inp{\M-\M^*}{\X_i}\right|\leq n\sqrt{\frac{2}{\pi}}\sqrt{\ku}\fro{\M-\M^*}.
	\end{align*}
	\label{cor:l1expecation noiseless}
\end{corollary}
\begin{proof}
	Take noiseless case in Lemma~\ref{teclem:l1expectation} and it yields $\sum_{i=1}^{n}\left|\inp{\M-\M^*}{\X_i}\right|$.
\end{proof}
\begin{proposition}(Empirical process in matrix recovery)\label{thm:empirical process}
	Suppose $f(\mathbf{M})$  is defined in Equation~\eqref{eq:loss:matrix} with sensing matrices in Assumption~\ref{assump:sensing operators:vec} and $\rho(\cdot)$ is Lipschitz continuous with $\tilde{L}$. Then there exists constants $C,C_1>0$, such that, with probability over $1-\exp\left(-\frac{C^{2}d_1r}{3}\right)-3\exp\left(-\frac{\sqrt{n}}{\log n}\right)$,
	\begin{equation}
		\big{|} f(\M+\Delta\M)-f(\M) - (\mathbb{E}f(\M+\Delta\M) - \mathbb{E}f(\M))\big{|}\leq C\tilde{L}\sqrt{nd_1r\ku}\fro{\Delta\M}
	\end{equation} holds for all $\M\in\mathbb{R}^{d_1\times d_2}$ and $\Delta\M\in\MM_r$.
\end{proposition}
\begin{proof}
	See Section~\ref{proof:thm:empirical}.
\end{proof}

\begin{lemma}(Tail inequality for sum of squared heavy-tailed random variables)\label{teclem:bound of second moment heavy tailed}
	Suppose $\xi,\xi_{1},\cdots,\xi_{n}$ are i.i.d. random variables and there exists $0<\varepsilon\leq1$ such that its $1+\varepsilon$ moment is finite, that is $\EE\vert \xi\vert^{1+\varepsilon}<+\infty$. Denote $\op{\xi}_{1+\eps}=\left( \EE\vert \xi\vert^{1+\varepsilon}\right)^{1/(1+\eps)}$. Then for any $s>0$, tail bound for sum of squares could be upper bounded $$\PP\left(\sum_{i=1}^{n}\xi_{i}^2>s\right)\leq 2n\op{\xi}_{1+\eps}^{1+\eps}s^{-\frac{1+\varepsilon}{2}}.$$
\end{lemma}
\begin{proof}
	For any $s>0$, it has $\PP\left(\xi^2\geq s\right)\leq s^{-\frac{1+\varepsilon}{2}} \EE\vert\xi\vert^{1+\varepsilon}$ which uses Markov's inequality. Introduce truncated random variable $\varphi_{i}:=\xi_{i}^21_{\{\xi_{i}^2<s\}}$. Note that $\sum_{i=1}^{n}\xi_{i}^2 $ could be written into sum of two parts $$\sum_{i=1}^{n}\xi_{i}^2=\sum_{i=1}^{n}\varphi_{i}+\sum_{i=1}^{n}(\xi_{i}^2-\varphi_{i}).$$
	Hence, we analyze these two parts respectively. First consider probability of $\sum_{i=1}^{n}\xi_{i}^2\neq\sum_{i=1}^{n}\varphi_{i}$:
	\begin{align*}
		\PP\left(\sum_{i=1}^{n}\xi_{i}^2\neq\sum_{i=1}^{n}\varphi_{i}\right)&\leq n\PP\left(\xi^2>s\right)\leq n\frac{\EE|\xi|^{1+\varepsilon}}{s^{\frac{1+\varepsilon}{2}}}=n\op{\xi}_{1+\eps}^{1+\eps}s^{-\frac{1+\varepsilon}{2}}.
	\end{align*}
	Then bound tail for sum of $\varphi_{i}$:
	\begin{align*}
		\PP\left(\sum_{i=1}^{n}\varphi_{i}>s\right)&\leq n s^{-1}\EE\varphi=ns^{-1}\EE\varphi^{\frac{1+\eps}{2}}\varphi^{\frac{1-\eps}{2}}\leq n\op{\xi}_{1+\eps}^{1+\eps}s^{-\frac{1+\varepsilon}{2}}.
	\end{align*}
	Finally, combine the above two equations, it has
	\begin{align*}
		\PP\left(\sum_{i=1}^{n}\xi_{i}^2>s\right)&\leq \PP\left(\sum_{i=1}^{n}\varphi_{i}>s\right)+\PP\left(\sum_{i=1}^{n}\xi_{i}^2\neq\sum_{i=1}^{n}\varphi_{i}\right)\leq2n\op{\xi}_{1+\eps}^{1+\eps}s^{-\frac{1+\varepsilon}{2}},
	\end{align*}
	which completes the proof.
\end{proof}

\begin{theorem}(Bernstein's  Inequality, Proposition 2 in \cite{koltchinskii2011nuclear})\label{thm:Bernstein Ineq}
	Let $\A,\A_1,\cdots,\A_n$ be independent $p\times q$ matrices that satisfy for some $\alpha\geq1$ (and all $i$) $$\EE \A_i=0,\qquad \Vert \Lambda_{\max}(\A_i)\Vert_{\Psi_\alpha}=:K<+\infty.$$
	Define $$S^2:=\max\left\{\Lambda_{\max}\left(\sum_{i=1}^n\EE\A_{i}\A_{i}^{\top}\right)\big/n, \Lambda_{\max}\left(\sum_{i=1}^n\EE\A_{i}^{\top}\A_{i}\right)\big/n\right\}.$$
	Then for some constant $C>0$ and for all $t>0$,
	\begin{align*}
		\PP\left(\Lambda_{\max}\left(\sum_{i=1}^{n}\A_{i}\right)\big/n\geq CS\sqrt{\frac{t+\log(p+q)}{n}}+CK\log^{1/\alpha}\left(\frac{K}{S}\right)\frac{t+\log(p+q)}{n} \right)\leq\exp(-t).
	\end{align*}
\end{theorem}

\section{Proof of Matrix Perturbation Lemma~\ref{teclem:perturbation}}
\label{proof:teclem:perturbation}
To prove Lemma~\ref{teclem:perturbation}, first consider a simpler case, when the matrix and the perturbation are symmetric. The core of the proof is Theorem~1 of work \cite{xia2021normal} and it is presented here.

Let $\M$ and $\Z$ be symmetric $d\times d$ matrices. The matrix $\M$ has rank $r=\text{rank}(\M)\leq d$ with eigen-decomposition $\M=\U\bLambda\U^{\top}$, where $\U\in\RR^{d\times r}$ contains orthonormal columns and $\bLambda=\text{diag}(\lambda_{1},\dots,\lambda_{r})\in\RR^{r\times r}$ is diagonal. 

Let $\U_{\perp}\in\RR^{d\times (d-r)}$ be complement of $\U$ such that $[\U,\U_{\perp}]\in\RR^{d\times d}$ is orthonormal. Suppose $\hat{\U}\in\RR^{d\times r}$ contains eigenvectors of $\M+\Z$ corresponding to largest $r$ eigenvalues in absolute values.

Define the following projectors
\begin{align}
	\frakP^{0}:=\U_{\perp}\U_{\perp}^{\top},\quad \frakP^{-k}:=\U\bLambda^{-k}\U^{\top},\quad \text{ for all }k\geq1,
	\label{eq:perturb-projector}
\end{align}
and denote the $k$-th order perturbation term
\begin{align}
	S_{\M,k}(\Z):=\sum_{\s:s_1+\dots+s_{k+1}=k}(-1)^{1+\tau(\s)} \cdot \frakP^{-s_1}\Z\frakP^{-s_2}\Z\cdots\Z\frakP^{-s_{k+1}},
	\label{eq:expZ}
\end{align}
where $\s=(s_1,\dots,s_{k+1})$ contains non-negative integer indices and $\tau(\s):=\sum_{j=1}^{k+1}1_{s_j>0}$.

\begin{theorem}[Theorem~1 of \cite{xia2021normal}]
	If $\op{\Z}<\frac{\min_{1\leq i\leq r}|\lambda_i|}{2}$, then it has \begin{align*}
		\hat{\U}\hat{\U}^{\top}-\U\U^\top=\sum_{k\geq1} S_{\M,k}(\Z).
	\end{align*}
	\label{tecthm:pertbation-expansion}
\end{theorem}
\begin{lemma}(Symmetric Matrix Perturbation)
	Suppose symmetric matrix $\mathbf{M}\in\RR^{d\times d}$ has rank $r$. Let $\mathbf{M}=\mathbf{U}\boldsymbol{\Lambda}\mathbf{U}^{\top}$ be its singular value decomposition with $\boldsymbol{\Lambda}=\operatorname{diag}\{\lambda_{1},\lambda_{2},\cdots,\lambda_{r}\}$, $\vert \lambda_{1}\vert\geq\vert \lambda_{2}\vert\geq\cdots\geq\vert\lambda_{r}\vert>0$. Then for any matrix $\hat{\M}\in\RR^{d\times d}$ satisfying $\Vert\hat{\M}-\M\Vert_{\mathrm{F}}<|\lambda_{r}|/8$ with $\hat{\mathbf{U}}\in\mathbb{R}^{d\times r}$ eigenvectors of $r$ largest absolute eigenvalues of $\hat{\mathbf{M}}$, then we have 
	\begin{equation*}
		\Vert \hat{\mathbf{U}}\hat{\mathbf{U}}^{\top}-\mathbf{U}\mathbf{U}^{\top}\Vert\leq\frac{8\Vert\hat{\mathbf{M}}-\mathbf{M}\Vert}{\vert\lambda_{r}\vert},
	\end{equation*}
	\begin{equation*}
		\Vert \operatorname{SVD}_{r}(\hat{\mathbf{M}})-\mathbf{M}\Vert\leq\Vert\hat{\mathbf{M}}-\mathbf{M}\Vert+40\frac{\Vert \hat{\mathbf{M}}-\mathbf{M}\Vert^2}{|\lambda_{r}|},
	\end{equation*}
	\begin{equation*}
		\Vert \operatorname{SVD}_{r}(\hat{\mathbf{M}})-\mathbf{M}\Vert_{\mathrm{F}}\leq\Vert\hat{\mathbf{M}}-\mathbf{M}\Vert_{\mathrm{F}}+40\frac{\Vert \hat{\mathbf{M}}-\mathbf{M}\Vert\Vert\hat{\mathbf{M}}- \mathbf{M}\Vert_{\mathrm{F}}}{|\lambda_{r}|}.
	\end{equation*}
	\label{teclem:symmetric perturbation}
\end{lemma}
\begin{proof}
	Denote $\mathbf{Z}=\hat{\mathbf{M}}-\mathbf{M}$. Define $\mathbf{U}_{\perp}\in\mathbb{R}^{d\times (d-r)}$ such that $[\mathbf{U}, \mathbf{U}_{\perp}]\in\mathbb{R}^{d\times d}$ is orthonormal. The projectors $\frakP^{-k}$, for non-negative integers $k$ and the $k$-th order perturbation term are given by Equation~\eqref{eq:perturb-projector} and Equation~\eqref{eq:expZ} respectively. 
	
	Then norm of $\mathcal{S}_{\mathbf{M},k}(\mathbf{Z})$ could be bounded in the way of
	\begin{equation*}
		\Vert \mathcal{S}_{\mathbf{M},k}(\mathbf{Z})\Vert\leq \binom{2k}{k} \frac{\Vert\mathbf{Z}\Vert^{k}}{\vert\lambda_{r}\vert^{k}}\leq\left(\frac{4\Vert\mathbf{Z}\Vert}{\vert\lambda_{r}\vert}\right)^{k},\quad 	\Vert \mathcal{S}_{\mathbf{M},k}(\mathbf{Z})\Vert_{\mathrm{F}}\leq \binom{2k}{k} \frac{\Vert\mathbf{Z}\Vert_{\mathrm{F}}^{k}}{\vert\lambda_{r}\vert^{k}}\leq\left(\frac{4\Vert\mathbf{Z}\Vert_{\mathrm{F}}}{\vert\lambda_{r}\vert}\right)^{k}.
	\end{equation*}
	Hence, sum up the above equations over $k$ and by Theorem~\ref{tecthm:pertbation-expansion}, it leads to
	\begin{equation}
		\Vert \hat{\mathbf{U}}\hat{\mathbf{U}}^{\top}-\mathbf{U}\mathbf{U}^{\top}\Vert\leq \sum_{k\geq1}\Vert\mathcal{S}_{\mathbf{M},k}(\mathbf{Z})\Vert\leq \frac{8\Vert\mathbf{Z}\Vert}{\vert\lambda_{r}\vert},\quad \Vert \hat{\mathbf{U}}\hat{\mathbf{U}}^{\top}-\mathbf{U}\mathbf{U}^{\top}\Vert_{\mathrm{F}}\leq \sum_{k\geq1}\Vert\mathcal{S}_{\mathbf{M},k}(\mathbf{Z})\Vert_{\mathrm{F}}\leq \frac{8\Vert\mathbf{Z}\Vert_{\mathrm{F}}}{\vert\lambda_{r}\vert}.
		\label{eq:perturb-eigenvector}
	\end{equation}
	Then consider $\mathcal{S}_{\mathbf{M},k}(\mathbf{Z})\mathbf{M}$:\begin{equation*}
		\begin{split}
			\mathcal{S}_{\mathbf{M},k}(\mathbf{Z})\mathbf{M}
			&=\sum_{\mathbf{s}:s_{1}+\cdots+s_{k+1}=k} (-1)^{1+\tau(\mathbf{s})} \mathfrak{P}^{-s_{1}}\mathbf{Z}\mathfrak{P}^{-s_{2}}\cdots\mathfrak{P}^{-s_{k}}\mathbf{Z}\mathfrak{P}^{-s_{k+1}}\mathbf{M}\\
			&=\sum_{\mathbf{s}:s_{1}+\cdots+s_{k+1}=k} (-1)^{1+\tau(\mathbf{s})} \mathfrak{P}^{-s_{1}}\mathbf{Z}\mathfrak{P}^{-s_{2}}\cdots\mathfrak{P}^{-s_{k}}\mathbf{Z}\mathfrak{P}^{-s_{k+1}}\mathbf{U}\boldsymbol{\Lambda}\mathbf{U}^{\top}.
		\end{split}
	\end{equation*}
	Note that when $s_{k+1}=0$, $\mathfrak{P}^{-s_{k+1}}=\mathbf{U}_{\perp}\mathbf{U}_{\perp}^{\top}$, it has $\mathfrak{P}^{-s_{k+1}}\mathbf{M}=0$ and thus we have
	\begin{align*}
		\mathcal{S}_{\mathbf{M},k}(\mathbf{Z})\mathbf{M}&=\sum_{\mathbf{s}:s_{1}+\cdots+s_{k+1}=k,\, s_{k+1}>0} (-1)^{1+\tau(\mathbf{s})} \mathfrak{P}^{-s_{1}}\mathbf{Z}\mathfrak{P}^{-s_{2}}\cdots\mathfrak{P}^{-s_{k}}\mathbf{Z}\mathfrak{P}^{-s_{k+1}}\U\bLambda\U^{\top}\\
		&=\sum_{\mathbf{s}:s_{1}+\cdots+s_{k+1}=k,\, s_{k+1}=1} (-1)^{1+\tau(\mathbf{s})} \mathfrak{P}^{-s_{1}}\mathbf{Z}\mathfrak{P}^{-s_{2}}\cdots\mathfrak{P}^{-s_{k}}\mathbf{Z}\U\U^{\top}\\
		&{~~~}+\sum_{\mathbf{s}:s_{1}+\cdots+s_{k+1}=k,\, s_{k+1}\geq2} (-1)^{1+\tau(\mathbf{s})} \mathfrak{P}^{-s_{1}}\mathbf{Z}\mathfrak{P}^{-s_{2}}\cdots\mathfrak{P}^{-s_{k}}\mathbf{Z}\mathfrak{P}^{1-s_{k+1}}
	\end{align*}
	The operator norm of $\Vert \mathcal{S}_{\mathbf{M},k}(\mathbf{Z})\mathbf{M}\Vert$ could be bounded with
	$$\Vert \mathcal{S}_{\mathbf{M},k}(\mathbf{Z})\mathbf{M}\Vert\leq \binom{2k-1}{k}\frac{\Vert \mathbf{Z}\Vert^{k}}{\vert\lambda_{r}\vert^{k-1}}\leq \Vert\mathbf{Z}\Vert\left(\frac{4\Vert\mathbf{Z}\Vert}{\vert\lambda_{r}\vert}\right)^{k-1},$$
	which implies
	\begin{align*}
		\Vert \sum_{k\geq2} \mathcal{S}_{\mathbf{M},k}(\mathbf{Z})\mathbf{M}\Vert\leq\frac{8\Vert\mathbf{Z}\Vert^{2}}{\vert\lambda_{r}\vert}.
	\end{align*}
	Specifically, when $k=1$, it is $\mathcal{S}_{\mathbf{M},1}(\mathbf{Z})\mathbf{M}= \mathbf{U}_{\perp}\mathbf{U}_{\perp}^{\top}\mathbf{Z}\mathbf{U}\mathbf{U}^{\top}$. Similarly, $\Vert \sum_{k\geq2}\mathbf{M} \mathcal{S}_{\mathbf{M},k}(\mathbf{Z})\Vert\leq\frac{8\Vert\mathbf{Z}\Vert^{2}}{\vert\lambda_{r}\vert} $ holds. In parallel, Frobenius norm of  $\mathcal{S}_{\mathbf{M},k}(\mathbf{Z})\mathbf{M} $ has $$\Vert \mathcal{S}_{\mathbf{M},k}(\mathbf{Z})\mathbf{M}\Vert_{\mathrm{F}}\leq \binom{2k-1}{k}\frac{\Vert \mathbf{Z}\Vert^{k-1}}{\vert\lambda_{r}\vert^{k-1}}\Vert \mathbf{Z}\Vert_{\mathrm{F}}\leq \Vert\mathbf{Z}\Vert_{\mathrm{F}}\left(\frac{4\Vert\mathbf{Z}\Vert}{\vert\lambda_{r}\vert}\right)^{k-1},$$
	\begin{align}
		\Vert \sum_{k\geq2} \mathcal{S}_{\mathbf{M},k}(\mathbf{Z})\mathbf{M}\Vert_{\mathrm{F}}\leq\frac{8\Vert \mathbf{Z}\Vert_{\mathrm{F}}\Vert\mathbf{Z}\Vert}{\vert\lambda_{r}\vert},\quad \Vert \sum_{k\geq2}\mathbf{M} \mathcal{S}_{\mathbf{M},k}(\mathbf{Z})\Vert_{\mathrm{F}}\leq\frac{8\Vert \mathbf{Z}\Vert_{\mathrm{F}}\Vert\mathbf{Z}\Vert}{\vert\lambda_{r}\vert}.
		\label{eq:perturbation-sumk>2}
	\end{align}
	Then expand $\operatorname{SVD}_{r}(\hat{\mathbf{M}})-\mathbf{M} $,
	\begin{equation*}
		\begin{split}
			&{~~~~}\operatorname{SVD}_{r}(\hat{\mathbf{M}})-\mathbf{M}\\
			&=\hat{\mathbf{U}}\hat{\mathbf{U}}^{\top}(\mathbf{M}+\mathbf{Z})\hat{\mathbf{U}}\hat{\mathbf{U}}^{\top}-\mathbf{M}\\
			&=\hat{\mathbf{U}}\hat{\mathbf{U}}^{\top}(\mathbf{M}+\mathbf{Z})\hat{\mathbf{U}}\hat{\mathbf{U}}^{\top}-\mathbf{U}\mathbf{U}^{\top}\mathbf{M}\mathbf{U}\mathbf{U}^{\top}\\
			&=(\hat{\mathbf{U}}\hat{\mathbf{U}}^{\top}-\mathbf{U}\mathbf{U}^{\top})\mathbf{M}\mathbf{U}\mathbf{U}^{\top} + \hat{\mathbf{U}}\hat{\mathbf{U}}^{\top}\mathbf{M}( \hat{\mathbf{U}}\hat{\mathbf{U}}^{\top}-\mathbf{U}\mathbf{U}^{\top})+ \hat{\mathbf{U}}\hat{\mathbf{U}}^{\top}\mathbf{Z}\hat{\mathbf{U}}\hat{\mathbf{U}}^{\top}\\
			&=\sum_{k\geq2} \mathcal{S}_{\mathbf{M},k}(\mathbf{Z})\mathbf{M}\mathbf{U}\mathbf{U}^{\top} + \hat{\mathbf{U}}\hat{\mathbf{U}}^{\top}\mathbf{M}\sum_{k\geq2} \mathcal{S}_{\mathbf{M},k}(\mathbf{Z})+\mathbf{U}_{\perp}\mathbf{U}_{\perp}^{\top}\mathbf{Z}\mathbf{U}\mathbf{U}^{\top} \\
			&{~~~} +\hat{\mathbf{U}}\hat{\mathbf{U}}^{\top} \mathbf{U}\mathbf{U}^{\top}\mathbf{Z}\mathbf{U}_{\perp}\mathbf{U}_{\perp}^{\top}+\hat{\mathbf{U}}\hat{\mathbf{U}}^{\top}\mathbf{Z}\hat{\mathbf{U}}\hat{\mathbf{U}}^{\top},
		\end{split}
	\end{equation*}
	where the last equation follows from Theorem~\ref{tecthm:pertbation-expansion} and $\mathcal{S}_{\mathbf{M},1}(\mathbf{Z})\mathbf{M}= \mathbf{U}_{\perp}\mathbf{U}_{\perp}^{\top}\mathbf{Z}\mathbf{U}\mathbf{U}^{\top}$, $\M\calS_{\M,1}(\Z)=\U\U^{\top}\Z\U_{\perp}\U_{\perp}^{\top}$. Then consider the latter three terms of the above equations,
	\begin{align*}
		&{~~~}\mathbf{U}_{\perp}\mathbf{U}_{\perp}^{\top}\mathbf{Z}\mathbf{U}\mathbf{U}^{\top}+\hat{\mathbf{U}}\hat{\mathbf{U}}^{\top} \mathbf{U}\mathbf{U}^{\top}\mathbf{Z}\mathbf{U}_{\perp}\mathbf{U}_{\perp}^{\top}+\hat{\mathbf{U}}\hat{\mathbf{U}}^{\top}\mathbf{Z}\hat{\mathbf{U}}\hat{\mathbf{U}}^{\top}\\
		&=\U_{\perp}\U_{\perp}^{\top}\Z\U\U^{\top}+\U\U^{\top}\Z\U_{\perp}\U_{\perp}^{\top}+\U\U^{\top}\Z\U\U^{\top} \\
		&{~~~}+ (\hat{\mathbf{U}}\hat{\mathbf{U}}^{\top}-\mathbf{U}\mathbf{U}^{\top}) \mathbf{U}\mathbf{U}^{\top}\mathbf{Z}\mathbf{U}_{\perp}\mathbf{U}_{\perp}^{\top}+(\hat{\mathbf{U}}\hat{\mathbf{U}}^{\top}\mathbf{Z}\hat{\mathbf{U}}\hat{\mathbf{U}}^{\top}-\mathbf{U}\mathbf{U}^{\top}\mathbf{Z}\mathbf{U}\mathbf{U}^{\top})\\
		&=\mathbf{Z}- \mathbf{U}_{\perp}\mathbf{U}_{\perp}^{\top}\mathbf{Z}\mathbf{U}_{\perp}\mathbf{U}_{\perp}^{\top}\\
		&{~~~}+ (\hat{\mathbf{U}}\hat{\mathbf{U}}^{\top}-\mathbf{U}\mathbf{U}^{\top}) \mathbf{U}\mathbf{U}^{\top}\mathbf{Z}\mathbf{U}_{\perp}\mathbf{U}_{\perp}^{\top}+(\hat{\mathbf{U}}\hat{\mathbf{U}}^{\top}\mathbf{Z}\hat{\mathbf{U}}\hat{\mathbf{U}}^{\top}-\mathbf{U}\mathbf{U}^{\top}\mathbf{Z}\mathbf{U}\mathbf{U}^{\top}),
	\end{align*}
	where the last equation is due to $\U\U^{\top}+\U_{\perp}\U_{\perp}^{\top}=\I$ and $\Z=(\U\U^{\top}+\U_{\perp}\U_{\perp}^{\top})\Z(\U\U^{\top}+\U_{\perp}\U_{\perp}^{\top})$. Hence, we get the expansion of $\textrm{SVD}_r(\hat{\M})$,
	\begin{align*}
		&{~~~~}\operatorname{SVD}_{r}(\hat{\mathbf{M}})-\mathbf{M}\\
		&=\sum_{k\geq2} \mathcal{S}_{\mathbf{M},k}(\mathbf{Z})\mathbf{M}\mathbf{U}\mathbf{U}^{\top} + \hat{\mathbf{U}}\hat{\mathbf{U}}^{\top}\mathbf{M}\sum_{k\geq2} \mathcal{S}_{\mathbf{M},k}(\mathbf{Z})+\mathbf{Z}- \mathbf{U}_{\perp}\mathbf{U}_{\perp}^{\top}\mathbf{Z}\mathbf{U}_{\perp}\mathbf{U}_{\perp}^{\top}\\
		&{~~~}+ (\hat{\mathbf{U}}\hat{\mathbf{U}}^{\top}-\mathbf{U}\mathbf{U}^{\top}) \mathbf{U}\mathbf{U}^{\top}\mathbf{Z}\mathbf{U}_{\perp}\mathbf{U}_{\perp}^{\top}+(\hat{\mathbf{U}}\hat{\mathbf{U}}^{\top}\mathbf{Z}\hat{\mathbf{U}}\hat{\mathbf{U}}^{\top}-\mathbf{U}\mathbf{U}^{\top}\mathbf{Z}\mathbf{U}\mathbf{U}^{\top}).
	\end{align*}
	Take Frobenius norm on each side of the above equation and it leads to
	\begin{equation}
		\begin{split}
			&{~~~~}\Vert \operatorname{SVD}_{r}(\hat{\mathbf{M}})-\mathbf{M}\Vert_{\mathrm{F}}\\
			&\leq \underbrace{\Vert \sum_{k\geq2} \mathcal{S}_{\mathbf{M},k}(\mathbf{Z})\mathbf{M}\mathbf{U}\mathbf{U}^{\top}\Vert_{\mathrm{F}}}_{A_1} +\underbrace{\Vert \hat{\mathbf{U}}\hat{\mathbf{U}}^{\top}\mathbf{M}\sum_{k\geq2} \mathcal{S}_{\mathbf{M},k}(\mathbf{Z})\Vert_{\mathrm{F}}}_{A_2}+\underbrace{\Vert\mathbf{Z}- \mathbf{U}_{\perp}\mathbf{U}_{\perp}^{\top}\mathbf{Z}\mathbf{U}_{\perp}\mathbf{U}_{\perp}^{\top}\Vert_{\mathrm{F}}}_{A_3}\\
			&{~~~}+ \underbrace{\Vert(\hat{\mathbf{U}}\hat{\mathbf{U}}^{\top}-\mathbf{U}\mathbf{U}^{\top}) \mathbf{U}\mathbf{U}^{\top}\mathbf{Z}\mathbf{U}_{\perp}\mathbf{U}_{\perp}^{\top}\Vert_{\mathrm{F}}}_{A_4}+\underbrace{\Vert\hat{\mathbf{U}}\hat{\mathbf{U}}^{\top}\mathbf{Z}\hat{\mathbf{U}}\hat{\mathbf{U}}^{\top}-\mathbf{U}\mathbf{U}^{\top}\mathbf{Z}\mathbf{U}\mathbf{U}^{\top}\Vert_{\mathrm{F}}}_{A_5}.
		\end{split}
		\label{eq:perturbation-svd}
	\end{equation}
	Note that by Equation~\eqref{eq:perturbation-sumk>2}, we have upper bound of $A_1$ and $A_2$,
	\begin{align*}
		A_1\vee A_2\leq \Vert \sum_{k\geq2} \mathcal{S}_{\mathbf{M},k}(\mathbf{Z})\mathbf{M}\Vert_{\mathrm{F}} \vee \Vert \mathbf{M}\sum_{k\geq2} \mathcal{S}_{\mathbf{M},k}(\mathbf{Z})\Vert_{\mathrm{F}}\leq\frac{8\Vert \mathbf{Z}\Vert_{\mathrm{F}}\Vert\mathbf{Z}\Vert}{\vert\lambda_{r}\vert}.
	\end{align*}
	Besides, it has $\Z-\U_{\perp}\U_{\perp}^{\top}\Z\U_{\perp}\U_{\perp}^{\top}= \U_{\perp}\U_{\perp}^{\top}\Z\U\U^{\top}+\U\U^{\top}\Z\U_{\perp}\U_{\perp}^{\top}+\U\U^{\top}\Z\U\U^{\top} $ and $\fro{\Z}^2=\fro{(\U\U^{\top}+\U_{\perp}\U_{\perp}^{\top})\Z(\U\U^{\top}+\U_{\perp}\U_{\perp}^{\top})}^2=\fro{\U\U^{\top}\Z\U\U^{\top}}^2+\fro{\U_\perp\U_\perp^{\top}\Z\U\U^{\top}}^2+\fro{\U\U^{\top}\Z\U_\perp\U_\perp^{\top}}^2+\fro{\U_\perp\U_\perp^{\top}\Z\U_\perp\U_\perp^{\top}}^2$ and it leads to \begin{align*}
		A_3^2=\fro{\U\U^{\top}\Z\U\U^{\top}}^2+\fro{\U_\perp\U_\perp^{\top}\Z\U\U^{\top}}^2+\fro{\U\U^{\top}\Z\U_\perp\U_\perp^{\top}}^2\leq\fro{\Z}^2,
	\end{align*}
	which shows $A_3\leq\fro{\Z}$. With Equation~\eqref{eq:perturb-eigenvector}, $A_4$ could be upper bounded by \begin{align*}
		A_4\leq\fro{\hat{\U}\hat{\U}^{\top}-\U\U^\top}\op{\Z}\leq \frac{8\Vert \mathbf{Z}\Vert_{\mathrm{F}}\Vert\mathbf{Z}\Vert}{\vert\lambda_{r}\vert}.
	\end{align*} Note that with triangular inequality and in a similar fashion, it has \begin{align*}
		A_5\leq \fro{\hat{\U}\hat{\U}^{\top}\Z(\hat{\U}\hat{\U}^{\top}-\U\U^{\top})}+\fro{(\hat{\U}\hat{\U}^{\top}-\U\U^{\top})\Z\U\U^{\top}}\leq \frac{16\Vert \mathbf{Z}\Vert_{\mathrm{F}}\Vert\mathbf{Z}\Vert}{\vert\lambda_{r}\vert}.
	\end{align*}
	Thus, Equation~\eqref{eq:perturbation-svd} could be further upper bounded by
	\begin{align*}
		\Vert \operatorname{SVD}_{r}(\hat{\mathbf{M}})-\mathbf{M}\Vert_{\mathrm{F}}\leq \fro{\Z}+40\frac{\Vert \mathbf{Z}\Vert_{\mathrm{F}}\Vert\mathbf{Z}\Vert}{\vert\lambda_{r}\vert}= \fro{\hat{\M}-\M}+40\frac{\fro{\hat{\M}-\M}\op{\hat{\M}-\M}}{|\lambda_{r}|}.
	\end{align*}
	Likewise, one has the bound of the operator norm \begin{equation*}
		\Vert \operatorname{SVD}_{r}(\hat{\mathbf{M}})-\mathbf{M}^{*}\Vert\leq\op{\hat{\M}-\M}+40\frac{\op{\hat{\M}-\M}^2}{|\lambda_{r}|}.
	\end{equation*}
	The proof completes.
\end{proof}
Then we are ready to prove Lemma~\ref{teclem:perturbation}, the assymetric case.
\begin{proof}[\textbf{Proof of Lemma~\ref{teclem:perturbation}}]
	To apply Lemma~\ref{teclem:symmetric perturbation}, first construct symmetrization of $\mathbf{M}$ and $\hat{\mathbf{M}}$, $$\mathbf{Y}:=\left(\begin{matrix}
		0&\mathbf{M}\\
		\mathbf{M}^{\top}&0
	\end{matrix}\right)\in\RR^{(d_1+d_2)\times (d_1+d_2)},\,\hat{\mathbf{Y}}:=\left(\begin{matrix}
		0&\hat{\mathbf{M}}\\
		\hat{\mathbf{M}}^{\top}&0
	\end{matrix}\right)\in\RR^{(d_1+d_2)\times (d_1+d_2)}.$$
	Denote the perturbation matrix as $\mathbf{Z}:=\hat{\mathbf{M}}-\mathbf{M}$ and similarly define
	$$\hat{\mathbf{Z}}:=\left(\begin{matrix}
		0&\mathbf{Z}\\
		\mathbf{Z}^{\top}&0
	\end{matrix}\right)=\left(\begin{matrix}
		0&\hat{\mathbf{M}}-\mathbf{M}\\
		\hat{\mathbf{M}}^{\top}-\mathbf{M}^{\top}&0
	\end{matrix}\right).$$
	Note that $\operatorname{rank}(\mathbf{Y})=2\operatorname{rank}(\M)=2r$. First analyze SVD of matrix $\Y$ and $\hat{\Y}$.  Notice that $$\frac{1}{\sqrt{2}}\left(\begin{matrix}
		\mathbf{U}&\mathbf{U}\\
		\mathbf{V}&-\mathbf{V}
	\end{matrix}\right)\in\mathbb{R}^{2d\times 2r}$$ contains orthonormal columns.
	Besides, it has
	\begin{align*}
		\frac{1}{\sqrt{2}}\left(\begin{matrix}
			\mathbf{U}&\mathbf{U}\\
			\mathbf{V}&-\mathbf{V}
		\end{matrix}\right)\left(\begin{matrix}
			\bSigma&0\\
			0&-\bSigma
		\end{matrix}\right)\frac{1}{\sqrt{2}}\left(\begin{matrix}
			\mathbf{U}&\mathbf{U}\\
			\mathbf{V}&-\mathbf{V}
		\end{matrix}\right)^{\top}=\left(\begin{matrix}
			0&\U\bSigma\V^\top\\
			\V\bSigma\U^{\top}&0
		\end{matrix}\right)&=\left(\begin{matrix}
			0&\M\\
			\M^{\top}&0
		\end{matrix}\right)=\Y,
	\end{align*}
	by which we obtain SVD of $\Y$. Without loss of generality, suppose $d_1\geq d_2$ and $\hat{\M}=\tilde{\U}\tilde{\bSigma}\tilde{\V}^{\top}$ is full SVD of $\hat{\M}$, with $\tilde{\U}\in\RR^{d_1\times d_2}$, $\tilde\bSigma=\text{diag}(\tilde{\sigma}_1,\dots, \tilde{\sigma}_{d_2})\in\RR^{d_2\times d_2}$, $\tilde{\V}\in\RR^{d_2\times d_2}$, where $\tilde{\sigma}_1\geq\cdots\geq\tilde{\sigma}_{d_2}\geq0$. Thus, leading $r$ components of $\hat{\M}$ is $\operatorname{SVD}_r(\hat{\M})=\hat{\U}\hat{\bSigma}\hat{\V}^{\top}$, where $\hat{\U}$ and $\hat{\V}$ are first $r$ columns of full matrix $\tilde{\U}$ and $\tilde{\V}$ respectively and $\hat{\bSigma}=\text{diag}(\tilde{\sigma}_1,\dots,\tilde{\sigma}_r)$.
	
	Notice that since $\hat{\Y}$ is symmetrization for $\hat{\M}$, its full SVD is
	\begin{align}
		\hat{\Y}=\frac{1}{\sqrt{2}}\left(\begin{matrix}
			\tilde{\U}&\tilde{\U}\\
			\tilde{\V}&-\tilde{\V}
		\end{matrix}\right)\left(\begin{matrix}
			\tilde{\bSigma}&0\\
			0&-\tilde{\bSigma}
		\end{matrix}\right)\frac{1}{\sqrt{2}}\left(\begin{matrix}
			\tilde{\U}&\tilde{\U}\\
			\tilde{\V}&-\tilde{\V}
		\end{matrix}\right)^{\top},
		\label{eq:perturbation-svdY}
	\end{align}
	which has eigenvalues $\tilde{\sigma}_1,-\tilde{\sigma}_1,\dots,\tilde{\sigma}_{d_2},-\tilde{\sigma}_{d_2}$. Thus, largest $2r$ eigenvalues of $\hat{\Y}$ in absolute are $\tilde{\sigma}_1,-\tilde{\sigma}_1,\dots,\tilde{\sigma}_{r},-\tilde{\sigma}_{r}$. Thus, since $\hat{\U}$, $\hat{\V}$ are first $r$ columns of $\tilde\U$, $\tilde\V$ respectively, Equation~\eqref{eq:perturbation-svdY} infers
	\begin{align*}
		\frac{1}{\sqrt{2}}\left(\begin{matrix}
			\hat{\mathbf{U}}&\hat{\mathbf{U}}\\
			\hat{\mathbf{V}}&-\hat{\mathbf{V}}
		\end{matrix}\right)^{\top}\hat{\mathbf{Y}}\frac{1}{\sqrt{2}}\left(\begin{matrix}
			\hat{\mathbf{U}}&\hat{\mathbf{U}}\\
			\hat{\mathbf{V}}&-\hat{\mathbf{V}}
		\end{matrix}\right)=\left(\begin{matrix}
			\hat\bSigma&0\\
			0&-\hat\bSigma
		\end{matrix}\right),
	\end{align*}which means $\operatorname{SVD}_{2r}(\hat{\Y})$ is \begin{align*}
		\operatorname{SVD}_{2r}(\hat{\Y})=\frac{1}{\sqrt{2}}\left(\begin{matrix}
			\hat{\mathbf{U}}&\hat{\mathbf{U}}\\
			\hat{\mathbf{V}}&-\hat{\mathbf{V}}
		\end{matrix}\right)\left(\begin{matrix}
			\hat\bSigma&0\\
			0&-\hat\bSigma
		\end{matrix}\right)
		\frac{1}{\sqrt{2}}\left(\begin{matrix}
			\hat{\mathbf{U}}&\hat{\mathbf{U}}\\
			\hat{\mathbf{V}}&-\hat{\mathbf{V}}
		\end{matrix}\right)^{\top}&=\left(\begin{matrix}
			0&\hat\U\hat\bSigma\hat\V^\top\\
			\hat\V\hat\bSigma\hat\U^{\top}&0
		\end{matrix}\right) \\
		&=\left(\begin{matrix}
			0&\operatorname{SVD}_r(\hat{\M})\\
			\operatorname{SVD}_r(\hat{\M})^{\top}&0
		\end{matrix}\right),
	\end{align*}
	and $\frac{1}{\sqrt{2}}\left(\begin{matrix}
		\hat{\mathbf{U}}&\hat{\mathbf{U}}\\
		\hat{\mathbf{V}}&-\hat{\mathbf{V}}
	\end{matrix}\right) $ contains leadning $2r$ eigenvectors of $\hat{\Y}$. Thus, we have \begin{align*}
		\operatorname{SVD}_{2r}(\hat{\Y})-\Y=\left(\begin{matrix}
			0&\operatorname{SVD}_r(\hat{\M})-\M\\
			\operatorname{SVD}_r(\hat{\M})^{\top}-\M^{\top}&0
		\end{matrix}\right).
	\end{align*}
	It infers \begin{align}
		\fro{\operatorname{SVD}_{2r}(\hat{\Y})-\Y }=\sqrt{2}\fro{\operatorname{SVD}_r(\hat{\M})-\M}.
		\label{eq:perturbation-YM}
	\end{align}
	Then apply Lemma~\ref{teclem:symmetric perturbation} to the rank $2r$ symmetric matrix $\Y$ and the symmetric perturbed matrix $\hat{\Y}$, \begin{equation}
		\fro{ \operatorname{SVD}_{2r}(\hat{\mathbf{Y}})-\mathbf{Y}^{*}}\leq\fro{\hat{\Y}-\Y}+40\frac{\fro{\hat{\Y}-\Y}\op{\hat{\Y}-\Y}}{\sigma_{r}},
		\label{eq26}
	\end{equation}
	which uses $|\lambda_{2r}(\Y)|=\sigma_{r}$. Note that definition of $\Y$ and $\hat{\Y}$ implies $\fro{\hat{\Y}-\Y}=\sqrt{2}\fro{\hat{\M}-\M}$ and $\op{\hat{\Y}-\Y}=\op{\hat{\M}-\M}$, respectively. Then, Equation~\eqref{eq:perturbation-YM}, Equation~\eqref{eq26} infer, \begin{align*}
		\fro{\operatorname{SVD}_r(\hat{\M})-\M}\leq \fro{\hat{\M}-\M}+40\frac{\fro{\hat{\M}-\M}\op{\hat{\M}-\M}}{\sigma_{r}}.
	\end{align*}
	Notice that proof of operator norm version $\Vert \operatorname{SVD}_{r}(\hat{\mathbf{M}})-\mathbf{M}\Vert\leq \op{\hat{\M}-\M} + 40\frac{\op{\hat{\M}-\M}^2}{\sigma_{r}}$ is similar to the Frobinius norm case and hence it is skipped.
	
	Apply the perturbated eigenvector results of Lemma~\ref{teclem:symmetric perturbation} to the symmetric rank $2r$ matrix $\Y$ and symmetric perturbed matrix $\hat{\Y}$, \begin{equation*}
		\bigg\Vert \frac{1}{\sqrt{2}}\left(\begin{matrix}
			\hat{\mathbf{U}}&\hat{\mathbf{U}}\\
			\hat{\mathbf{V}}&-\hat{\mathbf{V}}
		\end{matrix}\right)\frac{1}{\sqrt{2}}\left(\begin{matrix}
			\hat{\mathbf{U}}&\hat{\mathbf{U}}\\
			\hat{\mathbf{V}}&-\hat{\mathbf{V}}
		\end{matrix}\right)^{\top} - \frac{1}{\sqrt{2}}\left(\begin{matrix}
			\mathbf{U}&\mathbf{U}\\
			\mathbf{V}&-\mathbf{V}
		\end{matrix}\right)\frac{1}{\sqrt{2}}\left(\begin{matrix}
			\mathbf{U}&\mathbf{U}\\
			\mathbf{V}&-\mathbf{V}
		\end{matrix}\right)^{\top}\bigg\Vert\leq \frac{8\op{\hat{\Y}-\Y}}{\sigma_{r}},
	\end{equation*}
	which could be simplified as \begin{align*}
		\op{\left(\begin{matrix}
				\hat{\U}\hat{\U}^{\top}&0\\
				0&\hat{\V}\hat{\V}^{\top}
			\end{matrix}\right)-\left(\begin{matrix}
				\U\U^{\top}&0\\
				0&\V\V^{\top}
			\end{matrix}\right)}\leq\frac{8\op{\hat{\Y}-\Y}}{\sigma_{r}}.
	\end{align*}
	Then by $\op{\hat{\Y}-\Y}=\op{\hat{\M}-\M}$ and block matrix properties, it arrives at \begin{equation*}
		\Vert \hat{\mathbf{U}}\hat{\mathbf{U}}^{\top} -\mathbf{U}\mathbf{U}^{\top}\Vert\leq\frac{8}{\sigma_{r}}\Vert \hat{\mathbf{M}}-\mathbf{M}\Vert,\quad
		\Vert \hat{\mathbf{V}}\hat{\mathbf{V}}^{\top} -\mathbf{U}\mathbf{U}^{\top}\Vert\leq \frac{8}{\sigma_{r}}\Vert \hat{\mathbf{M}}-\mathbf{M}\Vert.
	\end{equation*}
	
\end{proof}

\section{Proof of Matrix Recovery Empirical Process Proposition~\ref{thm:empirical process}}
\label{proof:thm:empirical}
\begin{proof}
	The proof follows from Theorem~\ref{tecthm:orlicz norm empirical}. Recall $\MM_r:=\{\M\in\RR^{d_1\times d_2}: \rank(\M)\leq r\}$ is the set of matrices with rank bounded by $r$. Consider 
	\begin{align*}
		&~~~~Z:= \sup_{\substack{\Delta\M\in\MM_r \\ \M\in\RR^{d_1\times d_2} }} \left|f(\M+\Delta\M)-f(\M) - \EE(f(\M+\Delta\M)-f(\M))\right|\cdot\fro{\Delta\M}^{-1}.
	\end{align*}
	Note that $f(\M)= \sum_{i=1}^n\rho(\inp{\X_i}{\M}-Y_i)$ and then we have
	\begin{align*}
		\EE Z &\leq 2\EE\sup_{\substack{\Delta\M\in\MM_r \\ \M\in\RR^{d_1\times d_2} }} \left|\sum_{i=1}^n\wt\epsilon_i\cdot\big(\rho(\inp{\X_i}{\M+\Delta\M}-Y_i) - \rho(\inp{\X_i}{\M}-Y_i)\big)\right|\cdot\fro{\Delta\M}^{-1}\notag\\
		&\leq 4\tilde L\cdot\EE \sup_{\M\in\MM_r}\left|\sum_{i=1}^n\wt\epsilon_i\cdot\inp{\X_i}{\Delta\M}\right|\cdot\fro{\Delta\M}^{-1}\notag\\
		&\leq 4\tilde L\cdot\EE\|\sum_{i=1}^n\wt\epsilon_i\X_i\|_{\mathrm{F,r}}
		\leq 4\tilde L\cdot\sqrt{r}\EE\|\sum_{i=1}^n\wt\epsilon_i\X_i\|,
	\end{align*}
	where $\{\wt\epsilon_{i}\}_{i=1}^n$ is Rademacher sequence independent of $\X_1,\ldots,\X_n$. The first inequality is from Theorem \ref{Symmetrization of Expectation} and  the second inequality is from Theorem \ref{Contraction Theorem}. Notice that $\text{vec}\left(\sum_{i=1}^n\wt\epsilon_i\X_i\right)\sim N(\boldsymbol{0}, \sum_{i=1}^{n}\bSigma_i)$ and $\op{\sum_{i=1}^{n}\bSigma_i}\leq n\ku$.
	Lemma~\ref{teclem:randmatrxnorm} implies $\EE\|\sum_{i=1}^n\wt\epsilon_i\X_i\|\leq C_1\sqrt{nd_1\ku}$ for some absolute constant $C_1>0$. Thus we have $$\EE Z\leq C_1\tilde L\sqrt{nd_1r\ku}.$$
	For simplicity, denote $f_i(\M,\Delta\M):=\vert \rho(Y_{i}-\langle \mathbf{X}_{i}, \M+\Delta\M\rangle) - \rho(Y_{i} - \langle \mathbf{X}_{i}, \M\rangle)\vert\cdot\fro{\Delta\M}^{-1}$. Then we have
	\begin{align*}
		&{~~~}\Vert\max_{i=1,\ldots,n}\sup_{\substack{\Delta\M\in\MM_r \\ \M\in\RR^{d_1\times d_2} }} |f_i(\M,\Delta\M)-\EE f_i(\M,\Delta\M)| \Vert_{\Psi_{1}}\\
		&\leq 2\log n	\max_{i=1,\ldots,n}\Vert\sup_{\substack{\Delta\M\in\MM_r \\ \M\in\RR^{d_1\times d_2} }} |f_i(\M+\Delta\M)|\Vert_{\Psi_{1}}+2\log n\max_{i=1,\ldots,n}\EE\sup_{\substack{\Delta\M\in\MM_r \\ \M\in\RR^{d_1\times d_2} }} |f_i(\M+\Delta\M)|,
	\end{align*}
	which uses inequality for maximum of sub-exponential variables. Moreover, for each $i=1,\dots,n$, it has
	\begin{align*}
		&~~~~\Vert\sup_{\substack{\Delta\M\in\MM_r \\ \M\in\RR^{d_1\times d_2} }} |f_i(\M+\Delta\M)|\Vert_{\Psi_{1}}\\
		&=\Vert\sup_{\substack{\Delta\M\in\MM_r \\ \M\in\RR^{d_1\times d_2} }} \vert \rho(Y_{i}-\langle \mathbf{X}_{i}, \M+\Delta\M\rangle) - \rho(Y_{i} - \langle \mathbf{X}_{i}, \mathbf{M}\rangle)\vert\cdot\fro{\Delta\M}^{-1}\Vert_{\Psi_{1}}\\
		&\leq \tilde{L} \Vert\sup_{\Delta\M\in\MM_r} \vert \langle \mathbf{X}_{i},\Delta\M\rangle\vert\cdot\fro{\Delta\M}^{-1}\Vert_{\Psi_{1}}\\
		&\leq \tilde{L}\Vert \|\X_i\|_{\mathrm {F,r}}\Vert_{\Psi_{1}}\leq \tilde{L}\sqrt{r}\Vert\op{\X_i}\Vert_{\Psi_{1}}
		\leq C_2\tilde{L}\sqrt{d_1r\ku},
	\end{align*}
	for some absolute constant $C_2>0$ and the last line is from definition of partial Frobenius norm Lemma~\ref{teclem:partial F norm} and random multivariate Gaussian matrix norm Lemma~\ref{teclem:randmatrxnorm}. Similarly, we have
	\begin{align*}
		\EE\sup_{\substack{\Delta\M\in\MM_r \\ \M\in\RR^{d_1\times d_2} }}|f_i(\M)|&=\EE\sup_{\substack{\Delta\M\in\MM_r \\ \M\in\RR^{d_1\times d_2} }}\vert \rho(Y_{i}-\langle \mathbf{X}_{i}, \M+\Delta\M\rangle) - \rho(Y_{i} - \langle \mathbf{X}_{i}, \mathbf{M}\rangle)\vert\cdot\fro{\Delta\M}^{-1}\\
		&\leq \tilde{L}\cdot\EE\Vert\X_i\Vert_{\mathrm{F,r}}\leq C_3\tilde{L}\sqrt{d_1r\ku}.
	\end{align*} Combine the above three equations and then it has
	\begin{align*}
		\Vert\max_{i=1,\ldots,n}\sup_{\M\in\MM_r} |f_i(\M)-\EE f_i(\M)| \Vert_{\Psi_{1}}\leq C_4\tilde{L}\sqrt{d_1r\ku}\log n.
	\end{align*}
	Also, note that $\inp{\X_i}{\Delta\M}\sim N(0,\text{vec}(\Delta\M)^{\top}\bSigma_i\text{vec}(\Delta\M))$ and $\text{vec}(\M-\M^*)^{\top}\bSigma_i\text{vec}(\M-\M^*)\leq\ku\fro{\M-\M^*}^2$. Then we have
	\begin{align*}
		&{~~~~}\mathbb{E}\big[\big(\rho(Y_{i}-\langle \mathbf{X}_{i}, \M+\Delta\M\rangle) - \rho(Y_{i} - \langle \mathbf{X}_{i}, \mathbf{M}\rangle)\big)^2\big]/\fro{ \Delta\M}^{2}\\
		&\leq \tilde{L}^{2}\mathbb{E}\langle \mathbf{X}_{i}, \Delta\M\rangle^2/\fro{ \Delta\M}^{2}\\
		&\leq \ku\tilde{L}^{2}.
	\end{align*}
	Invoke Theorem \ref{tecthm:orlicz norm empirical} and Remark~\ref{rm:empiricalprocess}, and take $\alpha=1$, $\eta=1$, $\delta=0.5$ and there exists some constant $C>0$, such that 
	\begin{align*}
		\mathbb{P}(Z&\geq C_{2}\tilde{L}\sqrt{\ku nd_1r}+t)\leq \exp \left(-\frac{t^{2}}{3 n\ku\tilde{L}^{2}}\right)+3 \exp \left(-\frac{t}{C_5\tilde{L}\sqrt{\ku d_1r}\log n}\right)
	\end{align*}
	holds for any $t>0$.
	Take $t=C\sqrt{nd_1r\ku}\tilde{L}$ and then we have 
	\begin{equation*}
		\mathbb{P}(Z\geq C\tilde{L}\sqrt{nd_1r\ku})\leq \exp\left(-\frac{C^{2}d_1r}{3}\right)+3\exp\left(-\frac{\sqrt{n}}{\log n}\right),
	\end{equation*}
	which completes the proof.
\end{proof}

\begin{theorem}[Symmetrization of Expectations, \citep{van1996weak}]\label{Symmetrization of Expectation} 
	Consider $\mathbf{X}_{1},\mathbf{X}_{2},\cdots,\mathbf{X}_{n}$ independent matrices in $\chi$ and let $\mathcal{F}$ be a class of real-valued functions on $\chi$. Let $\tilde{\varepsilon}_{1},\cdots,\tilde{\varepsilon}_{n}$ be a Rademacher sequence independent of $\mathbf{X}_{1},\mathbf{X}_{2},\cdots,\mathbf{X}_{n}$, then
	\begin{equation}
		\mathbb{E}\big[\sup_{f\in\mathcal{F}}\big{|} \sum_{i=1}^{n} (f(\mathbf{X}_{i}) - \mathbb{E}f(\mathbf{X}_{i}))\big{|}\big]\leq 2\mathbb{E}\big[\sup_{f\in\mathcal{F}} \big{|} \sum_{i=1}^{n} \tilde{\varepsilon}_{i}f(\mathbf{X}_{i})\big{|}\big]
	\end{equation}
\end{theorem}

\begin{theorem}[Contraction Theorem, \citep{ludoux1991probability}]\label{Contraction Theorem}
	
	Consider the non-random elements $x_{1}, \ldots, x_{n}$ of $\chi$. Let $\mathcal{F}$ be a class of real-valued functions on $\chi$. Consider the Lipschitz continuous functions $\rho_{i}: \mathbb{R} \rightarrow \mathbb{R}$ with Lipschitz constant $L$, i.e.
	$$
	\left|\rho_{i}(\mu)-\rho_{i}(\tilde{\mu})\right| \leq L|\mu-\tilde{\mu}|, \text { for all } \mu, \tilde{\mu} \in \mathbb{R}
	$$
	Let $\tilde{\varepsilon}_{1}, \ldots, \tilde{\varepsilon}_{n}$ be a Rademacher sequence $.$ Then for any function $f^{*}: \chi \rightarrow \mathbb{R}$, we have
	
	\begin{equation}
		\mathbb{E}\left[\sup _{f \in \mathcal{F}}\left|\sum_{i=1}^{n} \tilde{\varepsilon}_{i}\left\{\rho_{i}\left(f\left(x_{i}\right)\right)-\rho_{i}\left(f^{*}\left(x_{i}\right)\right)\right\}\right|\right] \leq 2 \mathbb{E}\left[L\sup_{f \in \mathcal{F}} \mid \sum_{i=1}^{n} \tilde{\varepsilon}_{i}\left(f\left(x_{i}\right)-f^{*}\left(x_{i}\right)\right)\mid \right]
	\end{equation}
\end{theorem}

\begin{theorem}[Tail inequality for suprema of empirical process \citep{adamczak2008tail}]
	Let $\mathbf{X}_{1}, \ldots, \mathbf{X}_{n}$ be independent random variables with values in a measurable space $(\mathcal{S}, \mathcal{B})$ and let $\mathcal{F}$ be a countable class of measurable functions $f: \mathcal{S} \rightarrow \mathbb{R}$. Assume that for every $f \in \mathcal{F}$ and every $i, \mathbb{E} f\left(\mathbf{X}_{i}\right)=0$ and for some $\alpha \in(0,1]$ and all $i,\left\|\sup _{\mathcal{F}}\left|f\left(\mathbf{X}_{i}\right)\right|\right\|_{\Psi_{\alpha}}<\infty .$ Let
	$$
	Z=\sup _{f \in \mathcal{F}}\left|\sum_{i=1}^{n} f\left(\mathbf{X}_{i}\right)\right|
	$$
	Define moreover
	$$
	\sigma^{2}=\sup _{f \in \mathcal{F}} \sum_{i=1}^{n} \mathbb{E} f\left(\mathbf{X}_{i}\right)^{2}
	$$
	Then, for all $0<\eta<1$ and $\delta>0$, there exists a constant $C=C(\alpha, \eta, \delta)$, such that for all $t \geq 0$
	\begin{equation*}
		\begin{split}
			\mathbb{P}(Z&\geq(1+\eta) \mathbb{E}Z+t) \\
			& \leq \exp \left(-\frac{t^{2}}{2(1+\delta) \sigma^{2}}\right)+3 \exp \left(-\left(\frac{t}{C\left\|\max _{i} \sup _{f \in \mathcal{F}}\left|f\left(\mathbf{X}_{i}\right)\right|\right\|_{\Psi_{\alpha}}}\right)^{\alpha}\right)
		\end{split}
	\end{equation*}
	and
	$$
	\begin{aligned}
		\mathbb{P}(Z&\leq(1-\eta) \mathbb{E} Z-t) \\
		& \leq \exp \left(-\frac{t^{2}}{2(1+\delta) \sigma^{2}}\right)+3 \exp \left(-\left(\frac{t}{C\left\|\max _{i} \sup _{f \in \mathcal{F}}\left|f\left(\mathbf{X}_{i}\right)\right|\right\|_{\Psi_{\alpha}}}\right)^{\alpha}\right)
	\end{aligned}
	$$
	\label{tecthm:orlicz norm empirical}
\end{theorem}
\begin{remark}
	Notice that here we require $\mathbb{E}f(\mathbf{X}_{i})=0$ and when $\mathbb{E}f(\mathbf{X}_{i})\neq 0$, $\mathbf{Z}$ should be $$
	Z=\sup _{f \in \mathcal{F}}\left|\sum_{i=1}^{n} f\left(\mathbf{X}_{i}\right)- \mathbb{E}f(\mathbf{X}_{i})\right|,
	$$
	and tail inequality of $Z$ would be \begin{equation*}
		\begin{split}
			\mathbb{P}(Z&\geq(1+\eta) \mathbb{E} Z+t) \\
			& \leq \exp \left(-\frac{t^{2}}{2(1+\delta) \sigma^{2}}\right)+3 \exp \left(-\left(\frac{t}{C\left\|\max _{i} \sup _{f \in \mathcal{F}}\left|f\left(\mathbf{X}_{i}\right)-\EE f(\X_i)\right|\right\|_{\Psi_{\alpha}}}\right)^{\alpha}\right)
		\end{split}
	\end{equation*}
	\label{rm:empiricalprocess}
\end{remark}

\section{Proof of Outlier Case}\label{proof:outlier}
\subsection*{Proof of Theorem~\ref{thm:vec:sparse:outlier}}
We only need to prove the following two-phase regularity properties, with which we could obtain the convergence dynamics similarly to Section~\ref{proof:tec:vec}.
\begin{lemma}
	\label{lem:vec:sparse:outlier}
	Suppose Assumptions~\ref{assump:sensing operators:vec} holds and \ref{assump:heavy-tailed} holds for inliers. There exist absolute constants $C_1,C_2,C_3,c_0>0$ such that if $n\geq C_1\ku\kl^{-1}\tilde{s}\log(2d/\tilde{s})$ and $\epsilon\leq\frac{\sqrt{\kl}}{4\sqrt{\kl}+\sqrt{\ku}}$, then with probability over $1-2\exp\big(-c_0\tilde{s}\log \big(2d/\tilde{s})\big)-3\exp\big(-((1-\epsilon)n\log(2d/\tilde{s}))^{1/2}\log^{-1} (1-\epsilon)n\big)$, we have
	\begin{enumerate}[(1)]
		\item for all $\Bbeta\in\left\{\Bbeta\in\RR^{d}:\; \ltwo{\Bbeta-\Bbeta^*}\geq 8\kl^{-1/2}\gamma,\ |\text{supp}(\Bbeta)|\leq\tilde{s}\right\}$ and for all sub-gradient $\G\in\partial f(\Bbeta)$, 
		\begin{align*}
			f(\Bbeta)-f(\Bbeta^*)\geq \frac{n}{5}\kl^{1/2}\ltwo{\Bbeta-\Bbeta^*}\quad {\rm and}\quad \ltwo{\calP_{\Omega\cup\Pi\cup\Omega^*}(\G)}\leq n \ku^{1/2},
		\end{align*}
		where $\Omega:=\text{supp}(\Bbeta)$ and $\Pi:=\text{supp}\big(\calH_{\tilde{s}}(\calP_{\Omega^{\rm c}}(\G))\big)$;
		
		\item for all 
		\begin{align*}
			\Bbeta\in\left\{\Bbeta\in\RR^{d}:\ C_2\ku^{1/2}\kl^{-1}\cdot b_0\max\left\{\big(\tilde{s}/n\cdot\log(2d/\tilde{s})\big)^{1/2},\epsilon\right\}\leq\ltwo{\Bbeta-\Bbeta^*}\leq8\kl^{-1/2}\gamma,|\text{supp}(\Bbeta)|\leq\tilde{s}\right\}
		\end{align*}
		and for all sub-gradient $\G\in\partial f(\Bbeta)$, 
		\begin{align*}
			f(\Bbeta)-f(\Bbeta^*)\geq \frac{n\kl}{12b_0}\ltwo{\Bbeta-\Bbeta^*}^2\quad {\rm and}\quad \ltwo{\calP_{\Omega\cup\Pi\cup\Omega^*}(\G)}\leq C_3\frac{n\ku}{b_1}\ltwo{\Bbeta-\Bbeta^*},
		\end{align*}
		where $\Omega:=\text{supp}(\Bbeta)$ and $\Pi:=\text{supp}\big(\calH_{\tilde{s}}(\calP_{\Omega^{\rm c}}(\G))\big)$.
	\end{enumerate}
\end{lemma}

\begin{proof}[Proof of Lemma~\ref{lem:vec:sparse:outlier}]
	The difference with Lemma~\ref{lem:vec:sparse} is the appearance of outliers and we shall treat inliers, outliers differently. Denote \begin{align*}
		f_i(\Bbeta):=\sum_{i\in\calI}\left|Y_i-\inp{\X_i}{\Bbeta}\right|,\quad f_o(\Bbeta):=\sum_{i\in\calO}\left|Y_i-\inp{\X_i}{\Bbeta}\right|.
	\end{align*}
	Notice that we have $f(\Bbeta)=f_i(\Bbeta)+f_o(\Bbeta)$. We shall assume the event 
	\begin{align*}
		\bcalE_i&:=\left\{\sup_{\Bbeta\in\RR^d,\; |\text{supp}(\Delta\Bbeta)|\leq 3\tilde{s}}\left|f_i(\Bbeta+\Delta\Bbeta)-f_i(\Bbeta)-\EE\left[f_i(\Bbeta+\Delta\Bbeta)-f_i(\Bbeta) \right]\right|\cdot\ltwo{\Delta\Bbeta}^{-1}\right.\\
		&~~~~~~~~~~~~~~~~\left.\leq C\sqrt{(1-\eps)n\tilde{s}\ku\log(2d/\tilde{s})}\right\}
	\end{align*}
	holds. Proposition~\ref{prop:emp:vec} proves $\PP(\bcalE_i)\geq 1-\exp(-\frac{C^2\tilde{s}\log(2d/\tilde{s})}{3})-3\exp\left(-\frac{\sqrt{(1-\epsilon)n\log(2d/\tilde{s})}}{\log (1-\epsilon)n}\right)$. Also, we assume the event \begin{align*}
		\bcalE_o:=\left\{\sup_{|\text{supp}(\Delta\Bbeta)|\leq 3\tilde{s}}\sum_{i\in\calO}\bigg||\inp{\X_i}{\Delta\Bbeta}|-\EE|\inp{\X_i}{\Delta\Bbeta}|\bigg|\cdot \ltwo{\Delta\Bbeta}^{-1}\leq C\sqrt{\epsilon n\tilde{s}\ku\log(2d/\tilde{s})}\right\}
	\end{align*} holds and epsilon net theory or Gaussian process theory shows $\PP(\bcalE_o)\geq 1-\exp(-\frac{C^2\tilde{s}\log(2d/\tilde{s})}{3})$.
	\paragraph*{Phase One}
	First, show $\mu$-sharpness. Notice that $$f(\Bbeta)-f(\Bbeta^*)=f_i(\Bbeta)-f_i(\Bbeta^*)+f_o(\Bbeta)-f_o(\Bbeta^*).$$
	Same as Section~\ref{proof:tec:vec}, under event $\bcalE_i$, we have $$f_i(\Bbeta)-f_i(\Bbeta^*)\geq (1-\epsilon)n\sqrt{\kl\pi/2}\ltwo{\Bbeta-\Bbeta^*}-C\sqrt{(1-\epsilon) n\tilde{s}\ku\log(2d/\tilde{s})}\ltwo{\Bbeta-\Bbeta^*}-2(1-\epsilon)n\gamma.$$
	At the same time, under event $\bcalE_o$, we have
	\begin{align*}
		\big|f_o(\Bbeta)-f_o(\Bbeta^*)\big|&=\big|\sum_{i\in\calO}\left|Y_i-\inp{\X_i}{\Bbeta}\right|-\sum_{i\in\calO}\left|Y_i-\inp{\X_i}{\Bbeta^*}\right|\big|\\
		&\leq\sum_{i\in\calO}\left|\inp{\X_i}{\Bbeta-\Bbeta^*}\right|\\
		&\leq \EE\sum_{i\in\calO}\left|\inp{\X_i}{\Bbeta-\Bbeta^*}\right|+C\sqrt{\epsilon n\tilde{s}\ku\log(2d/\tilde{s})}\ltwo{\Bbeta-\Bbeta^*}\\
		&\leq \epsilon n\sqrt{\frac{2}{\pi}\ku}\ltwo{\Bbeta-\Bbeta^*}+C\sqrt{\epsilon n\tilde{s}\ku\log(2d/\tilde{s})}\ltwo{\Bbeta-\Bbeta^*},
	\end{align*}
	where the last line follows from Lemma \ref{teclem:l1expectation}. Thus, we have lower bound of $f(\M)-f(\M^*)$,
	\begin{align*}
		&~~~f(\Bbeta)-f(\Bbeta^*)\\
		&\geq f_i(\Bbeta)-f_i(\Bbeta^*)-\left|f_o(\Bbeta)-f_o(\Bbeta^*)\right|\\
		&\geq (1-\epsilon)n\sqrt{\frac{2}{\pi}\kl}\ltwo{\Bbeta-\Bbeta^*}-C\sqrt{(1-\epsilon) n\tilde{s}\ku\log(2d/\tilde{s})}\ltwo{\Bbeta-\Bbeta^*}\\
		&{~~~}- \epsilon n\sqrt{\frac{2}{\pi}\ku}\ltwo{\Bbeta-\Bbeta^*}-C\sqrt{\epsilon n\tilde{s}\ku\log(2d/\tilde{s})}\ltwo{\Bbeta-\Bbeta^*}-2(1-\epsilon)n\gamma\\
		&\geq \frac{n}{5}\sqrt{\kl}\ltwo{\Bbeta-\Bbeta^*},
	\end{align*}
	where the last line is due to $\epsilon\leq\sqrt{\kl}/(\sqrt{\kl}+4\sqrt{\ku})\leq 0.2$, $1-\epsilon\geq0.8$, $n\geq C_1\ku\kl^{-1}\tilde{s}\log(2d/\tilde{s})$ and phase one region constraint $\ltwo{\Bbeta-\Bbeta^*}\geq 8\kl^{-1/2}\gamma$. We could prove upper bound of $\ltwo{\calP_{\Omega\cup\Pi\cup\Omega^*}(\G)}$ in a similar fashion,
	$$\ltwo{\calP_{\Omega\cup\Pi\cup\Omega^*}(\G)}\leq n\sqrt{\ku}.$$
	
	\paragraph*{Phase Two} First consider lower bound of $f(\Bbeta)-f(\Bbeta^*)$. Follow Section~\ref{proof:tec:vec} and we obtain when $C_2\sqrt{\frac{\ku}{\kl^2} }b_0\max\{\sqrt{\frac{\tilde{s}\log(2d/\tilde{s})}{n}},\epsilon \}\leq\ltwo{\Bbeta-\Bbeta^*}\leq 8\sqrt{\kl^{-1}}\gamma$, \begin{align*}
		f_i(\Bbeta)-f_i(\Bbeta^*)\geq (1-\epsilon)\frac{n}{12b_0}\kl\ltwo{\Bbeta-\Bbeta^*}^2.
	\end{align*}
	Besides, same as phase one analyses, under event $\bcalE_o$, it has
	\begin{align*}
		\big|f_o(\Bbeta)-f_o(\Bbeta^*)\big|
		\leq \epsilon n\sqrt{\frac{2}{\pi}\ku}\ltwo{\Bbeta-\Bbeta^*}+C\sqrt{\epsilon n\tilde{s}\ku\log(2d/\tilde{s})}\ltwo{\Bbeta-\Bbeta^*}.
	\end{align*}
	Thus, the above two equations lead to
	\begin{align*}
		f(\Bbeta)-f(\Bbeta^*)&=f_i(\Bbeta)-f_i(\Bbeta^*)+f_o(\Bbeta)-f_o(\Bbeta^*)\\
		&\geq f_i(\Bbeta)-f_i(\Bbeta^*)-\left|f_o(\Bbeta)-f_o(\Bbeta^*)\right|\\
		&\geq (1-\epsilon)\frac{n}{12b_0}\kl\ltwo{\Bbeta-\Bbeta^*}^2-\epsilon n\sqrt{\frac{2}{\pi}\ku}\ltwo{\Bbeta-\Bbeta^*}-C\sqrt{\epsilon n\tilde{s}\ku\log(2d/\tilde{s})}\ltwo{\Bbeta-\Bbeta^*}\\
		&\geq \frac{n}{C_3b_0}\kl\ltwo{\Bbeta-\Bbeta^*}^2,
	\end{align*}
	where the last line uses $\ltwo{\Bbeta-\Bbeta^*}\geq C_2\sqrt{\frac{\ku}{\kl^2} }b_0\max\{\sqrt{\frac{\tilde{s}\log(2d/\tilde{s})}{n}},\epsilon \}$. Similar to Section~\ref{proof:tec:vec}, we have
	\begin{align*}
		\ltwo{\calP_{\Omega\cup\Pi\cup\Omega^*}(\G)}\leq C_4\ku\frac{n}{b_1}\ltwo{\Bbeta-\Bbeta^*}.
	\end{align*}
\end{proof}
\subsection*{Proof of Theorem~\ref{thm:heavytail-l1:outlier}} We only need to prove the following regularity properties.
\begin{lemma}\label{lem:heavytail-l1:outlier}
	Assume $\{\xi_i\}_{i\in\calI}$ and $\{{\rm vec}(\X_i)\}_{i=1}^n$ satisfy Assumptions~\ref{assump:heavy-tailed} and \ref{assump:sensing operators:vec}, respectively.  There exist absolute constants $C_1, C_2, C_3,c_1>0$ such that if $n\geq C_1rd_1\ku\kl^{-1}$, $\epsilon\leq \sqrt{\kl}/(4\sqrt{\kl}+\sqrt{\ku})$ and $\BB_1$, $\BB_2$ are given by
	\begin{align*}
		&\BB_1:=\left\{\M\in\MM_r: \|\M-\M^{\ast}\|_{\rm F}\geq 8\sqrt{\kl^{-1}}\gamma\right\},\\
		&\BB_2:=\left\{\M\in\MM_r: C_2b_0\cdot\sqrt{\frac{\ku}{\kl^2} }\cdot\max\left\{\sqrt{\frac{d_1r}{n}},\epsilon\right\} \leq\|\M-\M^{\ast}\|_{\rm F}< 8\sqrt{\kl^{-1}}\gamma\right\},
	\end{align*}
	then with probability at least $1-2\exp(-c_1rd_1)-3\exp(-\sqrt{n}/\log n)$, the absolute loss $f(\M)=\sum_{i=1}^n|Y_i-\langle \M, \X_i\rangle|$ satisfies Condition~\ref{assump:two-phase}:
	\begin{enumerate}[(1)]
		\item the rank-$r$ restricted \textbf{two-phase sharpness} with respect to $\M^{\ast}$,
		$$
		f(\M)-f(\M^{\ast})\geq 
		\begin{cases}
			\frac{n}{4}\sqrt{\kl} \|\M-\M^{\ast}\|_{\rm F}, & \textrm{ for }\ \  \M\in\BB_1; \\
			\frac{n}{12b_0}\kl \|\M-\M^{\ast}\|_{\rm F}^2, & \textrm{ for }\ \  \M\in\BB_2;
		\end{cases}
		$$
		
		\item the rank-r restricted \textbf{two-phase sub-gradient bound} with respect to $\M^{\ast}$,
		$$
		\|\G\|_{\rm F,  r}\leq 
		\begin{cases}
			2n\sqrt{\ku}, & \textrm{ for }\ \ \M\in\BB_1;\\
			C_3nb_1^{-1}\ku\|\M-\M^{\ast}\|_{\rm F},& \textrm{ for }\ \ \M\in\BB_2,
		\end{cases}
		$$
		where $\G\in\partial f(\M)$ is any sub-gradient.
	\end{enumerate}
\end{lemma}
Proof of Lemma~\ref{lem:heavytail-l1:outlier} would be a combination of absolute loss low-rank study in Lemma~\ref{lem:heavytail-l1} and outlier analyses in Lemma~\ref{lem:vec:sparse:outlier}. We also study loss function of inliers and outliers seperately,
\begin{align*}
	f_i(\M):=\sum_{i\in\calI}\left|Y_i-\inp{\X_i}{\M}\right|,\quad f_o(\M):=\sum_{i\in\calO}\left|Y_i-\inp{\X_i}{\M}\right|.
\end{align*}
Hence, we skip the proof.

\end{document}